\documentclass{amsart}
\usepackage{amsmath}
\usepackage{amssymb}
\usepackage{amsthm}
\usepackage{amscd}
\usepackage{stmaryrd}
\usepackage[dvipdfmx]{graphicx}
\usepackage[all]{xy}
\usepackage{wrapfig}
\usepackage{color}
\usepackage{latexsym}
\usepackage{enumerate}
\usepackage{cases}
\usepackage[top=25mm,bottom=25mm,left=30mm,right=30mm]{geometry}
\usepackage{graphics}
\usepackage[dvipdfmx]{graphicx}
\usepackage{tikz}
\usetikzlibrary{patterns}
\usepackage{pxpgfmark}
\usepackage{multirow}
\usepackage[foot]{amsaddr}
\usetikzlibrary{intersections, calc,decorations.markings}
\usepackage{comment}

\numberwithin{equation}{section}
\theoremstyle{definition}
\newtheorem{example}{Example}[section]
\newtheorem{definition}[example]{Definition}

\newtheorem{remark}[example]{Remark}
\theoremstyle{plain}
\newtheorem{lemma}[example]{Lemma}
\newtheorem{theorem}[example]{Theorem}
\newtheorem{proposition}[example]{Proposition}
\newtheorem{corollary}[example]{Corollary}
\newtheorem{conjecture}[example]{Conjecture}

\DeclareMathOperator{\Int}{\mathsf{Int}}

\DeclareMathOperator{\bZ}{\mathbb{Z}}

\DeclareMathOperator{\cA}{\mathcal{A}}

\DeclareMathOperator{\cF}{\mathcal{F}}

\DeclareMathOperator{\g}{\gamma}
\DeclareMathOperator{\de}{\delta}
\newcommand{\ctext}[1]{\raise0.1ex\hbox{\textcircled{\scriptsize{#1}}}}

\title{$F$-matrices of cluster algebras from triangulated surfaces}
\author{Yasuaki Gyoda \and Toshiya Yurikusa}
\address{Y. Gyoda: Graduate School of Mathematics, Nagoya University, Chikusa-ku, Nagoya, 464-8602 Japan}
\email{m17009g@math.nagoya-u.ac.jp}
\address{T. Yurikusa: Mathematical Institute, Tohoku University, Aoba-ku, Sendai, 980-8578, Japan}
\email{toshiya.yurikusa.d8@tohoku.ac.jp}

\begin{document}


\maketitle
\begin{abstract}
 For a given marked surface $(S,M)$ and a fixed tagged triangulation $T$ of $(S,M)$, we show that each tagged triangulation $T'$ of $(S,M)$ is uniquely determined by the intersection numbers of tagged arcs of $T$ and tagged arcs of $T'$. As consequence, each cluster in the cluster algebra $\cA(T)$ is uniquely determined by its $F$-matrix which is a new numerical invariant of the cluster introduced by Fujiwara and Gyoda.
\end{abstract}

\section{Introduction}

Cluster algebras are commutative subrings of rational function fields. They were introduced in \cite{FZ02} to study total positivity of semisimple Lie groups and canonical bases of quantum groups. Nowadays, it is found that cluster algebras appear in various subjects in mathematics, for example, representation theory of quivers, Poisson geometry, integrable systems, and so on.

 One of important classes of cluster algebras is given from marked surfaces that were developed in \cite{FoG06,FoG09,FoST,FoT,GSV}. For a marked surface $(S,M)$ and the associated cluster algebra, its cluster complex is identified with a connected component of the tagged arc complex of $(S,M)$ \cite{FoST}. In this way, cluster variables correspond to tagged arcs, and clusters correspond to tagged triangulations. Many properties of the cluster algebra can be shown by using this correspondence (see e.g. \cite{FeST,FoST,FoT,L,M,MSW11,MSW13}). Qiu and Zhou \cite{QZ} introduced an intersection number of two tagged arcs to study cluster categories.

 The aim of this paper is to study a new numerical invariant of cluster variables and clusters, called $f$-vectors and $F$-matrices respectively, introduced in \cite{FuG,FK} for the cluster algebra associated with $(S,M)$. To do it, we use intersection numbers of tagged arcs since it was proved by \cite{Y} that intersection vectors coincide with $f$-vectors in the associated cluster algebra. We fix a tagged triangulation $T$ of $(S,M)$. For a tagged arc $\de$ of $(S,M)$, we consider a vector, called its intersection vector, whose entries are intersection numbers of $\de$ and tagged arcs of $T$. We show that a tagged triangulation $T'$ of $(S,M)$ is uniquely determined by the intersection vectors of tagged arcs of $T'$ (Theorem \ref{main}). It induces our main result: in this case, clusters are uniquely determined by their $F$-matrices (Corollary \ref{Funique}).

 This paper is organized as follows. In the rest of this section, we give the results of this paper. In Section \ref{modif}, we prove our results Theorems \ref{main} and \ref{ec} below. For that reason, we introduce modifications of tagged arcs. It plays a key role in our proofs that they are uniquely determined by their intersection vectors (Theorem \ref{intinj}). In Section \ref{seclass}, we study a more detailed result of Theorem \ref{main}. In Section \ref{secfvec}, we recall the notions of $f$-vectors and $F$-matrices. Using the correspondence between $f$-vectors and intersection vectors given in \cite{Y}, we apply the results in the previous sections to study properties of $f$-vectors and $F$-matrices including Corollary \ref{Funique}. In Sections \ref{segments} and \ref{pfintinj}, we are devoted to prove Theorem \ref{intinj} by a lengthy case analysis. In Section \ref{Ex}, we give an example of our results.

\subsection{Our results}\label{MR}

 Let $(S,M)$ be a marked surface. {\it Tagged arcs} of $(S,M)$ are certain curves in $S$ whose endpoints are in $M$ and each end is tagged in one of two ways, {\it plain} or {\it notched} (see Subsection \ref{tagarc}). We represent tagged arcs as follows:
\[
\begin{tikzpicture}
 \coordinate (0) at (0,0) node[left]{plain};
 \coordinate (1) at (1,0);   \fill (1) circle (0.7mm);
 \draw (0) to (1);
\end{tikzpicture}
\hspace{7mm}
\begin{tikzpicture}
 \coordinate (0) at (0,0) node[left]{notched};
 \coordinate (1) at (1,0);   \fill (1) circle (0.7mm);
 \draw (0) to node[pos=0.8]{\rotatebox{90}{\footnotesize $\bowtie$}} (1);
\end{tikzpicture}
\]
 We call a tagged arc $\de$
\begin{itemize}
 \item a {\it plain arc} if its both ends are tagged plain;
 \item a {\it $1$-notched arc} if an end of $\de$ is tagged plain and the other end is tagged notched;
 \item a {\it $2$-notched arc} if its both ends are tagged notched.
\end{itemize}
 We denote by $\overline{\de}$ the plain arc corresponding to a tagged arc $\de$ of $(S,M)$.
 For tagged arcs $\de$ and $\epsilon$ such that $\overline{\de}=\overline{\epsilon}$, if exactly one of them is a $1$-notched arc, then the pair $(\de,\epsilon)$ is called a {\it pair of conjugate arcs} (see Figure \ref{pair}).
\begin{figure}[ht]
\begin{tikzpicture}
 \coordinate (0) at (0,0);
 \coordinate (1) at (0,-1.2);
 \draw (0) to node[left]{$\de$} (1);
 \draw (0) to [out=80,in=100,relative] node[pos=0.2]{\rotatebox{40}{\footnotesize $\bowtie$}} node[right]{$\epsilon$} (1);
 \fill(0) circle (0.7mm); \fill (1) circle (0.7mm);
\end{tikzpicture}
\hspace{20mm}
\begin{tikzpicture}
 \coordinate (0) at (0,0);
 \coordinate (1) at (0,-1.2);
 \draw (0) to node[left]{$\de$} node[pos=0.8]{\rotatebox{0}{\footnotesize $\bowtie$}} (1);
 \draw (0) to [out=80,in=100,relative] node[pos=0.2]{\rotatebox{40}{\footnotesize $\bowtie$}} node[pos=0.75]{\rotatebox{-40}{\footnotesize $\bowtie$}} node[right]{$\epsilon$} (1);
 \fill(0) circle (0.7mm); \fill (1) circle (0.7mm);
\end{tikzpicture}
   \caption{Pairs $(\de,\epsilon)$ of conjugate arcs}
   \label{pair}
\end{figure}
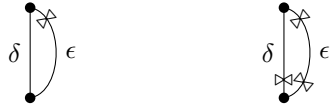

 For tagged arcs $\de$ and $\epsilon$ of $(S,M)$, the {\it intersection number of $\de$ and $\epsilon$} was defined in \cite[Definition 3.3]{QZ} as follows: We assume that $\de$ and $\epsilon$ intersect transversally in a minimum number of points in $S \setminus M$. Then we define the intersection number $\Int(\de,\epsilon)=A+B+C$ \footnote{Note that this definition is slightly different from the ``intersection number'' $(\de | \epsilon)$ defined in \cite[Definition 8.4]{FoST}. The intersection numbers in this paper coincide with entries of $f$-vectors in cluster algebras, and ones in \cite[Definition 8.4]{FoST} coincide with entries of $d$-vectors (see a paragraph right after Theorem \ref{class} and \cite{FoST}). They are the same if $(S,M)$ has no puncture.},where
\begin{itemize}
 \item $A$ is the number of intersection points of $\de$ and $\epsilon$ in $S \setminus M$;
 \item $B$ is the number of pairs of an end of $\de$ and an end of $\epsilon$ that are incident to a common puncture such that their tags are different;
 \item $C=0$ unless $\de$ and $\epsilon$ form a pair of conjugate arcs, in which case $C=-1$.
\end{itemize}
 Tagged arcs $\de$ and $\epsilon$ are called {\it compatible} if $\Int(\de,\epsilon)=0$. A {\it tagged triangulation} is a maximal set of pairwise compatible tagged arcs.

 The number of tagged arcs in a tagged triangulation of $(S,M)$ is constant \cite[Theorem 7.9]{FoST}. Fix a tagged triangulation $T$ of $(S,M)$ with $n$ tagged arcs. For a tagged arc $\de$ of $(S,M)$, we define
\[
 \Int(T,\de):=(\Int(t,\de))_{t \in T} \in \bZ_{\ge 0}^{n},
\]
called an {\it intersection vector of $\de$ with respect to $T$}. For a tagged triangulation $T'=\{\de_1,\ldots,\de_n\}$ of $(S,M)$, we denote by $\Int(T,T')$ the non-negative integer matrix with columns $\Int(T,\de_1),\ldots,\Int(T,\de_n)$. We are ready to state the key result of this paper.

\begin{theorem}\label{main}
 Let $T$ be a tagged triangulation of $(S,M)$. If tagged triangulations $T'$ and $T''$ of $(S,M)$ have $\Int(T,T')=\Int(T,T'')$ up to permutations of columns, then $T'=T''$.
\end{theorem}
We consider whether a tagged arc $\delta\notin T$ is determined by its intersection vector with $T$. Note that when $\de\in T$, then clearly we have $\Int(T,\de)=0$, so that arcs in $T$ are clearly not determined by their intersection number with $T$. Thus we study the following property.

\begin{definition}
 For a tagged triangulation $T$ of $(S,M)$, we say that $T$ {\it detects tagged arcs} if it satisfies the following condition:\par
 $\bullet$ If tagged arcs $\de$ and $\epsilon$ of $(S,M)$ have a common non-zero intersection vector $\Int(T,\de)=\Int(T,\epsilon)$, then $\de=\epsilon$.
\end{definition}

 We give a characterization of this property. In particular, a tagged triangulation does not detect tagged arcs generally.

\begin{theorem}\label{ec}
 Let $T$ be a tagged triangulation of $(S,M)$. Then $T$ detects tagged arcs if and only if there are no tagged arcs $\de$ and $\epsilon$ of $T$ connecting two (possibly same) common punctures such that $\overline{\de} \neq \overline{\epsilon}$.
\end{theorem}

 Next, we give a complete list of marked surfaces which have tagged triangulations detecting tagged arcs.

\begin{theorem}\label{class}
 $(1)$ If $S$ is not closed, then there is at least one tagged triangulation of $(S,M)$ detecting tagged arcs.

 $(2)$ If $S$ is closed, then there is at least one tagged triangulation of $(S,M)$ detecting tagged arcs if and only if the inequality
\begin{equation}\label{ineq}
 p\ge\left\{
     \begin{array}{ll}
 10 &\text{if $g=2$},\\
 \smallskip \\
 \cfrac{7+\sqrt{1+48g}}{2} &\text{if $g\neq2$},
     \end{array}\right.
\end{equation}
holds\, where $p$ is the number of punctures of $(S,M)$ and $g$ is the genus of $S$.\footnote{The lower part of the right hand side of \eqref{ineq} is known as the \emph{Heawood number}. This number appears in the version of the four-color theorem for higher genus surface \cite{RY}.}

 $(3)$ All tagged triangulation of $(S,M)$ detect tagged arcs if and only if $(S,M)$ is one of the followings:
\begin{itemize}
 \item a marked surface with no punctures;
 \item a marked surface of genus $0$ with exactly $1$ boundary component and at most $2$ punctures;
 \item a marked surface of genus $0$ with exactly $2$ boundary components and a $1$ puncture.
\end{itemize}
\end{theorem}

 Finally we apply our results to a cluster algebra $\cA(T)$ associated with a tagged triangulation $T$ (see Subsection \ref{clalg}). Then each tagged arc $\de$ of $(S,M)$ gives rise to the cluster variables $z_{\de}$ in $\cA(T)$. It was shown in \cite{Y} that the intersection vector $\Int(T,\de)$ is equal to the $f$-vector of $z_{\de}$, that is, the maximal degree of $F$-polynomial of $z_{\de}$. As consequence, we get the main result of this paper.

\begin{corollary}[Corollary \ref{Funique}]
 Let $T$ be a tagged triangulation of $(S,M)$. For clusters $\mathbf{z}$ and $\mathbf{z'}$ of $\cA(T)$, if the $f$-vectors of cluster variables in $\mathbf{z}$ coincide with ones in $\mathbf{z'}$, then $\mathbf{z}=\mathbf{z'}$.
\end{corollary}

\begin{remark}
 In cluster algebras, there are four families of integer vectors which are $f$-vectors, $d$-vectors, $g$-vectors and $c$-vectors (see e.g. \cite{FZ02,FZ04,FZ07}). In cluster algebras defined from marked surfaces, they are given by $\Int(\cdot,\cdot)$, $(\cdot | \cdot)$ and shear coordinates \cite{Rea14, FoT,FeT}.
\[
\renewcommand{\arraystretch}{1.7}{
\begin{tabular}{c||c|c|c}
 Cluster algebras & $f$-vectors & $d$-vectors & $g$-vectors, $c$-vectors\\\hline
 Marked surfaces & $\Int(\cdot,\cdot)$ & $(\cdot | \cdot)$ & shear coordinates
  \end{tabular}}
\]
Analogues of Corollary \ref{Funique} for $d$-vectors, $g$-vectors and $c$-vectors have proved already in the case of general cluster algebras. About $d$-vectors, see \cite[Theorem 4.22 (i)]{cl16}, and about $c$-vectors and $g$-vectors, see \cite[Corollary 4.5, Theorem 4.8]{Nak}. In particular, also according a footnote of the definition of intersection number, when $(S,M)$ has no punctures, $f$-vectors coincide with $d$-vectors, and thus Corollary \ref{Funique} follows from \cite[Theorem 4.22 (i)]{cl16} in this case. 
\end{remark}

\medskip\noindent{\bf Acknowledgements}.
 The authors are grateful to Tomoki Nakanishi and Osamu Iyama for helpful advice. We also thanks Futaba Fujie and Masakazu Tsuda for helpful comments. Authors are Research Fellows of Society for the Promotion of Science (JSPS). This work was supported by JSPS KAKENHI Grant Number JP17J04270, JP20J12675.

\section{Modifications of tagged arcs}\label{modif}
\subsection{Tagged arcs}\label{tagarc}

 Let $S$ be a connected compact oriented Riemann surface with (possibly empty) boundary and $M$ a non-empty finite set of marked points on $S$ with at least one marked point on each boundary component. We call the pair $(S,M)$ a {\it marked surface}. Any marked point in the interior of $S$ is called a {\it puncture}. For technical reasons, throughout this paper we assume $(S,M)$ is not a monogon with at most one puncture, a digon without punctures, a triangle without punctures, and a sphere with at most three punctures (cf. \cite{FoST}).

\begin{definition}
 A {\it tagged arc} is a curve in $S$, considered up to isotopy, whose endpoints are in $M$ and each end is tagged in one of two ways, {\it plain} or {\it notched}, such that the following conditions are satisfied:
\begin{itemize}
 \item it does not intersect itself except at its endpoints;
 \item it is disjoint from $M$ and from the boundary of $S$ except at its endpoints;
 \item it does not cut out a monogon with at most one puncture or a digon without punctures;
 \item its endpoint lying on the boundary of $S$ is tagged plain;
 \item both ends of a loop are tagged in the same way,
\end{itemize}
 where a {\it loop} is a tagged arc with two identical endpoints.
\end{definition}

 For a tagged arc $\de$ and a puncture $p$ of $(S,M)$, we define that $\de^{(p)}$ is the tagged arc obtained from $\de$ by changing its tags at $p$. If $\de$ is not incident to $p$, then $\de^{(p)}=\de$. By definition, we have $\Int(\de^{(p)},\epsilon^{(p)})=\Int(\de,\epsilon)$ for any tagged arcs $\de, \epsilon$ and puncture $p$ of $(S,M)$. Therefore, to consider intersection vectors with respect to a tagged triangulation $T$ of $(S,M)$, by changing tags, we can assume that $T$ satisfies the following condition:
\begin{itemize}
 \item[$(\Diamond)$] The tagged triangulation $T$ consists of plain arcs and $1$-notched arcs, with at most one $1$-notched arc incident to each puncture.
\end{itemize}

\subsection{Puzzle pieces}\label{puzzle}

 A key of many proofs in this paper is a puzzle piece decomposition of tagged triangulations studied in \cite{FoST}. We denote by $T_3$ a tagged triangulation satisfying $(\Diamond)$ of a $4$-punctured sphere consisting of three pairs of conjugate arcs (see the right diagram of Figure \ref{pp}). Any tagged triangulation satisfying $(\Diamond)$ which is not $T_3$ is obtained by gluing together a number of puzzle pieces in Figure \ref{pp} (see \cite[Remark 4.2]{FoST}). We say that a puzzle piece in the first (resp., second, third) diagram from the left on Figure \ref{pp} is a {\it triangle piece} (resp., a {\it $1$-puncture piece}, a {\it $2$-puncture piece}).
\begin{figure}[ht]
\begin{tikzpicture}[baseline=5mm]
 \coordinate (0) at (0,0);
 \coordinate (1) at (120:1.5);
 \coordinate (2) at (180:1.5);
 \draw (0)--(1)--(2)--(0);
 \fill(0) circle (0.7mm); \fill (1) circle (0.7mm); \fill (2) circle (0.7mm);
\end{tikzpicture}
   \hspace{10mm}
\begin{tikzpicture}[baseline=-10mm]
 \coordinate (0) at (0,0);
 \coordinate (1) at (0,-1);
 \coordinate (2) at (0,-2);
 \draw (2) to [out=180,in=180] (0);
 \draw (2) to [out=0,in=0] (0);
 \draw (1) to (2);
 \draw (1) to [out=60,in=120,relative] node[pos=0.2]{\rotatebox{40}{\footnotesize $\bowtie$}} (2);
 \fill(0) circle (0.7mm); \fill (1) circle (0.7mm); \fill (2) circle (0.7mm);
\end{tikzpicture}
   \hspace{10mm}
\begin{tikzpicture}[baseline=-10mm]
 \coordinate (0) at (0,0);
 \coordinate (1) at (-0.5,-0.8); \fill (1) circle (0.7mm);
 \coordinate (1') at (0.5,-0.8); \fill (1') circle (0.7mm);
 \coordinate (2) at (0,-2); \fill (2) circle (0.7mm);
 \draw (0,-1) circle (1);
 \draw (1) to (2);
 \draw (1) to [out=-60,in=-120,relative] node[pos=0.2]{\rotatebox{170}{\footnotesize $\bowtie$}} (2);
 \draw (1') to (2);
 \draw (1') to [out=60,in=120,relative] node[pos=0.2]{\rotatebox{180}{\footnotesize $\bowtie$}} (2);
\end{tikzpicture}
   \hspace{10mm}
$T_3=$
\begin{tikzpicture}[baseline=2mm]
 \coordinate (0) at (0,0);
 \coordinate (u) at (90:1);
 \coordinate (r) at (-30:1);
 \coordinate (l) at (210:1);
 \draw (0) to (u);
 \draw (0) to [out=-60,in=-120,relative] node[pos=0.8]{\rotatebox{20}{\footnotesize $\bowtie$}} (u);
 \draw (0) to (r);
 \draw (0) to [out=-60,in=-120,relative] node[pos=0.8]{\rotatebox{100}{\footnotesize $\bowtie$}} (r);
 \draw (0) to (l);
 \draw (0) to [out=-60,in=-120,relative] node[pos=0.8]{\rotatebox{-30}{\footnotesize $\bowtie$}} (l);
 \fill (0) circle (0.7mm); \fill (u) circle (0.7mm); \fill (l) circle (0.7mm); \fill (r) circle (0.7mm);
\end{tikzpicture}
 \caption{The three puzzle pieces (triangle piece, $1$-puncture piece, $2$-puncture piece) and the tagged triangulation $T_3$}\label{pp}
\end{figure}
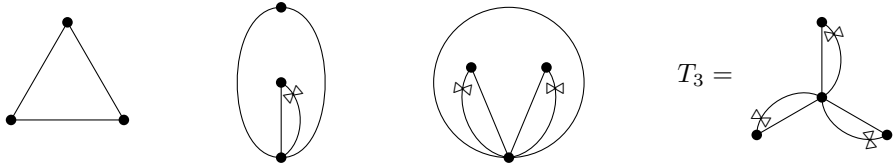

\subsection{Modifications of tagged arcs}

 In this subsection, unless otherwise noted, let $T$ be a tagged triangulation of $(S,M)$ satisfying $(\Diamond)$. To prove Theorems \ref{main} and \ref{ec}, we define modifications of tagged arcs with respect to $T$.

 Let $\de \notin T$ be a tagged arc of $(S,M)$. First, we change tags of $\delta$ at a puncture $p$ if $\de$ and a tagged arc of $T$ are tagged notched at $p$, and denote it by $\hat{\de}$ (see Figure \ref{hat}). Note that a notched arc of $T$ is a $1$-notched arc inside a pair of conjugate arcs of $T$ by $(\Diamond)$. Second, we construct a deformed curve ${\sf M}'_T(\hat{\de})$ as follows: for $\overline{\de} \notin T$,
\begin{itemize}\setlength{\leftskip}{-5mm}
 \item if $\hat{\de}$ is a plain arc, ${\sf M}'_T(\hat{\de})=\hat{\de}$;
 \item if $\hat{\de}$ is a notched arc and is not a loop, ${\sf M}'_T(\hat{\de})$ is obtained from $\hat{\de}$ by replacing its ends tagged notched as in the left diagram of Figure \ref{modification};
 \item if $\hat{\de}$ is a $2$-notched loop and there are both sides of $\hat{\de}$ in the same puzzle piece divided by $T$, ${\sf M}'_T(\hat{\de})$ is obtained from $\hat{\de}$ by replacing its ends as in the middle diagram of Figure \ref{modification};
 \item otherwise, ${\sf M}'_T(\hat{\de})$ is obtained from $\hat{\de}$ by replacing its ends as in the right diagram of Figure \ref{modification};
\end{itemize}
for $\overline{\de} \in T$, in particular, $\de$ is a notched arc since $\de \notin T$,
\begin{itemize}\setlength{\leftskip}{-5mm}
 \item if $\hat{\de}$ is a $1$-notched arc, ${\sf M}'_T(\hat{\de})$ is a $1$-punctured loop corresponding to $\hat{\de}$;
 \item if $\hat{\de}$ is a $2$-notched arc, ${\sf M}'_T(\hat{\de})$ is a pair of cycles which surround each endpoint of $\hat{\de}$ and do not include any punctures in their curves (we call this circle a \emph{$1$-punctured cycle}).
\end{itemize}
 Finally, we change tags of ${\sf M}'_T(\hat{\de})$ at $p$ again if $\de$ and a tagged arc of $T$ are tagged notched at $p$. We say that the result is a {\it modified tagged arc of $\de$ with respect to $T$}, and denote it by ${\sf M}_T(\de)$ (see Figure \ref{hat} and Example \ref{ex}). 
  
\begin{figure}[ht]
\begin{tikzpicture}[baseline=-5mm]
 \coordinate (0) at (0,0);
 \coordinate (1) at (0,-1);
 \draw (0) to (1);
 \draw (0) to node[pos=0.2]{\rotatebox{90}{\footnotesize $\bowtie$}} node[above]{$\de$} (-1,0);
 \draw (0) to [out=80,in=100,relative] node[pos=0.2]{\rotatebox{40}{\footnotesize $\bowtie$}} node[right]{in $T$} (1);
 \fill(0) circle (0.7mm); \fill (1) circle (0.7mm);
\end{tikzpicture}
     $\hspace{3mm}\rightarrow\hspace{3mm}$
\begin{tikzpicture}[baseline=-5mm]
 \coordinate (0) at (0,0);
 \coordinate (1) at (0,-1);
 \draw (0) to (1);
 \draw (0) to node[above]{$\hat{\de}$} (-1,0);
 \draw (0) to [out=80,in=100,relative] node[pos=0.2]{\rotatebox{40}{\footnotesize $\bowtie$}} (1);
 \fill(0) circle (0.7mm); \fill (1) circle (0.7mm);
\end{tikzpicture}
     $\hspace{3mm}\rightarrow\hspace{3mm}$
\begin{tikzpicture}[baseline=-5mm]
 \coordinate (0) at (0,0);
 \coordinate (1) at (0,-1);
 \draw (0) to (1);
 \draw (0) to node[pos=0.2]{\rotatebox{90}{\footnotesize $\bowtie$}} node[above]{${\sf M}_T(\de)$} (-1,0);
 \draw (0) to [out=80,in=100,relative] node[pos=0.2]{\rotatebox{40}{\footnotesize $\bowtie$}} (1);
 \fill(0) circle (0.7mm); \fill (1) circle (0.7mm);
\end{tikzpicture}
\caption{From $\de$ to $\hat{\de}$ and ${\sf M}_T(\de)$}\label{hat}
\end{figure}

\begin{figure}[ht]
\begin{tikzpicture}[baseline=6mm]
 \coordinate (0) at (0,0);
 \coordinate (1) at (90:1.3);
 \draw[blue] (0) to node[pos=0.15]{\footnotesize $\bowtie$} (1);
 \fill(0) circle (0.7mm);
\end{tikzpicture}
     $\hspace{3mm}\rightarrow\hspace{3mm}$
\begin{tikzpicture}[baseline=6mm]
 \coordinate (0) at (0,0);
 \coordinate (1) at (90:1.3);
 \draw[blue] (0,0.3) to (1); \draw[blue](0) circle(3mm); 
 \fill(0) circle (0.7mm);
\end{tikzpicture}
     \hspace{10mm}
\begin{tikzpicture}[baseline=6mm]
 \coordinate (0) at (0,0);
 \coordinate (l) at (140:1.3);
 \coordinate (r) at (40:1.3);
 \draw (l)--(0)--(r) (l)--(r);
 \draw[blue] (0) .. controls (-0.25,0.3) and (-0.25,1) ..node[pos=0.3]{\rotatebox{20}{\footnotesize $\bowtie$}} (-0.25,1.3);
 \draw[blue] (0) .. controls (0.25,0.3) and (0.25,1) ..node[pos=0.3]{\rotatebox{-15}{\footnotesize $\bowtie$}} (0.25,1.3);
 \fill(0) circle (0.7mm);
\end{tikzpicture}
     $\hspace{3mm}\rightarrow\hspace{3mm}$
\begin{tikzpicture}[baseline=6mm]
 \coordinate (0) at (0,0);
 \coordinate (l) at (140:1.3);
 \coordinate (r) at (40:1.3);
 \coordinate (l') at (120:0.5);
 \coordinate (r') at (60:0.5);
 \draw (l)--(0)--(r) (l)--(r);
 \draw[blue] (-0.25,1.3)--(l') (0.25,1.3)--(r'); \draw[blue](0) circle(3mm);
 \draw[blue] (l') arc (120:420:0.5);
 \fill(0) circle (0.7mm);
\end{tikzpicture}
     \hspace{10mm}
\begin{tikzpicture}[baseline=0mm]
 \coordinate (0) at (0,0);
 \coordinate (lu) at (140:1);
 \coordinate (ru) at (40:1);
 \coordinate (ld) at (-140:1);
 \coordinate (rd) at (-40:1);
 \coordinate (u) at (90:1);
 \coordinate (d) at (-90:1);
 \draw[blue] (u)--node[pos=0.8]{\footnotesize $\bowtie$}(0)--node[pos=0.25]{\footnotesize $\bowtie$}(d);
 \draw (lu)--(0)--(ru) (ld)--(0)--(rd);
 \fill(0) circle (0.7mm);
\end{tikzpicture}
     $\hspace{3mm}\rightarrow\hspace{3mm}$
\begin{tikzpicture}[baseline=0mm]
 \coordinate (0) at (0,0);
 \coordinate (lu) at (140:1);
 \coordinate (ru) at (40:1);
 \coordinate (ld) at (-140:1);
 \coordinate (rd) at (-40:1);
 \coordinate (u) at (90:1);
 \coordinate (d) at (-90:1);
 \draw[blue] (u)--(0,0.5) (0,-0.5)--(d);
 \draw (lu)--(0)--(ru) (ld)--(0)--(rd); \draw[blue](0) circle(3mm); \draw[blue](0) circle(5mm);
 \fill(0) circle (0.7mm);
\end{tikzpicture}
\caption{Modifications ${\sf M}'_T(\hat{\de})$ of $\hat{\de}$}\label{modification}
\end{figure}

\begin{example}\label{ex}
 We consider the following tagged triangulation $T$ and tagged arcs $\de_1$, $\de_2$ and $\de_3$:
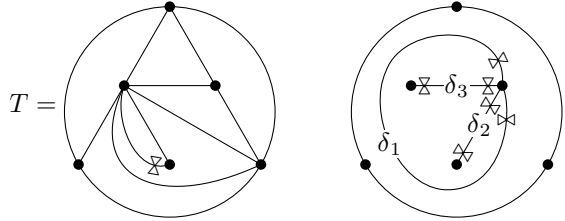
\[
T=
\begin{tikzpicture}[baseline=0mm,scale=0.7]
 \coordinate (u) at (0,2);
 \coordinate (l) at (-150:2);
 \coordinate (r) at (-30:2);
 \coordinate (cd) at (0,-1);
 \coordinate (cl) at (150:1);
 \coordinate (cr) at (30:1);
 \draw (0,0) circle (2);
 \draw (r)--(cr)--(u)--(cl)--(cr);
 \draw (r)--(cl)--(l);
 \draw (cl) .. controls (-145:2) and (-100:2) .. (r);
 \draw (cl) to (cd);
 \draw (cl) to [out=-50,in=-100,relative] node[pos=0.85]{\rotatebox{80}{\footnotesize $\bowtie$}} (cd);
 \fill(u) circle (1mm); \fill(l) circle (1mm); \fill(r) circle (1mm); \fill(cd) circle (1mm); \fill(cl) circle (1mm); \fill(cr) circle (1mm);
\end{tikzpicture}
     \hspace{10mm}
\begin{tikzpicture}[baseline=0mm,scale=0.7]
 \coordinate (u) at (0,2);
 \coordinate (l) at (-150:2);
 \coordinate (r) at (-30:2);
 \coordinate (cd) at (0,-1);
 \coordinate (cl) at (150:1);
 \coordinate (cr) at (30:1);
 \draw (0,0) circle (2);
 \draw (cr) .. controls (60:2) and (120:2) .. node[pos=0.15]{\rotatebox{25}{\footnotesize $\bowtie$}} (150:1.55);
 \draw (150:1.55) .. controls (-180:1.7) and (-130:1.8) .. node[fill=white,inner sep=1]{$\de_1$} (-100:1.5);
 \draw (-100:1.5) .. controls (-60:1.8) and (0:1.1) .. node[pos=0.7]{\rotatebox{0}{\footnotesize $\bowtie$}} (cr);
 \draw (cl) to node[pos=0.15]{\rotatebox{90}{\footnotesize $\bowtie$}} node[pos=0.85]{\rotatebox{90}{\footnotesize $\bowtie$}} node[fill=white,inner sep=0.1]{$\de_3$} (cr);
 \draw (cd) to node[pos=0.15]{\rotatebox{150}{\footnotesize $\bowtie$}} node[pos=0.75]{\rotatebox{150}{\footnotesize $\bowtie$}} node[fill=white,inner sep=0.1]{$\de_2$} (cr);
 \fill(u) circle (1mm); \fill(l) circle (1mm); \fill(r) circle (1mm); \fill(cd) circle (1mm); \fill(cl) circle (1mm); \fill(cr) circle (1mm);
\end{tikzpicture}
\]
 Then the corresponding modified tagged arcs ${\sf M}_T(\de_i)$ with respect to $T$ are given as follows:
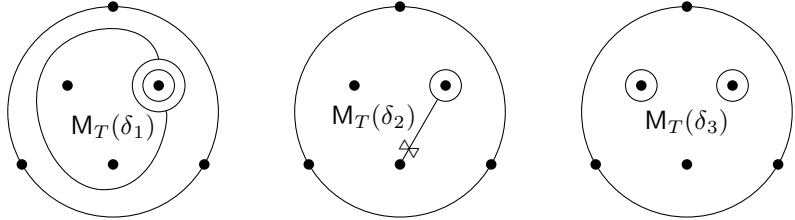
\[
\begin{tikzpicture}[baseline=0mm,scale=0.7]
 \coordinate (u) at (0,2);
 \coordinate (l) at (-150:2);
 \coordinate (r) at (-30:2);
 \coordinate (cd) at (0,-1);
 \coordinate (cl) at (150:1);
 \coordinate (cr) at (30:1);
 \draw (0,0) circle (2); \draw(cr) circle(3mm); \draw(cr) circle(5mm);
 \draw (cr)+(90:0.5) .. controls (70:2) and (120:2) .. (150:1.55);
 \draw (150:1.55) .. controls (-180:1.7) and (-130:1.8) .. (-100:1.5);
 \draw (cr)+(-75:0.5) .. controls (0:1.1) and (-60:1.8) .. (-100:1.5);
 \fill(u) circle (1mm); \fill(l) circle (1mm); \fill(r) circle (1mm); \fill(cd) circle (1mm); \fill(cl) circle (1mm); \fill(cr) circle (1mm);
 \node at(0,-0.3) {${\sf M}_T(\de_1)$};
\end{tikzpicture}     \hspace{10mm}
\begin{tikzpicture}[baseline=0mm,scale=0.7]
 \coordinate (u) at (0,2);
 \coordinate (l) at (-150:2);
 \coordinate (r) at (-30:2);
 \coordinate (cd) at (0,-1);
 \coordinate (cl) at (150:1);
 \coordinate (cr) at (30:1);
 \draw (0,0) circle (2); \draw(cr) circle(3mm);
 \draw (cr)+(-120:0.3) to node[pos=0.75]{\rotatebox{150}{\footnotesize $\bowtie$}} (cd);
 \fill(u) circle (1mm); \fill(l) circle (1mm); \fill(r) circle (1mm); \fill(cd) circle (1mm); \fill(cl) circle (1mm); \fill(cr) circle (1mm);
 \node at(-0.5,-0.1) {${\sf M}_T(\de_2)$};
\end{tikzpicture}     \hspace{10mm}
\begin{tikzpicture}[baseline=0mm,scale=0.7]
 \coordinate (u) at (0,2);
 \coordinate (l) at (-150:2);
 \coordinate (r) at (-30:2);
 \coordinate (cd) at (0,-1);
 \coordinate (cl) at (150:1);
 \coordinate (cr) at (30:1);
 \draw (0,0) circle (2); \draw(cl) circle(3mm); \draw(cr) circle(3mm);
 \fill(u) circle (1mm); \fill(l) circle (1mm); \fill(r) circle (1mm); \fill(cd) circle (1mm); \fill(cl) circle (1mm); \fill(cr) circle (1mm);
 \node at(0,-0.2) {${\sf M}_T(\de_3)$};
\end{tikzpicture}
\]
\end{example}

 We can define the intersection number of a modified tagged arc ${\sf m}$ and a tagged arc $\de$ in the same way as of tagged arcs, denote by $\Int({\sf m},\de)$. Although the map ${\sf M}_T$ may seem strange, it is defined so as to satisfy the following properties.

\begin{proposition}\label{mT}
 $(1)$ For a tagged arc $\de$ of $(S,M)$, we have $\Int(T,\de)=\Int({T,\sf M}_T(\de))$.\par
 $(2)$ The map ${\sf M}_T$ restricting to the set
\[
 A := \{\text{tagged arcs $\de$ of $(S,M)$ $\mid$ $\de \notin T$ and ${\sf M}_T(\de)$ is not a pair of $1$-punctured cycles}\}
\]
 is injective. Moreover, if ${\sf M}_T(\de)={\sf M}_T(\epsilon)$ for $\de \in A$ and any tagged arc $\epsilon \notin T$, then $\de=\epsilon$ holds.
\end{proposition}

\begin{proof}
 The assertions follow from the definition of intersection numbers and the map ${\sf M}_T$.
\end{proof}

\begin{remark}\label{not1to1}
 For a tagged arc $\de \notin T \cup A$ of $(S,M)$, ${\sf M}(\de)$ does not always correspond to $\de$ bijectively. Indeed, we consider the following tagged triangulation $T$ and tagged arcs $\de$, $\epsilon$:
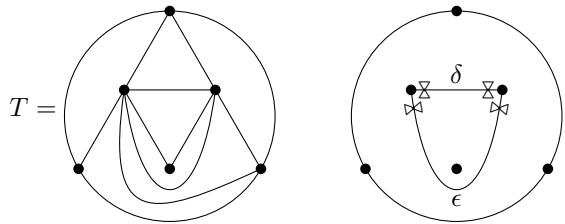
\[
 T=
\begin{tikzpicture}[baseline=0mm,scale=0.7]
 \coordinate (u) at (0,2);
 \coordinate (l) at (-150:2);
 \coordinate (r) at (-30:2);
 \coordinate (cd) at (0,-1);
 \coordinate (cl) at (150:1);
 \coordinate (cr) at (30:1);
 \draw (0,0) circle (2);
 \draw (r)--(cr)--(u)--(cl)--(cr);
 \draw (cl)--(l);
 \draw (cl) .. controls (-120:2.5) and (-100:2) .. (r);
 \draw (cl)--(cd)--(cr);
  \draw (cr) .. controls (-75:2.1) and (-105:2.1) .. (cl);
 \fill(u) circle (1mm); \fill(l) circle (1mm); \fill(r) circle (1mm); \fill(cd) circle (1mm); \fill(cl) circle (1mm); \fill(cr) circle (1mm);
\end{tikzpicture}
     \hspace{10mm}
\begin{tikzpicture}[baseline=0mm,scale=0.7]
 \coordinate (u) at (0,2);
 \coordinate (l) at (-150:2);
 \coordinate (r) at (-30:2);
 \coordinate (cd) at (0,-1);
 \coordinate (cl) at (150:1);
 \coordinate (cr) at (30:1);
 \draw (0,0) circle (2);
 \draw (cr) to node[pos=0.15]{\rotatebox{90}{\footnotesize $\bowtie$}} node[pos=0.85]{\rotatebox{90}{\footnotesize $\bowtie$}}  (cl);
 \draw (cr) .. controls (-75:2.1) and (-105:2.1) .. node[pos=0.05]{\rotatebox{-10}{\footnotesize $\bowtie$}} node[pos=0.95]{\rotatebox{10}{\footnotesize $\bowtie$}}  (cl);
 \fill(u) circle (1mm); \fill(l) circle (1mm); \fill(r) circle (1mm); \fill(cd) circle (1mm); \fill(cl) circle (1mm); \fill(cr) circle (1mm);
 \node at(0,-1.6) {$\epsilon$};
 \node at(0,0.8) {$\de$};
\end{tikzpicture}
\]  
 Then the corresponding modified tagged arcs ${\sf M}_T(\de)$ and ${\sf M}_T(\epsilon)$ with respect to $T$ are given as follows:
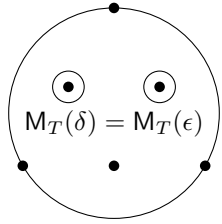
\[
\begin{tikzpicture}[baseline=0mm,scale=0.7]
 \coordinate (u) at (0,2);
 \coordinate (l) at (-150:2);
 \coordinate (r) at (-30:2);
 \coordinate (cd) at (0,-1);
 \coordinate (cl) at (150:1);
 \coordinate (cr) at (30:1);
 \draw (0,0) circle (2); \draw(cl) circle(3mm); \draw(cr) circle(3mm);
 \fill(u) circle (1mm); \fill(l) circle (1mm); \fill(r) circle (1mm); \fill(cd) circle (1mm); \fill(cl) circle (1mm); \fill(cr) circle (1mm);
 \node at(0,-0.2) {${\sf M}_T(\de)={\sf M}_T(\epsilon)$};
\end{tikzpicture}
\]
\end{remark}

 The following theorem is a key of the proofs of Theorems \ref{main} and \ref{ec}.

\begin{theorem}\label{intinj}
 If modified tagged arcs ${\sf m}$ and ${\sf m}'$ with respect to $T$ have $\Int(T,{\sf m})=\Int(T,{\sf m}')$, then ${\sf m}={\sf m}'$.
\end{theorem}

 We will prove Theorem \ref{intinj} in Section \ref{pfintinj}.

\begin{corollary}\label{taginj}
 If tagged arcs $\de$ and $\epsilon$ in $A$ have $\Int(T,\de)=\Int(T,\epsilon)$, then $\de=\epsilon$.
\end{corollary}

\begin{proof}
 Proposition \ref{mT}(1) implies that $\Int(T,{\sf M}_T(\de))=\Int(T,{\sf M}_T(\epsilon))$. By Theorem \ref{intinj} and Proposition \ref{mT}(2), we have $\de=\epsilon$.
\end{proof}

 These results provide the proofs of Theorems \ref{main} and \ref{ec}.

\begin{proof}[Proof of Theorem \ref{main}]
 By changing tags, we can assume that $T$ satisfies $(\Diamond)$. Let $T'=\{\de_1,\ldots,\de_n\}$ and $T''=\{\epsilon_1,\ldots,\epsilon_n\}$ be tagged triangulations of $(S,M)$ such that $\Int(T,\de_i)=\Int(T,\epsilon_i)$ for any $i$. We set $V=(v_1 \cdots v_n)=\Int(T,T')$, where $v_i=\Int(T,\de_i) \in \bZ_{\ge 0}^n$. Without loss of generality, we assume that $\de_i \in A$ for $i \in \{1,\ldots,k\}$ and $\de_j \notin A$ for $j \in \{k+1,\ldots,n\}$, that is, either $\de_j, \epsilon_j \in T$ or ${\sf M}_T(\de_j)={\sf M}_T(\epsilon_j)$ is a pair of $1$-punctured cycles by Theorem \ref{intinj}. Corollary \ref{taginj} implies that $\de_i=\epsilon_i$ for $i \in \{1,\ldots,k\}$.

 If $T' \neq T''$, then there exist $f, g \in \{k+1,\dots,n\}$ such that $\Int(\de_f,\epsilon_g) \neq 0$. Otherwise, it conflicts with the maximality of $T'$. Since $\overline{\de_f}$ and $\overline{\epsilon_g}$ are contained in $T$, $\de_f$ and $\epsilon_g$ must have different tags at the common endpoint. Without loss of generality, we assume that $\de_f$ is contained in $T$ and ${\sf M}_T(\de_g)={\sf M}_T(\epsilon_g)$ is a pair of $1$-punctured cycles. Since $\de_f$ and $\de_g$ have the common endpoint and $\Int(\de_f,\de_g)=0$, $\de_f$ is a $1$-notched arc of $T$ by $(\Diamond)$. Then $\hat{\de_g}$ is not a $2$-notched arc, thus it is contradictory to the fact that ${\sf M}_T(\epsilon_g)$ is a pair of $1$-punctured cycles. This finishes the proof.
\end{proof}

\begin{proof}[Proof of Theorem \ref{ec}]
 By changing tags, we can assume that $T$ satisfies $(\Diamond)$. First, we prove ``if'' part. Let $\de$ and $\epsilon$ be tagged arcs with a common non-zero intersection vector $\Int(T,\de)=\Int(T,\epsilon)$ with respect to $T$. Then $\de$ and $\epsilon$ are not contained in $T$ by definition of intersection vectors. By Corollary \ref{taginj}, it suffice to show that if ${\sf M}_T(\de)$ is a pair of $1$-punctured cycles, then $\de=\epsilon$. In this case, $\de$ and $\epsilon$ are $2$-notched arcs such that $\overline{\de}$ and $\overline{\epsilon}$ are plain arcs of $T$ such that both endpoints of $\overline{\de}$ correspond to ones of $\overline{\epsilon}$ since ${\sf M}_T(\de)={\sf M}_T(\epsilon)$ by Theorem \ref{intinj}. Therefore, we have $\de=\epsilon$ by the assumption.

 Second, we prove ``only if'' part. Suppose that $T$ has a pair of different plain arcs $\g$ and $\g'$ such that both endpoints of $\g$ correspond to ones of $\g'$ which are punctures. Let $\de$ and $\epsilon$ be $2$-notched arcs such that $\overline{\de}=\g$ and $\overline{\epsilon}=\g'$. Then we have $\de\neq\epsilon$ and $\Int(T,\de)=\Int(T,\epsilon)$ which is not zero, that is, $T$ does not detect tagged arcs.
\end{proof}

\section{Proof of Theorem \ref{class}}\label{seclass}
 First of all, we prove Theorem \ref{class}(3).

\begin{proof}[Proof of Theorem \ref{class}(3)]
 It is easy to show that for $(S,M)$ as in Theorem \ref{class}(3), any tagged triangulation of $(S,M)$ detects tagged arcs by Theorem \ref{ec}. Conversely, if $(S,M)$ is not one of the above cases, a part of $(S,M)$ must have one of the pairs of plain arcs $\de$ and $\epsilon$ as in Table \ref{part}. Then a tagged triangulation $T$ of $(S,M)$ including $\de$ and $\epsilon$ does not detect tagged arcs by Theorem \ref{ec}.
\end{proof}

\renewcommand{\arraystretch}{1.7}
{\begin{table}[ht]
\begin{tabular}{c|c|c|c|c|c}
 $g$ & \multicolumn{4}{c|}{0} & $\geq1$ \\\hline
 $b$ & $0$ & $1$ & $2$ & $\geq3$ & any \\\hline
 $p$ & $\geq4$ & $\geq3$ & $\geq2$ & $\geq1$ & $\geq1$ \\\hline
 $\de, \epsilon$ &
\begin{tikzpicture}[baseline=0mm,scale=0.5]
 \coordinate (u) at (0,	0.5);
 \coordinate (d) at (0,-0.5);
 \coordinate (l) at (-1,0);
 \coordinate (r) at (1,0);
 \draw[semithick] (l)-- node[fill=white,inner sep=1]{$\epsilon$} (r);
 \draw[semithick] (l) .. controls (-1,1.3) and (1,1.3) .. node[above]{$\de$} (r);
 \fill(u) circle (1.4mm); \fill(l) circle (1.4mm); \fill(r) circle (1.4mm); \fill(d) circle (1.4mm);
\end{tikzpicture} &
\begin{tikzpicture}[baseline=0mm,scale=0.5]
 \coordinate (l) at (-1,0);
 \coordinate (r) at (1,0);
 \coordinate (c) at (0,0);
 \draw (0,0) circle (2);
 \draw[semithick] (l) .. controls (-1,1) and (1,1) .. node[above]{$\de$} (r);
 \draw[semithick] (l) .. controls (-1,-1) and (1,-1) .. node[below]{$\epsilon$} (r);
 \fill(c) circle (1.4mm); \fill(l) circle (1.4mm); \fill(r) circle (1.4mm);
 \node at(0,2.2) {};
\end{tikzpicture} &
\begin{tikzpicture}[baseline=0mm,scale=0.5]
 \coordinate (l) at (-1.2,0);
 \coordinate (r) at (1.2,0);
 \draw (0,0) circle (2);
 \draw (0) circle (0.6);
 \draw[semithick] (l) .. controls (-1,1.2) and (1,1.2) .. node[above]{$\de$} (r);
 \draw[semithick] (l) .. controls (-1,-1.2) and (1,-1.2) .. node[below]{$\epsilon$} (r);
 \fill(l) circle (1.4mm); \fill(r) circle (1.4mm);
\end{tikzpicture} &
\begin{tikzpicture}[baseline=0mm,scale=0.5]
 \coordinate (c) at (0,0);
 \draw (0,0) circle (2);
 \draw (-1,0) circle (0.5); \draw (1,0) circle (0.5);
 \draw[semithick] (c) .. controls (-0.7,1.5) and (-1.8,1) .. node[above,pos=0.3]{$\de$} (-1.8,0);
 \draw[semithick] (c) .. controls (-0.7,-1.5) and (-1.8,-1) .. (-1.8,0);
 \draw[semithick] (c) .. controls (0.7,1.5) and (1.8,1) .. node[above,pos=0.3]{$\epsilon$} (1.8,0);
 \draw[semithick] (c) .. controls (0.7,-1.5) and (1.8,-1) .. (1.8,0);
 \fill(c) circle (1.4mm);
\end{tikzpicture} &
\begin{tikzpicture}[baseline=0mm,scale=0.5]
\coordinate (u) at (0,1.5);
\coordinate (d) at (0,-1.5);
\coordinate (m1) at (-3,0);
\coordinate (m2) at (-1.5,0.4);
\coordinate (m3) at (1.5,0.4);
\coordinate (m4) at (3,0);
\coordinate(m5)at (-1.1,0.15);
\coordinate(m6)at (1.1,0.15);
\coordinate (p) at (-0.4,-0.8);
\coordinate (m7) at (0,-0.15);
 \fill(p) circle(1.4mm);
 \draw (u) to [out=180,in=90] (m1);
 \draw (m1) to [out=-90,in=180] (d);
 \draw (d) to [out=0,in=-90] (m4);
 \draw (m4) to [out=90,in=0] (u);
 \draw (m2) to [out=-40,in=-140] (m3); 
 \draw (m5) to [out=20,in=160] (m6); 
 \draw[semithick] (d) to [out=170,in=-170] (m7);
 \draw[semithick,dashed] (d) to [out=10,in=-10] (m7);
 \draw[semithick] (p).. controls (-3,-0.5)and (-3, 1) .. (0,1);
 \draw[semithick] (0,1).. controls (3,1)and (3, -0.5) .. (p);
\node at(0,-2) {$\epsilon$};
 \node at(2.5,0) {$\de$};
\end{tikzpicture}
  \end{tabular}\vspace{5mm}
 \caption{Tagged arcs $\de$ and $\epsilon$ connecting two (possibly same) common punctures such that $\overline{\de} \neq \overline{\epsilon}$, where $g$ is the genus, $b$ is the number of components of the boundary and $p$ is the number of punctures in $(S,M)$}\label{part}
\end{table}}

 We consider the case that $S$ is not closed. The following lemma is basic.

\begin{lemma}\label{bpconn}
 If $S$ is not closed, then there is a tagged triangulation of $(S,M)$ whose any tagged arc is a plain arc with at least one marked point on the boundary of $S$ as its endpoints.
\end{lemma}

\begin{proof}
 For a puncture $p$ of $(S,M)$, we can construct triangles with $p$ and two marked points $l$ and $r$ (possibly $l=r$) on the boundary of $S$ as follows:
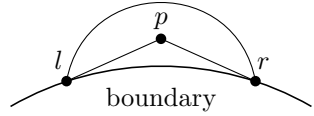
\[
\begin{tikzpicture}[baseline=0mm]
 \coordinate (r) at (15:1.3);
 \coordinate (l) at (165:1.3);
 \coordinate (p) at (0,0.9);
 \draw[semithick] (0:2) arc (60:120:4);
 \draw (l)--(p)--(r);
 \draw (l) .. controls (120:2) and (60:2) .. (r);
 \fill(l) circle (0.7mm); \fill(r) circle (0.7mm); \fill(p) circle (0.7mm);
 \node[above] at(p){$p$}; \node[left=3,above=1] at(l){$l$}; \node[right=3,above=1] at(r){$r$}; \node at(0,0.1){boundary};
\end{tikzpicture}
\]
 Then, for another puncture $q$ of $(S,M)$, it is easy to construct triangles with $q$, $l$ and $r$ in the same way. We have the set of triangles containing all punctures of $(S,M)$ by the inductive construction. There is a tagged triangulation of $(S,M)$ containing these triangles, thus it is what is desired.
\end{proof}

\begin{proof}[Proof of Theorem \ref{class}(1)]
 The assertion follows from Theorem \ref{ec} and Lemma \ref{bpconn}.
\end{proof}

 Next, we consider the case that $S$ is closed. In the rest of this section, let $g$ be the genus of $S$ and $p$ be the number of punctures of $(S,M)$. To prove Theorem \ref{class}(2), we need some preparations.

\begin{lemma}\label{noloops}
 We assume that $S$ is closed and $g>0$. If a tagged triangulation $T$ of $(S,M)$ has loops, then $T$ does not detect tagged arcs.
\end{lemma}

\begin{proof}
 A puzzle piece with loops is one of the followings:
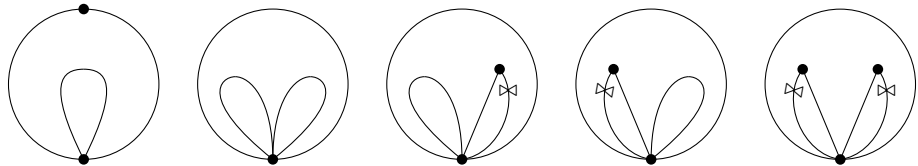
\[\begin{tikzpicture}
 \coordinate (0) at (0,0); 
 \coordinate (1) at (0,1);
 \coordinate (2) at (0,-1);
 \draw (0) circle (1);
 \draw (2) .. controls (0.5,0) and (0.3,0.2) .. (0,0.2);
 \draw (2) .. controls (-0.5,0) and (-0.3,0.2) .. (0,0.2);
 \fill(1) circle (0.7mm); \fill(2) circle (0.7mm);
\end{tikzpicture}
     \hspace{5mm}
\begin{tikzpicture}
 \coordinate (0) at (0,0);
 \coordinate (1) at (-0.5,0.1);
 \coordinate (1') at (0.5,0.1);
 \coordinate (2) at (0,-1);
 \draw (0) circle (1);
 \draw (2) .. controls (-1,-0.2) and (-0.7,0.1) .. (1);
 \draw (2) .. controls (0,-0.2) and (-0.3,0.1) .. (1);
 \draw (2) .. controls (1,-0.2) and (0.7,0.1) .. (1');
 \draw (2) .. controls (0,-0.2) and (0.3,0.1) .. (1');
 \fill(2) circle (0.7mm);
\end{tikzpicture}
 \hspace{5mm}
\begin{tikzpicture}
 \coordinate (0) at (0,0);
 \coordinate (1) at (-0.5,0.1);
 \coordinate (1') at (0.5,0.2); \fill (1') circle (0.7mm);
 \coordinate (2) at (0,-1);
 \draw (0) circle (1);
 \draw (2) .. controls (-1,-0.2) and (-0.7,0.1) .. (1);
 \draw (2) .. controls (0,-0.2) and (-0.3,0.1) .. (1);
 \draw (1') to (2);
 \draw (1') to [out=60,in=120,relative] node[pos=0.2]{\rotatebox{180}{\footnotesize $\bowtie$}} (2);
 \fill(2) circle (0.7mm);
\end{tikzpicture}
     \hspace{5mm}
\begin{tikzpicture}
 \coordinate (0) at (0,0);
 \coordinate (1) at (-0.5,0.2); \fill (1) circle (0.7mm);
 \coordinate (1') at (0.5,0.1);
 \coordinate (2) at (0,-1);
 \draw (0) circle (1);
 \draw (1) to (2);
 \draw (1) to [out=-60,in=-120,relative] node[pos=0.2]{\rotatebox{170}{\footnotesize $\bowtie$}} (2);
 \draw (2) .. controls (1,-0.2) and (0.7,0.1) .. (1');
 \draw (2) .. controls (0,-0.2) and (0.3,0.1) .. (1');
 \fill(2) circle (0.7mm);
\end{tikzpicture}
\hspace{5mm}
\begin{tikzpicture}
 \coordinate (0) at (0,0); 
 \coordinate (1) at (-0.5,-0.8); \fill (1) circle (0.7mm);
 \coordinate (1') at (0.5,-0.8); \fill (1') circle (0.7mm);
 \coordinate (2) at (0,-2); \fill (2) circle (0.7mm);
 \draw (0,-1) circle (1);
 \draw (1) to  (2);
 \draw (1) to [out=-60,in=-120,relative] node[pos=0.2]{\rotatebox{170}{\footnotesize $\bowtie$}} (2);
 \draw (1') to  (2);
 \draw (1') to [out=60,in=120,relative] node[pos=0.2]{\rotatebox{180}{\footnotesize $\bowtie$}}  (2);
\end{tikzpicture}
\]
 In these puzzle pieces, only the $2$-punctured piece does not have a pairs of different plain arcs connecting two (possibly same) common punctures. Therefore, by Theorem \ref{ec}, if a tagged triangulation $T$ with loops of $(S,M)$ detects tagged arcs, then $T$ is obtained by gluing two $2$-punctured pieces and by changing tags if necessary. This is in conflict with $g>0$.
\end{proof}

\begin{lemma}\label{nonotches}
We assume that $S$ is closed and $g>0$. If a tagged triangulation $T$ of $(S,M)$ satisfies $(\Diamond)$ and has 1-notched arcs, then $T$ does not detect tagged arcs.
\end{lemma}

\begin{proof}
A puzzle piece with 1-notched arcs is one of the followings:
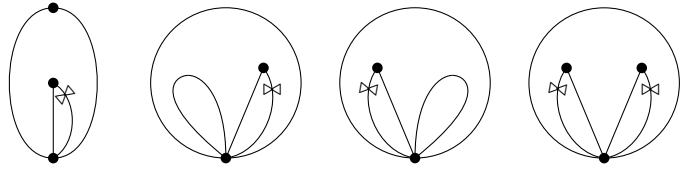
\[\begin{tikzpicture}
 \coordinate (0) at (0,0);
 \coordinate (1) at (0,-1);
 \coordinate (2) at (0,-2);
 \draw (2) to [out=180,in=180]  (0);
 \draw (2) to [out=0,in=0] (0);
 \draw (1) to  (2);
 \draw (1) to [out=60,in=120,relative] node[pos=0.2]{\rotatebox{40}{\footnotesize $\bowtie$}} (2);
 \fill(0) circle (0.7mm); \fill (1) circle (0.7mm); \fill (2) circle (0.7mm);
\end{tikzpicture}
 \hspace{5mm}
\begin{tikzpicture}
 \coordinate (0) at (0,0);
 \coordinate (1) at (-0.5,0.1);
 \coordinate (1') at (0.5,0.2); \fill (1') circle (0.7mm);
 \coordinate (2) at (0,-1);
 \draw (0) circle (1);
 \draw (2) .. controls (-1,-0.2) and (-0.7,0.1) .. (1);
 \draw (2) .. controls (0,-0.2) and (-0.3,0.1) .. (1);
 \draw (1') to (2);
 \draw (1') to [out=60,in=120,relative] node[pos=0.2]{\rotatebox{180}{\footnotesize $\bowtie$}} (2);
 \fill(2) circle (0.7mm);
\end{tikzpicture}
     \hspace{5mm}
\begin{tikzpicture}
 \coordinate (0) at (0,0);
 \coordinate (1) at (-0.5,0.2); \fill (1) circle (0.7mm);
 \coordinate (1') at (0.5,0.1);
 \coordinate (2) at (0,-1);
 \draw (0) circle (1);
 \draw (1) to (2);
 \draw (1) to [out=-60,in=-120,relative] node[pos=0.2]{\rotatebox{170}{\footnotesize $\bowtie$}} (2);
 \draw (2) .. controls (1,-0.2) and (0.7,0.1) .. (1');
 \draw (2) .. controls (0,-0.2) and (0.3,0.1) .. (1');
 \fill(2) circle (0.7mm);
\end{tikzpicture}
\hspace{5mm}
\begin{tikzpicture}
 \coordinate (0) at (0,0); 
 \coordinate (1) at (-0.5,-0.8); \fill (1) circle (0.7mm);
 \coordinate (1') at (0.5,-0.8); \fill (1') circle (0.7mm);
 \coordinate (2) at (0,-2); \fill (2) circle (0.7mm);
 \draw (0,-1) circle (1);
 \draw (1) to  (2);
 \draw (1) to [out=-60,in=-120,relative] node[pos=0.2]{\rotatebox{170}{\footnotesize $\bowtie$}} (2);
 \draw (1') to  (2);
 \draw (1') to [out=60,in=120,relative] node[pos=0.2]{\rotatebox{180}{\footnotesize $\bowtie$}}  (2);
\end{tikzpicture}
\]
In these puzzle pieces, only the $2$-punctured piece does not have a pairs of different plain arcs connecting two (possibly same) common punctures. Therefore, the assertion follows in the same way as Lemma \ref{noloops}.
\end{proof}

\begin{theorem}\cite[Theorem 1.1]{JR}\label{JR}
 We assume that $S$ is closed. If $p$ is the minimal integer to satisfy \eqref{ineq}, then there is a tagged triangulation $T$ of $(S,M)$ satisfying the following conditions:
\begin{itemize}
 \item[(T1)] any tagged arc of $T$ is a plain arc;
 \item[(T2)] any triangle of $T$ has three distinct vertices;
 \item[(T3)] the intersection of two distinct triangles of $T$ is either empty, a single vertex, or a single edge.
\end{itemize}
 Conversely, if there is a tagged triangulation of $(S,M)$ satisfying (T1)-(T3), then \eqref{ineq} holds.
\end{theorem}

\begin{proposition}\label{JRdetect}
 We assume that $S$ is closed and $g>0$. Then a tagged triangulation $T$ of $(S,M)$ satisfies (T1)-(T3) if and only if $T$ detects tagged arcs.
\end{proposition}

\begin{proof}
 We assume that $T$ satisfies (T1)-(T3) and does not detect tagged arcs. By Theorem \ref{ec}, there are tagged arcs $\de$ and $\epsilon$ of $T$ connecting two common punctures such that $\overline{\de} \neq \overline{\epsilon}$. Then they are not contained in a single triangle of $T$ by (T2).  The intersection of a triangle with $\de$ and a triangle with $\epsilon$ has two vertices and does not have an edge connecting them. It conflicts with (T3).

 Conversely, we assume that $T$ detects tagged arcs. By Lemma \ref{nonotches}, we can also assume that $T$ satisfies (T1). By Lemma \ref{noloops}, $T$ satisfies (T2). It is easy to show that if the intersection of two distinct triangles of $T$ is either two vertices, three vertices, or two edges, then there are tagged arcs $\de$ and $\epsilon$ of $T$ connecting two common punctures such that $\overline{\de} \neq \overline{\epsilon}$. Thus it is a contradiction by Theorem \ref{ec}. If the intersection of two distinct triangles of $T$ is three edges, then $(S,M)$ must be a sphere with exactly three punctures, thus it conflicts with our assumption. Therefore, $T$ satisfies (T3).
\end{proof}

\begin{proof}[Proof of Theorem \ref{class}(2)]
 When $g=0$, we have $p \ge 4$ by our assumption, in which case \eqref{ineq} holds. We consider the tagged triangulation
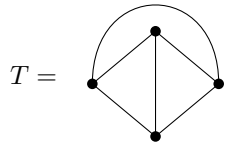
\[
 T=\hspace{3mm}
\begin{tikzpicture}[baseline=0mm,scale=0.7]
 \coordinate (u) at (0,	1);
 \coordinate (d) at (0,-1);
 \coordinate (l) at (-1.2,0);
 \coordinate (r) at (1.2,0);
 \draw (u)--(d) (l)--(u)--(r) (l)--(d)--(r);
 \draw (l) .. controls (-1.2,2) and (1.2,2) .. (r);
 \fill(u) circle (1mm); \fill(l) circle (1mm); \fill(r) circle (1mm); \fill(d) circle (1mm);
\end{tikzpicture}
\]
on the $2$-dimensional sphere $S$. The tagged triangulation $T$ does not have different plain arcs connecting two common punctures. We add a puncture and arcs to a triangle of $T$ as follows:
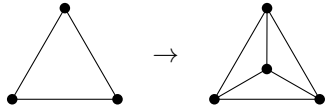
\[
\begin{tikzpicture}[baseline=5mm,scale=0.7]
 \coordinate (u) at (60:2);
 \coordinate (l) at (0,0);
 \coordinate (r) at (0:2);
 \draw (u)--(l)--(r)--(u);
 \fill(u) circle (1mm); \fill(l) circle (1mm); \fill(r) circle (1mm);
\end{tikzpicture}
 \hspace{3mm}\rightarrow\hspace{3mm}
\begin{tikzpicture}[baseline=5mm,scale=0.7]
 \coordinate (c) at (30:1.15);
 \coordinate (u) at (60:2);
 \coordinate (l) at (0,0);
 \coordinate (r) at (0:2);
 \draw (u)--(l)--(r)--(u) (u)--(c)--(r) (c)--(l);
 \fill(u) circle (1mm); \fill(l) circle (1mm); \fill(r) circle (1mm); \fill(c) circle (1mm);
\end{tikzpicture}
\]
 Then we have inductively a tagged triangulation without different plain arcs connecting two common punctures for any $p$. By Theorem \ref{ec}, it detects tagged arcs.

 We assume that $g>0$. By Theorem \ref{JR} and Proposition \ref{JRdetect}, if there is a tagged triangulation of $(S,M)$ detecting tagged arcs, then \eqref{ineq} holds. Conversely, if $p$ is the minimal integer to satisfy \eqref{ineq}, then there is a tagged triangulation $T$ of $(S,M)$ detecting tagged arcs. In the same way as the case of $g=0$, we have inductively a tagged triangulation without different plain arcs connecting two common punctures for any $p$ satisfying \eqref{ineq}. By Theorem \ref{ec}, it detects tagged arcs.
\end{proof}

\begin{example}
 When $g=1$, \eqref{ineq} means that $p\geq 7$. We consider the tagged triangulation
\[
 T=\ \ 
\begin{tikzpicture}[baseline=-15mm]
\node[circle,draw,inner sep=0.5](1) at (0,0){1}; 
\node[circle,draw,inner sep=0.5](2) at (0,-1){2}; 
\node[circle,draw,inner sep=0.5](3) at (0,-2){3}; 
\node[circle,draw,inner sep=0.5](1') at (0,-3){1}; 
\node[circle,draw,inner sep=0.5](4) at (1,0){4}; 
\node[circle,draw,inner sep=0.5](5) at (2,0){5}; 
\node[circle,draw,inner sep=0.5](1") at (3,0){1}; 
\node[circle,draw,inner sep=0.5](2') at (3,-1){2}; 
\node[circle,draw,inner sep=0.5](3') at (3,-2){3}; 
\node[circle,draw,inner sep=0.5](1''') at (3,-3){1}; 
\node[circle,draw,inner sep=0.5](4') at (1,-3){4}; 
\node[circle,draw,inner sep=0.5](5') at (2,-3){5}; 
\node[circle,draw,inner sep=0.5](6) at (1.5,-1){6}; 
\node[circle,draw,inner sep=0.5](7) at (1.5,-2){7};    
  \draw (1) to (2);
  \draw (1') to (3);
  \draw (1) to (4);
  \draw (1') to (4');
  \draw (1) to (6);
  \draw (2) to (3);
  \draw (2) to (4');
  \draw (2) to (6);
  \draw (2) to (7);
  \draw (3) to (4');
  \draw (4) to (5);
  \draw (4') to (5');
  \draw (4) to (6);
  \draw (4') to (7);
  \draw (5) to (6);
  \draw (5') to (7);
  \draw (6) to (7);
  \draw (1") to (5);
  \draw (2') to (5);
  \draw (3') to (5);
  \draw (3') to (6);
  \draw (3') to (7);
  \draw (1''') to (7);
  \draw (1''') to (5');
  \draw (1") to (2');
  \draw (2') to (3');
  \draw (3') to (1''');
\end{tikzpicture}
\]
on the torus $S$ with $7$ punctures, where we identify each of two vertical lines and two horizontal lines\footnote{This example is also related to coloring problems of closed surfaces. It is the example that proves that we need at least 7 colors to properly cover a graph on the torus.}. Then $T$ does not have different plain arcs connecting two common punctures. Thus $T$ detects tagged arcs by Theorem \ref{ec}.
\end{example}

\section{$f$-vectors in cluster algebras}\label{secfvec}

 In this section, we apply our results to the theory of cluster algebras.

\subsection{Cluster algebras and $f$-vectors}\label{clalg}

 We begin with recalling cluster algebras with coefficients associated with ice quivers \cite{FZ02,FZ07} (see also \cite{K}). For that, we need to prepare some notations. A {\it cluster quiver} is a finite quiver without loops and $2$-cycles. For positive integers $n \le m$, an {\it ice quiver of type $(n,m)$} is a cluster quiver $Q$ with vertices $Q_0=\{1,\ldots,m\}$ such that there are no arrows between vertices in $\{n+1,\ldots,m\}$ which are called {\it frozen vertices}. Let $\cF:={\mathbb Q}(t_1,\ldots,t_m)$ be the field of rational functions in $m$ variables over ${\mathbb Q}$.

\begin{definition}
 $(1)$ A {\it labeled seed} (or simply, {\it seed}) is a pair $(\mathbf{z},Q)$ consisting of the following data:
\begin{itemize}
 \item[(a)] $\mathbf{z}=(z_1,\ldots,z_n,y_1,\ldots,y_{m-n})$ is a free generating set of $\cF$ over $\mathbb{Q}$.
 \item[(b)] $Q$ is an ice quiver of type $(n,m)$.
\end{itemize}
 Then we refer to the $n$-tuple $(z_1,\ldots,z_n)$ as the {\it cluster}, to each $z_i$ as a {\it cluster variable} and $y_i$ as a {\it coefficient}.

 $(2)$ For a seed $(\mathbf{z},Q)$, the {\it mutation $\mu_k(\mathbf{z},Q)=(\mathbf{z'},Q')$ in direction $k$} $(1 \le k \le n)$ is defined as follows:
\begin{itemize}
 \item[(a)] $\mathbf{z'}=(z'_1,\ldots,z'_n,y_1,\ldots,y_{m-n})$ is defined by
 \begin{equation*}
  z_k z'_k = \prod_{(j \rightarrow k) \in Q_1}z_j\prod_{(j \rightarrow k) \in Q_1}y_{j-n}+\prod_{(j \leftarrow k) \in Q_1}z_j\prod_{(j \leftarrow k) \in Q_1}y_{j-n} \ \ \text{and} \ \ z'_i = z_i \ \ \text{if} \ \ i \neq k,
 \end{equation*}
where $z_{n+1}=\cdots=z_{m}=1=y_{1-n}=\cdots=y_0$ and $Q_1$ is the set of arrows in $Q$.
 \item[(b)] $Q'$ is the ice quiver obtained from $Q$ by the following steps:\par
 \begin{itemize}
  \item[(i)] For any path $i \rightarrow k \rightarrow j$, add an arrow $i \rightarrow j$.
  \item[(ii)] Reverse all arrows incident to $k$.
  \item[(iii)] Remove a maximal set of disjoint $2$-cycles.
  \item[(iv)] Remove all arrow connecting two frozen vertices.
 \end{itemize}
\end{itemize}
\end{definition}

We remark that $\mu_k$ is an involution, that is, we have $\mu_k\mu_k(\mathbf{z},Q)=(\mathbf{z},Q)$. Moreover, it is elementary that $\mu_k(\mathbf{z},Q)$ is also a seed.

 Now we define cluster algebras with coefficients associated with ice quivers. For an ice quiver $Q$ of type $(n,m)$, we fix a seed $(\mathbf{x}=(x_1,\ldots,x_n,y_1,\ldots,y_{m-n}),Q)$ which we call the {\it initial seed}.  We also call each $x_i$ an {\it initial cluster variable}.

\begin{definition}
 The {\it cluster algebra} $\cA(\mathbf{x},Q)$ {\it with coefficients} for the initial seed $(\mathbf{x},Q)$ is a $\mathbb{Z}$-subalgebra of ${\mathcal F}$ generated by the cluster variables and the coefficients obtained by all sequences of mutations from $(\mathbf{x},Q)$.
\end{definition}

 Next, we recall the definition of cluster algebras with principal coefficients \cite{FZ07}. Let $Q$ be an ice quiver of type $(n,n)$ with vertices $Q_0=\{1,\ldots,n\}$. The {\it framed quiver} associated with $Q$ is the ice quiver $\hat{Q}$ of type $(n,2n)$ which is obtained from $Q$ by adding frozen vertices $\{1',\ldots,n'\}$ and arrows $\{i \rightarrow i' \mid i \in Q_0\}$. Then $\cA(Q):=\cA((x_1,\ldots,x_n,y_1,\ldots,y_n),\hat{Q})$ is called a {\it cluster algebra with principal coefficients}.

 One of the remarkable properties of cluster algebras is the {\it strongly Laurent phenomenon} \cite[Proposition 3.6]{FZ07}: Every element of the cluster algebra $\cA(Q)$ with principal coefficients is a Laurent polynomial over $\bZ[y_1,\ldots,y_n]$ in initial cluster variables, that is, $\cA(Q) \subseteq \bZ[x_1^{\pm 1},\ldots,x_n^{\pm 1},y_1,\ldots,y_n]$. Then we denote the Laurent expression of a cluster variable $z$ of $\cA(Q)$ by $z(x_1,\ldots,x_n,y_1,\ldots,y_n)$. The {\it $F$-polynomial} of $z$ is the rational function $z(1,\ldots,1,y_1,\ldots,y_n)$, which is a polynomial by the strongly Laurent phenomenon. Let $f_{z,1},\ldots,f_{z,n}$ be the maximal degrees of $y_1,\ldots,y_n$ in $z(1,\ldots,1,y_1,\ldots,y_n)$, respectively. The {\it $f$-vector} of $z$ is the integer vector $f_z:=(f_{z,1},\ldots,f_{z,n}) \in \bZ_{\ge 0}^n$. For a cluster $\mathbf{z}=(z_1,\ldots,z_n)$ of $\cA(Q)$, the {\it $F$-matrix} of $\mathbf{z}$ is defined by the non-negative integer $n \times n$-matrix $F_\mathbf{z}$ with columns $f_{z_1},\ldots,f_{z_n}$ \cite[Definition 2.6]{FuG}.

\begin{example}\label{example1}
 Let $Q$ be a quiver $1 \leftarrow 2 \leftarrow 3$ of type $A_3$. We can compute the mutation of the initial seed $((x_1,x_2,x_3,y_1,y_2,y_3),\hat{Q})$ of $\cA(Q)$ in direction $1$ as follows:
\[
\mu_1\Biggl((x_1,x_2,x_3,y_1,y_2,y_3),\begin{tikzpicture}[baseline=3mm]
 \node (1) at(-1,0) {$1$}; \node (2) at(0,0) {$2$}; \node (3) at(1,0) {$3$};
 \node (1') at(-1,1) {$1'$}; \node (2') at(0,1) {$2'$}; \node (3') at(1,1) {$3'$};
 \draw[<-] (1)--(2); \draw[<-] (2)--(3); \draw[->] (1)--(1'); \draw[->] (2)--(2'); \draw[->] (3)--(3');
\end{tikzpicture}\Biggr) = \Biggl(\biggl(\cfrac{y_1+x_2}{x_1},x_2,x_3,y_1,y_2,y_3\biggr),\begin{tikzpicture}[baseline=3mm]
 \node (1) at(-1,0) {$1$}; \node (2) at(0,0) {$2$}; \node (3) at(1,0) {$3$};
 \node (1') at(-1,1) {$1'$}; \node (2') at(0,1) {$2'$}; \node (3') at(1,1) {$3'$};
 \draw[->] (1)--(2); \draw[<-] (2)--(3); \draw[<-] (1)--(1'); \draw[->] (2)--(2'); \draw[->] (3)--(3'); \draw[->] (2)--(1');
\end{tikzpicture}\Biggr)
\]
 Repeating mutations, we get all the cluster variables as in Table \ref{examp}. Therefore, the cluster algebra is
{\setlength\arraycolsep{0mm}
\begin{eqnarray*}
 \cA(Q) =\bZ\biggl[&x_1,x_2,x_3,\cfrac{y_1+x_2}{x_1},\cfrac{y_2x_1+x_3}{x_2},\cfrac{1+y_3x_2}{x_3},\cfrac{y_1y_2x_1+y_1x_3+x_2x_3}{x_1x_2},\\
&\cfrac{y_2x_1+x_3+y_2y_3x_1x_2}{x_2x_3},\cfrac{y_1y_2x_1+y_1x_3+y_1y_2y_3x_1x_2+x_2x_3}{x_1x_2x_3}\biggr].
\end{eqnarray*}}
 The $F$ -polynomial of a cluster variable
\[
 z = \cfrac{y_1y_2x_1+y_1x_3+y_1y_2y_3x_1x_2+x_2x_3}{x_1x_2x_3}
\]
is $y_1y_2+y_1+y_1y_2y_3+1$, thus we have $f_z = (1,1,1)$. All $f$-vectors appear in Table \ref{examp}.
\end{example}

\begin{table}[ht]
\begin{tabular}{c|c|c}
 Cluster variable $z$ & $f_z=(f_{z,1},f_{z,2},f_{z,3})$ & Tagged arcs $\de$ such that $z_{\de}=z$
\\\hline
\end{tabular}
\vspace{2mm}\\
\begin{minipage}{0.3\hsize}
\begin{center}\begin{tabular}{c|c|c}
 $x_1$ & $(0,0,0)$ &
\begin{tikzpicture}[baseline=-1mm,scale=0.5]
 \coordinate (u) at(90:1); \coordinate (lu) at(150:1); \coordinate (ld) at(-150:1);
 \coordinate (ru) at(30:1); \coordinate (rd) at(-30:1); \coordinate (d) at(-90:1);
 \draw (u)--(lu)--(ld)--(d)--(rd)--(ru)--(u); \node at(0,1.3) {}; \node at(0,-1.3) {}; \draw[blue,thick] (u)--(ld);
\end{tikzpicture}
\\\hline
 $x_2$ & $(0,0,0)$ &
\begin{tikzpicture}[baseline=-1mm,scale=0.5]
 \coordinate (u) at(90:1); \coordinate (lu) at(150:1); \coordinate (ld) at(-150:1);
 \coordinate (ru) at(30:1); \coordinate (rd) at(-30:1); \coordinate (d) at(-90:1);
 \draw (u)--(lu)--(ld)--(d)--(rd)--(ru)--(u); \node at(0,1.3) {}; \node at(0,-1.3) {}; \draw[blue,thick] (u)--(d);
\end{tikzpicture}
\\\hline
 $x_3$ & $(0,0,0)$ &
\begin{tikzpicture}[baseline=-1mm,scale=0.5]
 \coordinate (u) at(90:1); \coordinate (lu) at(150:1); \coordinate (ld) at(-150:1);
 \coordinate (ru) at(30:1); \coordinate (rd) at(-30:1); \coordinate (d) at(-90:1);
 \draw (u)--(lu)--(ld)--(d)--(rd)--(ru)--(u); \node at(0,1.3) {}; \node at(0,-1.3) {}; \draw[blue,thick] (u)--(rd);
\end{tikzpicture}
\end{tabular}\end{center}
\end{minipage}
\begin{minipage}{0.3\hsize}
\begin{center}\begin{tabular}{c|c|c}
 $\cfrac{y_1+x_2}{x_1}$ & $(1,0,0)$ &
\begin{tikzpicture}[baseline=-1mm,scale=0.5]
 \coordinate (u) at(90:1); \coordinate (lu) at(150:1); \coordinate (ld) at(-150:1);
 \coordinate (ru) at(30:1); \coordinate (rd) at(-30:1); \coordinate (d) at(-90:1);
 \draw (u)--(lu)--(ld)--(d)--(rd)--(ru)--(u); \node at(0,1.3) {}; \node at(0,-1.3) {}; \draw[blue,thick] (lu)--(d);
\end{tikzpicture}
\\\hline
 $\cfrac{y_2x_1+x_3}{x_2}$ & $(0,1,0)$ &
\begin{tikzpicture}[baseline=-1mm,scale=0.5]
 \coordinate (u) at(90:1); \coordinate (lu) at(150:1); \coordinate (ld) at(-150:1);
 \coordinate (ru) at(30:1); \coordinate (rd) at(-30:1); \coordinate (d) at(-90:1);
 \draw (u)--(lu)--(ld)--(d)--(rd)--(ru)--(u); \node at(0,1.3) {}; \node at(0,-1.3) {}; \draw[blue,thick] (ld)--(rd);
\end{tikzpicture}
\\\hline
 $\cfrac{1+y_3x_2}{x_3}$ & $(0,0,1)$ &
\begin{tikzpicture}[baseline=-1mm,scale=0.5]
 \coordinate (u) at(90:1); \coordinate (lu) at(150:1); \coordinate (ld) at(-150:1);
 \coordinate (ru) at(30:1); \coordinate (rd) at(-30:1); \coordinate (d) at(-90:1);
 \draw (u)--(lu)--(ld)--(d)--(rd)--(ru)--(u); \node at(0,1.3) {}; \node at(0,-1.3) {}; \draw[blue,thick] (ru)--(d);
\end{tikzpicture}
\end{tabular}\end{center}
\end{minipage}
\begin{tabular}{c|c|c}
\hline
 $\cfrac{y_1y_2x_1+y_1x_3+x_2x_3}{x_1x_2}$ & $(1,1,0)$ &
\begin{tikzpicture}[baseline=-1mm,scale=0.5]
 \coordinate (u) at(90:1); \coordinate (lu) at(150:1); \coordinate (ld) at(-150:1);
 \coordinate (ru) at(30:1); \coordinate (rd) at(-30:1); \coordinate (d) at(-90:1);
 \draw (u)--(lu)--(ld)--(d)--(rd)--(ru)--(u); \node at(0,1.3) {}; \node at(0,-1.3) {}; \draw[blue,thick] (lu)--(rd);
\end{tikzpicture}
\\\hline
 $\cfrac{y_2x_1+x_3+y_2y_3x_1x_2}{x_2x_3}$ & $(0,1,1)$ &
\begin{tikzpicture}[baseline=-1mm,scale=0.5]
 \coordinate (u) at(90:1); \coordinate (lu) at(150:1); \coordinate (ld) at(-150:1);
 \coordinate (ru) at(30:1); \coordinate (rd) at(-30:1); \coordinate (d) at(-90:1);
 \draw (u)--(lu)--(ld)--(d)--(rd)--(ru)--(u); \node at(0,1.3) {}; \node at(0,-1.3) {}; \draw[blue,thick] (ru)--(ld);
\end{tikzpicture}
\\\hline
 $\cfrac{y_1y_2x_1+y_1x_3+y_1y_2y_3x_1x_2+x_2x_3}{x_1x_2x_3}$ & $(1,1,1)$ &
\begin{tikzpicture}[baseline=-1mm,scale=0.5]
 \coordinate (u) at(90:1); \coordinate (lu) at(150:1); \coordinate (ld) at(-150:1);
 \coordinate (ru) at(30:1); \coordinate (rd) at(-30:1); \coordinate (d) at(-90:1);
 \draw (u)--(lu)--(ld)--(d)--(rd)--(ru)--(u); \node at(0,1.3) {}; \node at(0,-1.3) {}; \draw[blue,thick] (lu)--(ru);
\end{tikzpicture}
\\\hline
  \end{tabular}
 \vspace{5mm}
 \caption{In $\cA(Q)$ for a quiver $Q$ of type $A_3$, all the $9$ cluster variables, the corresponding $f$-vectors and tagged arcs}\label{examp}
\end{table}

 In Table \ref{examp}, different non-initial cluster variables have different $f$-vectors. In general, it is not true (see Proposition \ref{ec2}). However, we conjecture that different clusters have different $F$-matrices.

\begin{conjecture}\label{conjF}
 Let ${\cA}$ be an arbitrary cluster algebra with principal coefficients. If clusters $\mathbf{z}$ and $\mathbf{z'}$ in $\cA$ have $F_{\mathbf{z}}=F_{\mathbf{z'}}$, then $\mathbf{z}=\mathbf{z'}$.
\end{conjecture}

\begin{remark}
 In this paper, for simplicity, we define $f$-vectors only for cluster algebras with principal coefficients. In fact, we can also define $f$-vectors of cluster variables for cluster algebras with any coefficients by recurrences of \cite[(2.18),(2.19)]{FuG} (it is well-defined by \cite[Theorem 4.6]{CKLP}). However, when we consider the general coefficients version of Conjecture \ref{conjF}, it suffices to show this conjecture only in the principal coefficients case because $f$-vectors defined by \cite[(2.18),(2.19)]{FuG} depends only on quiver $Q$, and not on coefficients.  
\end{remark}
 In the next subsection, we prove Conjecture \ref{conjF} for the cluster algebra with principal coefficients defined from each tagged triangulation of $(S,M)$.

\subsection{Applying for cluster algebras defined from tagged triangulations}\label{appli}

 For a tagged triangulation $T$ satisfying $(\Diamond)$, we construct a cluster quiver $Q_T$ whose vertices are arcs of $T$ and whose arrows are obtained as in Figure \ref{QT} for puzzle pieces of $T$ or $T=T_3$. For any puncture $p$ of $(S,M)$, we define $Q_{T^{(p)}}=Q_T$. Thus we have the associated cluster quiver $Q_T$ for any tagged  triangulation $T$.
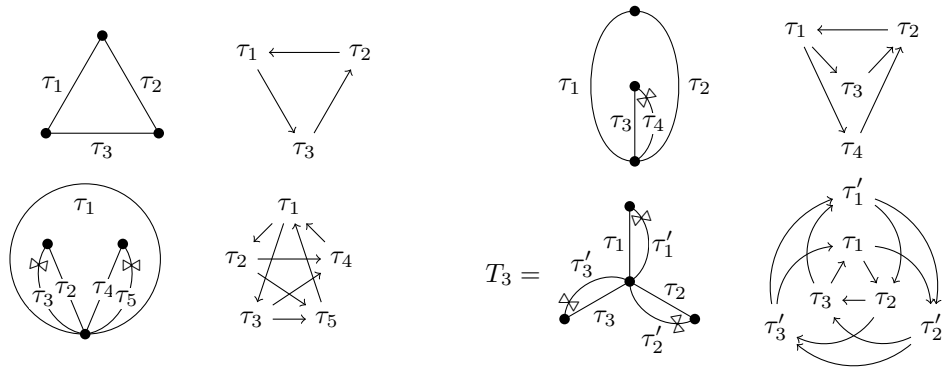
\begin{figure}[ht]
\begin{tikzpicture}
 \coordinate (0) at (0,0);
 \coordinate (1) at (120:1.5);
 \coordinate (2) at (180:1.5);
 \draw (0) to node[right]{$\tau_2$} (1);
 \draw (1) to node[left]{$\tau_1$} (2);
 \draw (0) to node[below]{$\tau_3$} (2);
 \fill(0) circle (0.7mm); \fill (1) circle (0.7mm); \fill (2) circle (0.7mm);
\end{tikzpicture}
   \hspace{5mm}
\begin{tikzpicture}
 \node (3) at (0,0) {$\tau_3$};
 \node (2) at (60:1.5) {$\tau_2$};
 \node (1) at (120:1.5) {$\tau_1$};
 \draw[->] (3) -- (2); \draw[->] (2) -- (1); \draw[->] (1) -- (3);
\end{tikzpicture}
   \hspace{20mm}
\begin{tikzpicture}
 \coordinate (0) at (0,0);
 \coordinate (1) at (0,-1);
 \coordinate (2) at (0,-2);
 \draw (2) to [out=180,in=180] node[left]{$\tau_1$} (0);
 \draw (2) to [out=0,in=0] node[right]{$\tau_2$} (0);
 \draw (1) to node[left=-3]{$\tau_3$} (2);
 \draw (1) to [out=60,in=120,relative] node[pos=0.2]{\rotatebox{40}{\footnotesize $\bowtie$}} node[fill=white,inner sep=1]{$\tau_4$} (2);
 \fill(0) circle (0.7mm); \fill (1) circle (0.7mm); \fill (2) circle (0.7mm);
\end{tikzpicture}
   \hspace{5mm}
\begin{tikzpicture}
 \node (3) at (0,0.5) {$\tau_3$};
 \node (4) at (0,-0.3) {$\tau_4$};
 \node (2) at (60:1.5) {$\tau_2$};
 \node (1) at (120:1.5) {$\tau_1$};
 \draw[->] (3) -- (2); \draw[->] (2) -- (1); \draw[->] (1) -- (3);
 \draw[->] (4) -- (2); \draw[->] (1) -- (4);
\end{tikzpicture}
\\
\begin{tikzpicture}[baseline=-13mm]
 \coordinate (0) at (0,0); \node at (0,-0.3) {$\tau_1$};
 \coordinate (1) at (-0.5,-0.8); \fill (1) circle (0.7mm);
 \coordinate (1') at (0.5,-0.8); \fill (1') circle (0.7mm);
 \coordinate (2) at (0,-2); \fill (2) circle (0.7mm);
 \draw (0,-1) circle (1);
 \draw (1) to node[fill=white,inner sep=1]{$\tau_2$} (2);
 \draw (1) to [out=-60,in=-120,relative] node[pos=0.2]{\rotatebox{170}{\footnotesize $\bowtie$}} node[fill=white,inner sep=1]{$\tau_3$} (2);
 \draw (1') to node[fill=white,inner sep=1]{$\tau_4$} (2);
 \draw (1') to [out=60,in=120,relative] node[pos=0.2]{\rotatebox{180}{\footnotesize $\bowtie$}} node[fill=white,inner sep=1]{$\tau_5$} (2);
\end{tikzpicture}
   \hspace{5mm}
\begin{tikzpicture}[baseline=-10mm]
 \node (1) at (0,0) {$\tau_1$};
 \node (2) at (-0.7,-0.7) {$\tau_2$};
 \node (3) at (-0.5,-1.5) {$\tau_3$};
 \node (4) at (0.7,-0.7) {$\tau_4$};
 \node (5) at (0.5,-1.5) {$\tau_5$};
 \draw[->] (1) -- (2); \draw[->] (2) -- (4); \draw[->] (4) -- (1);
 \draw[->] (1) -- (3); \draw[->] (3) -- (4);
 \draw[->] (2) -- (5); \draw[->] (5) -- (1); \draw[->] (3) -- (5);
\end{tikzpicture}
   \hspace{14mm}
$T_3=$
\begin{tikzpicture}[baseline=0mm]
 \coordinate (0) at (0,0);
 \coordinate (u) at (90:1);
 \coordinate (r) at (-30:1);
 \coordinate (l) at (210:1);
 \draw (0) to node[left,fill=white,inner sep=1]{$\tau_1$} (u);
 \draw (0) to [out=-60,in=-120,relative] node[pos=0.8]{\rotatebox{20}{\footnotesize $\bowtie$}} node[right,fill=white,inner sep=1]{$\tau_1'$} (u);
 \draw (0) to node[right=5,above=-3]{$\tau_2$} (r);
 \draw (0) to [out=-60,in=-120,relative] node[pos=0.8]{\rotatebox{100}{\footnotesize $\bowtie$}} node[below]{$\tau_2'$} (r);
 \draw (0) to node[below=7,right=-5]{$\tau_3$} (l);
 \draw (0) to [out=-60,in=-120,relative] node[pos=0.8]{\rotatebox{-30}{\footnotesize $\bowtie$}} node[left=2,above=-1]{$\tau_3'$} (l);
 \fill (0) circle (0.7mm); \fill (u) circle (0.7mm); \fill (l) circle (0.7mm); \fill (r) circle (0.7mm);
\end{tikzpicture}
   \hspace{5mm}
\begin{tikzpicture}[baseline=0mm]
 \node (1) at (90:0.5) {$\tau_1$};
 \node (2) at (90:1.2) {$\tau_1'$};
 \node (3) at (-30:0.5) {$\tau_2$};
 \node (4) at (-30:1.2) {$\tau_2'$};
 \node (5) at (210:0.5) {$\tau_3$};
 \node (6) at (210:1.2) {$\tau_3'$};
 \draw[->] (1) -- (3); \draw[->] (3) -- (5); \draw[->] (5) -- (1);
 \draw[->] (2) to [out=40,in=140,relative] (4); \draw[->] (4) to [out=40,in=140,relative] (6); \draw[->] (6) to [out=40,in=140,relative] (2);
 \draw[->] (2) to [out=40,in=140,relative] (3); \draw[->] (4) to [out=40,in=140,relative] (5); \draw[->] (6) to [out=40,in=140,relative] (1);
 \draw[->] (1) to [out=40,in=140,relative] (4); \draw[->] (3) to [out=40,in=140,relative] (6); \draw[->] (5) to [out=40,in=140,relative] (2);
\end{tikzpicture}
\caption{Quivers corresponding to puzzle pieces and $T_3$}\label{QT}
\end{figure}
 Then we have a cluster algebra $\cA(T):=\cA(Q_T)$ for any tagged triangulation $T$ of $(S,M)$. This cluster algebra has the following properties.

\begin{theorem}\cite[Theorem 7.11]{FoST}\cite[Theorem 6.1]{FoT}\label{bijtc}
 Let $T$ be a tagged triangulation of $(S,M)$.
\begin{itemize}
\item[(1)]
 If $(S,M)$ is not $1$-punctured closed surface, the tagged arcs $\de$ of $(S,M)$ correspond bijectively with the cluster variables $z_{\de}$ in $\cA(T)$. This induces that the tagged triangulations $T'$ of $(S,M)$ correspond bijectively with the clusters $\mathbf{z}_{T'}$ in $\cA(T)$.
 \item[(2)]
 If $(S,M)$ is $1$-punctured closed surface, the plain arcs $\de$ of $(S,M)$ correspond bijectively with the cluster variables $z_{\de}$ in $\cA(T)$. This induces that the tagged triangulations $T'$ which consist of plain arcs $\de$ of $(S,M)$ correspond bijectively with the clusters $\mathbf{z}_{T'}$ in $\cA(T)$.
 \end{itemize} 
\end{theorem}

\begin{example}\label{example2}
 For a marked surface $(S,M)$ and a tagged triangulation $T$ of $(S,M)$ as follows:
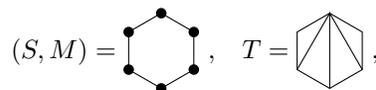
\[
 (S,M) = 
\begin{tikzpicture}[baseline=-1mm,scale=0.5]
 \coordinate (u) at(90:1); \coordinate (lu) at(150:1); \coordinate (ld) at(-150:1);
 \coordinate (ru) at(30:1); \coordinate (rd) at(-30:1); \coordinate (d) at(-90:1);
 \draw (u)--(lu)--(ld)--(d)--(rd)--(ru)--(u);
 \fill (u) circle (1.4mm); \fill (lu) circle (1.4mm); \fill (ld) circle (1.4mm); \fill (ru) circle (1.4mm); \fill (rd) circle (1.4mm); \fill (d) circle (1.4mm);
\end{tikzpicture}\ ,\hspace{3mm}
 T = 
\begin{tikzpicture}[baseline=-1mm,scale=0.5]
 \coordinate (u) at(90:1); \coordinate (lu) at(150:1); \coordinate (ld) at(-150:1);
 \coordinate (ru) at(30:1); \coordinate (rd) at(-30:1); \coordinate (d) at(-90:1);
 \draw (u)--(lu)--(ld)--(d)--(rd)--(ru)--(u); \draw (ld)--(u)--(d) (u)--(rd);
\end{tikzpicture}\ ,
\]
 $Q_T$ is a quiver $1 \leftarrow 2 \leftarrow 3$ of type $A_3$. The bijection between the set of tagged arcs of $(S,M)$ and the set of cluster variables in $\cA(T)$ is given in Table \ref{examp}. Figure \ref{hexagoncomplex} gives the tagged arc complex of $(S,M)$. In this case, three tagged arcs whose each pair is combined by an edge form a tagged triangulation.
 
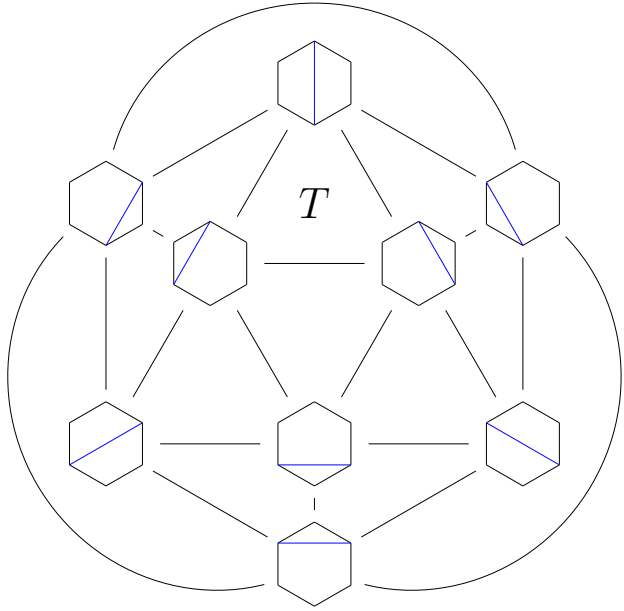
\begin{figure}[ht]
\begin{center}
\scalebox{0.8}{
\begin{tikzpicture}
\coordinate (0) at (0,0); 
\coordinate (u*) at (90:5.33); 
\coordinate (u) at (90:4);  
\coordinate (ul) at (150:4);   
\coordinate (ur) at (30:4);   
\coordinate (uml) at (150:2);   
\coordinate (umr) at (30:2);   
\coordinate (dmc) at (-90:2);  
\coordinate (dl) at (-150:4); 
 \coordinate (dl*) at (-150:5.33); 
\coordinate (dr) at (-30:4);  
\coordinate (dr*) at (-30:5.33); 
\coordinate (d) at (-90:4);  
\draw (u) to (ul);
\draw (u) to (ur);
\draw (ul) to (uml);
\draw (ur) to (umr);
\draw (umr) to (uml);
\draw (u) to (uml);
\draw (u) to (umr);
\draw (uml) to (dl);
\draw (umr) to (dr);
\draw (ul) to (dl);
\draw (ur) to (dr);
\draw (uml) to (dmc);
\draw (umr) to (dmc);
\draw (dl) to (dmc);
\draw (dr) to (dmc);
\draw (dl) to (d);
\draw (dmc) to (d);
\draw (dr) to (d);
\draw(ul) [out=90,in=180]to (u*);
\draw(ur) [out=90,in=0]to (u*);
\draw(ul) [out=210,in=120]to (dl*);
\draw(d) [out=210,in=300]to (dl*);
\draw(d) [out=330,in=240]to (dr*);
\draw(ur) [out=330,in=60]to (dr*);
\fill [white](u) circle (0.9cm);
\fill[white] (ul) circle (0.9cm);
\fill[white] (ur) circle (0.9cm);
\fill[white] (uml) circle (0.9cm);
\fill[white] (umr) circle (0.9cm);
\fill[white] (dmc) circle (0.9cm);
\fill[white] (dl) circle (0.9cm);
\fill[white] (dr) circle (0.9cm);
\fill[white] (d) circle (0.9cm);
\draw (u)++(210:0.7cm) to ++(90:0.7cm);
\draw (u)++(150:0.7cm) to ++(30:0.7cm);
\draw (u)++(90:0.7cm) to ++(330:0.7cm);
\draw (u)++(30:0.7cm) to ++(270:0.7cm);
\draw (u)++(330:0.7cm) to ++(210:0.7cm);
\draw (u)++(270:0.7cm) to ++(150:0.7cm);
\draw (ur)++(210:0.7cm) to ++(90:0.7cm);
\draw (ur)++(150:0.7cm) to ++(30:0.7cm);
\draw (ur)++(90:0.7cm) to ++(330:0.7cm);
\draw (ur)++(30:0.7cm) to ++(270:0.7cm);
\draw (ur)++(330:0.7cm) to ++(210:0.7cm);
\draw (ur)++(270:0.7cm) to ++(150:0.7cm);
\draw (ul)++(210:0.7cm) to ++(90:0.7cm);
\draw (ul)++(150:0.7cm) to ++(30:0.7cm);
\draw (ul)++(90:0.7cm) to ++(330:0.7cm);
\draw (ul)++(30:0.7cm) to ++(270:0.7cm);
\draw (ul)++(330:0.7cm) to ++(210:0.7cm);
\draw (ul)++(270:0.7cm) to ++(150:0.7cm);
\draw (umr)++(210:0.7cm) to ++(90:0.7cm);
\draw (umr)++(150:0.7cm) to ++(30:0.7cm);
\draw (umr)++(90:0.7cm) to ++(330:0.7cm);
\draw (umr)++(30:0.7cm) to ++(270:0.7cm);
\draw (umr)++(330:0.7cm) to ++(210:0.7cm);
\draw (umr)++(270:0.7cm) to ++(150:0.7cm);
\draw (uml)++(210:0.7cm) to ++(90:0.7cm);
\draw (uml)++(150:0.7cm) to ++(30:0.7cm);
\draw (uml)++(90:0.7cm) to ++(330:0.7cm);
\draw (uml)++(30:0.7cm) to ++(270:0.7cm);
\draw (uml)++(330:0.7cm) to ++(210:0.7cm);
\draw (uml)++(270:0.7cm) to ++(150:0.7cm);
\draw (dl)++(210:0.7cm) to ++(90:0.7cm);
\draw (dl)++(150:0.7cm) to ++(30:0.7cm);
\draw (dl)++(90:0.7cm) to ++(330:0.7cm);
\draw (dl)++(30:0.7cm) to ++(270:0.7cm);
\draw (dl)++(330:0.7cm) to ++(210:0.7cm);
\draw (dl)++(270:0.7cm) to ++(150:0.7cm);
\draw (dr)++(210:0.7cm) to ++(90:0.7cm);
\draw (dr)++(150:0.7cm) to ++(30:0.7cm);
\draw (dr)++(90:0.7cm) to ++(330:0.7cm);
\draw (dr)++(30:0.7cm) to ++(270:0.7cm);
\draw (dr)++(330:0.7cm) to ++(210:0.7cm);
\draw (dr)++(270:0.7cm) to ++(150:0.7cm);
\draw (dmc)++(210:0.7cm) to ++(90:0.7cm);
\draw (dmc)++(150:0.7cm) to ++(30:0.7cm);
\draw (dmc)++(90:0.7cm) to ++(330:0.7cm);
\draw (dmc)++(30:0.7cm) to ++(270:0.7cm);
\draw (dmc)++(330:0.7cm) to ++(210:0.7cm);
\draw (dmc)++(270:0.7cm) to ++(150:0.7cm);
\draw (d)++(210:0.7cm) to ++(90:0.7cm);
\draw (d)++(150:0.7cm) to ++(30:0.7cm);
\draw (d)++(90:0.7cm) to ++(330:0.7cm);
\draw (d)++(30:0.7cm) to ++(270:0.7cm);
\draw (d)++(330:0.7cm) to ++(210:0.7cm);
\draw (d)++(270:0.7cm) to ++(150:0.7cm);
\draw [blue](uml)++(90:0.7cm) to ++(240:1.21cm);
\draw [blue](umr)++(90:0.7cm) to ++(300:1.21cm);
\draw [blue](dmc)++(210:0.7cm) to ++(0:1.21cm);
\draw [blue](dl)++(210:0.7cm) to ++(30:1.39cm);
\draw [blue](dr)++(150:0.7cm) to ++(330:1.39cm);
\draw [blue](u)++(90:0.7cm) to ++(270:1.39cm);
\draw [blue](ul)++(270:0.7cm) to ++(60:1.21cm);
\draw [blue](d)++(150:0.7cm) to ++(0:1.21cm);
\draw [blue](ur)++(150:0.7cm) to ++(300:1.21cm);
\node (c) at (0,2) {{\Huge $T$}};
\end{tikzpicture}}
\end{center}
\caption{Triangulations of hexagon\label{hexagoncomplex}}
\end{figure}
\end{example}

\begin{theorem}\label{f=Int}\cite[Theorem 1.8]{Y}
 Let $T$ be a tagged triangulation of $(S,M)$. If $(S,M)$ is a $1$-punctured closed surface, for any plain arc $\de$ of $(S,M)$, we have $f_{z_{\de}}=\Int(T,\de)$. If not, for any tagged arc $\de$ of $(S,M)$, we have $f_{z_{\de}}=\Int(T,\de)$.
\end{theorem}

 Thanks to Theorem \ref{f=Int}, we can apply the results in the previous sections to the theory of cluster algebras and get the main result in this paper.

\begin{corollary}\label{Funique}
 Let $T$ be a tagged triangulation of $(S,M)$. If tagged triangulations $T'$ and $T''$ of $(S,M)$ satisfy $F_{\mathbf{z}_{T'}}=F_{\mathbf{z}_{T''}}$, then $\mathbf{z}_{T'}=\mathbf{z}_{T''}$.
\end{corollary}

\begin{proof}
 The assertion follows immediately from Theorems \ref{main} and \ref{f=Int}. 
\end{proof}

\begin{definition}
 For a cluster algebra $\cA$, we say that $\cA$ {\it detects cluster variables by $f$-vectors} if it satisfies the following condition:\par
 $\bullet$ For non-initial cluster variables $z$ and $z'$ of $\cA$, if $f_z=f_{z'}$, then $z=z'$.
\end{definition}

\begin{proposition}\label{ec2}
 Let $T$ be a tagged triangulation of $(S,M)$. Then $T$ detects cluster variables by $f$-vectors if and only if either of the following conditions holds:
\begin{itemize}
 \item $(S,M)$ is a $1$-punctured closed surface;
 \item there are no tagged arcs $\de$ and $\epsilon$ of $T$ connecting two (possibly same) common punctures such that $\overline{\de} \neq \overline{\epsilon}$.
\end{itemize}
\end{proposition}

\begin{proof}
 If $(S,M)$ is not a $1$-punctured closed surface, the  assertion follows from Theorems \ref{ec}, \ref{bijtc} and \ref{f=Int}. If $(S,M)$ is a $1$-punctured closed surface, there are no $2$-notched arcs corresponding to cluster variables by Theorem \ref{bijtc}(2). Therefore, the  assertion follows from Corollary \ref{taginj} and Theorem \ref{f=Int}.
\end{proof}

\begin{corollary}
 $(1)$ If $S$ is not closed, then there is at least one tagged triangulation of $(S,M)$ detecting cluster variables by $f$-vectors.

 $(2)$ If $S$ is closed, then there is at least one tagged triangulation of $(S,M)$ detecting cluster variables by $f$-vectors if and only if $(S,M)$ is a $1$-punctured closed surface or the inequality \eqref{ineq} holds.

 $(3)$ All tagged triangulation of $(S,M)$ detect cluster variables by $f$-vectors if and only if $(S,M)$ is one of the followings:
\begin{itemize}
 \item a $1$-punctured closed surface;
 \item a marked surface with no punctures;
 \item a marked surface of genus $0$ with exactly $1$ boundary component and at most $2$ punctures;
 \item a marked surface of genus $0$ with exactly $2$ boundary components and a $1$ puncture.
\end{itemize}
\end{corollary}

\begin{proof}
 The assertion follows immediately from Theorem \ref{class} and Proposition \ref{ec2}.
\end{proof}

\section{List of segments in each puzzle piece}\label{segments}

 In this section, we prepare some tables to show Theorem \ref{intinj}.

 We fix a tagged triangulation $T$ of $(S,M)$ satisfying $(\Diamond)$ which is not $T_3$. Let $\square$ be a puzzle piece of $T$. If $\square$ is a triangle piece or a $1$-puncture piece, we say that $\square$ is Case($-$) (resp., Case($\tau_i$)) if its edges are not loops (resp., its edge $\tau_i$ is an only loop in $\square$). Similarly, we can define Case($\tau_i$,$\tau_{i+1}$) if $\square$ is a triangle piece (see Figures \ref{Ctri} and \ref{C1}).
\begin{figure}[ht]
\begin{minipage}{0.48\hsize}
\[
\begin{tikzpicture}[baseline=-3mm]
 \coordinate (0) at (0,0);
 \coordinate (1) at (120:1.5);
 \coordinate (2) at (180:1.5);
 \draw (0) to (1);
 \draw (1) to (2);
 \draw (0) to (2);
 \fill(0) circle (0.7mm); \fill (1) circle (0.7mm); \fill (2) circle (0.7mm);
\end{tikzpicture}
     \hspace{5mm}
\begin{tikzpicture}
 \coordinate (0) at (0,0); 
 \coordinate (1) at (0,1);
 \coordinate (2) at (0,-1);
 \draw (0) circle (1);
 \draw (2) .. controls (0.5,0) and (0.3,0.2) .. (0,0.2);
 \draw (2) .. controls (-0.5,0) and (-0.3,0.2) .. (0,0.2);
 \fill(1) circle (0.7mm); \fill(2) circle (0.7mm);
 \node at(0,0.35) {$\tau_i$};
\end{tikzpicture}
     \hspace{5mm}
\begin{tikzpicture}
 \coordinate (0) at (0,0);
 \coordinate (1) at (-0.5,0.1);
 \coordinate (1') at (0.5,0.1);
 \coordinate (2) at (0,-1);
 \draw (0) circle (1);
 \draw (2) .. controls (-1,-0.2) and (-0.7,0.1) .. (1);
 \draw (2) .. controls (0,-0.2) and (-0.3,0.1) .. (1);
 \draw (2) .. controls (1,-0.2) and (0.7,0.1) .. (1');
 \draw (2) .. controls (0,-0.2) and (0.3,0.1) .. (1');
 \fill(2) circle (0.7mm);
 \node at(-0.4,0.3) {$\tau_{i+1}$};
 \node at(0.4,0.3) {$\tau_i$};
\end{tikzpicture}
\]
\begin{center}
 \hspace{1mm} Case($-$) \hspace{8mm} Case($\tau_i$) \hspace{8mm} Case($\tau_i$,$\tau_{i+1}$)
\end{center}
 \caption{Cases of a triangle piece}\label{Ctri}
\end{minipage}
\begin{minipage}{0.5\hsize}
\[
\begin{tikzpicture}
 \coordinate (0) at (0,0);
 \coordinate (1) at (0,-1);
 \coordinate (2) at (0,-2);
 \draw (2) to [out=180,in=180] (0);
 \draw (2) to [out=0,in=0] (0);
 \draw (1) to (2);
 \draw (1) to [out=60,in=120,relative] node[pos=0.2]{\rotatebox{40}{\footnotesize $\bowtie$}} (2);
 \fill(0) circle (0.7mm); \fill (1) circle (0.7mm); \fill (2) circle (0.7mm);
\end{tikzpicture}
     \hspace{5mm}
\begin{tikzpicture}
 \coordinate (0) at (0,0);
 \coordinate (1) at (-0.5,0.1);
 \coordinate (1') at (0.5,0.2); \fill (1') circle (0.7mm);
 \coordinate (2) at (0,-1);
 \draw (0) circle (1);
 \draw (2) .. controls (-1,-0.2) and (-0.7,0.1) .. (1);
 \draw (2) .. controls (0,-0.2) and (-0.3,0.1) .. (1);
 \draw (1') to (2);
 \draw (1') to [out=60,in=120,relative] node[pos=0.2]{\rotatebox{180}{\footnotesize $\bowtie$}} (2);
 \fill(2) circle (0.7mm);
 \node at(-0.4,0.3) {$\tau_1$};
\end{tikzpicture}
     \hspace{5mm}
\begin{tikzpicture}
 \coordinate (0) at (0,0);
 \coordinate (1) at (-0.5,0.2); \fill (1) circle (0.7mm);
 \coordinate (1') at (0.5,0.1);
 \coordinate (2) at (0,-1);
 \draw (0) circle (1);
 \draw (1) to (2);
 \draw (1) to [out=-60,in=-120,relative] node[pos=0.2]{\rotatebox{170}{\footnotesize $\bowtie$}} (2);
 \draw (2) .. controls (1,-0.2) and (0.7,0.1) .. (1');
 \draw (2) .. controls (0,-0.2) and (0.3,0.1) .. (1');
 \fill(2) circle (0.7mm);
 \node at(0.4,0.3) {$\tau_2$};
\end{tikzpicture}
\]
\begin{center}
 Case($-$) \hspace{7mm} Case($\tau_1$) \hspace{9mm} Case($\tau_2$)
\end{center}
 \caption{Cases of a $1$-puncture piece}\label{C1}
\end{minipage}
\end{figure}
 Let $\de$ be a tagged arc of $(S,M)$ which is not contained in $T$. We have the set of curves $\de \cap \square$ and call its each curve a {\it segment (of $\de$) in $\square$}. It is easy to show that Table \ref{S0} (resp., Table \ref{S1}, Table \ref{S2}) gives a complete list of segments of $\de$ in a triangle piece (resp., a $1$-puncture piece, a $2$-puncture piece), where $a_i$ is the intersection number of each segment and $\tau_i$. Moreover, we have the set of `curves' ${\sf M}_T(\de) \cap \square$ and call its each curve a {\it modified segment (of $\de$) in $\square$}. Let $m$ be a modified segment in $\square$ which is not a segment. If there are two distinct segments $s$ and $s'$ in $\square$ such that ${\sf M}_T(s)={\sf M}_T(s')=m$, then $m$ is one as in Figure \ref{ys}. Otherwise, there is exactly one segment $s$ in $\square$ such that ${\sf M}_T(s)=m$. In this case, abusing notation, we denote ${\sf M}_T(s)$ by $s$. In particular, Table \ref{sm} gives all segments $s$ in $\square$ such that $s \neq {\sf M}_T(s)$.

 On the other hand, it is also easy to show that Table \ref{ss0} (resp., Table \ref{ss1}, Table \ref{ss2}) gives a complete list of $\de \cap \square$ and ${\sf M}_T(\de) \cap \square$, where $\square$ is a triangle piece (resp., a $1$-puncture piece, a $2$-puncture piece). Note that if an end which is not in $\square$ is tagged notched at a vertex of $\square$, it does not appear in $\de \cap \square$, but appear in ${\sf M}_T(\de) \cap \square$. So we identify its end to the corresponding modified segment as in Figure \ref{out} in Tables \ref{ss0}, \ref{ss1} and \ref{ss2}. For example, in the five line from the top of Table \ref{ss0} Case ($-$), the segment $e_1$ is given by this identification.

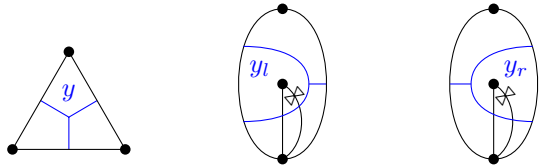
\begin{figure}[ht]
\begin{tikzpicture}[baseline=-15mm]
 \coordinate (0) at (0,0);
 \coordinate (1) at (-60:1.5);
 \coordinate (2) at (-120:1.5);
 \coordinate (o) at (-90:0.87);
 \draw[blue] (-120:0.75)--(o)--(-60:0.75) (o)--(-90:1.3);
 \draw (0)--(1)--(2)--(0);
 \fill(0) circle (0.7mm); \fill (1) circle (0.7mm); \fill (2) circle (0.7mm);
 \node[blue] at(0,-0.55) {$y$};
\end{tikzpicture}
   \hspace{10mm}
\begin{tikzpicture}
 \coordinate (0) at (0,0);
 \coordinate (1) at (0,-1);
 \coordinate (2) at (0,-2);
 \coordinate (3) at (-0.52,-0.5);
 \coordinate (5) at (-0.52,-1.5);
 \draw (2) to [out=180,in=180] (0);
 \draw (2) to [out=0,in=0] (0);
 \draw (1) to (2);
 \draw (1) to [out=60,in=120,relative] node[pos=0.2]{\rotatebox{40}{\footnotesize $\bowtie$}} (2);
 \draw[blue] (3) to [out=0,in=90] (0.35,-1);
 \draw[blue] (0.35,-1) to [out=-90,in=0] (5);
 \draw[blue] (0.35,-1) to (0.58,-1);
 \fill(0) circle (0.7mm); \fill (1) circle (0.7mm); \fill (2) circle (0.7mm);
 \node[blue] at(-0.3,-0.8) {$y_l$};
\end{tikzpicture}
   \hspace{10mm}
\begin{tikzpicture}
 \coordinate (0) at (0,0);
 \coordinate (1) at (0,-1);
 \coordinate (2) at (0,-2);
 \coordinate (4) at (0.52,-0.5);
 \coordinate (6) at (0.52,-1.5);;
 \draw (2) to [out=180,in=180] (0);
 \draw (2) to [out=0,in=0] (0);
 \draw (1) to (2);
 \draw (1) to [out=60,in=120,relative] node[pos=0.2]{\rotatebox{40}{\footnotesize $\bowtie$}} (2);
 \draw[blue] (4) to [out=180,in=90] (-0.3,-1);
 \draw[blue] (-0.3,-1) to [out=-90,in=180] (6);
 \draw[blue] (-0.3,-1) to (-0.58,-1);
 \fill(0) circle (0.7mm); \fill (1) circle (0.7mm); \fill (2) circle (0.7mm);
 \node[blue] at(0.3,-0.8) {$y_r$};
\end{tikzpicture}
\caption{Exceptional segments of modified tagged arcs}\label{ys}
\end{figure}

\begin{figure}[ht]
\[
\begin{tikzpicture}[baseline=5mm]
 \coordinate (0) at (0,0);
 \coordinate (1) at (120:1.5);
 \coordinate (2) at (180:1.5);
 \draw (0) to (1);
 \draw (1) to (2);
 \draw (0) to (2);  
 \draw[blue] (1) to node[pos=0.3]{\rotatebox{90}{\footnotesize $\bowtie$}} (0,1.3);
 \fill(0) circle (0.7mm); \fill (1) circle (0.7mm); \fill (2) circle (0.7mm);
\end{tikzpicture}
     \rightarrow
\begin{tikzpicture}[baseline=5mm]
 \coordinate (0) at (0,0);
 \coordinate (1) at (120:1.5);
 \coordinate (2) at (180:1.5);
  \coordinate (6) at (150:1.3);   
  \coordinate (7) at (120:0.75);  
 \draw (0) to (1);
 \draw (1) to (2);
 \draw (0) to (2);  
 \draw [blue] (6) to (7);
 \fill(0) circle (0.7mm); \fill (1) circle (0.7mm); \fill (2) circle (0.7mm);
\end{tikzpicture}
   \hspace{7mm}
\begin{tikzpicture}[scale=0.7,baseline=-10mm]
 \coordinate (0) at (0,0);
 \coordinate (1) at (0,-1);
 \coordinate (2) at (0,-2);
 \draw (2) to [out=180,in=180] (0);
 \draw (2) to [out=0,in=0] (0);
 \draw (1) to (2);
 \draw (1) to [out=60,in=120,relative] node[pos=0.2]{\rotatebox{40}{\footnotesize $\bowtie$}} (2);
 \draw[blue] (2) to node[pos=0.3]{\rotatebox{60}{\footnotesize $\bowtie$}} (1,-2.5);
 \fill(0) circle (0.7mm); \fill (1) circle (0.7mm); \fill (2) circle (0.7mm);
\end{tikzpicture}
     \rightarrow
\begin{tikzpicture}[scale=0.8,baseline=-10mm]
 \coordinate (0) at (0,0);
 \coordinate (1) at (0,-1);
 \coordinate (2) at (0,-2);
 \coordinate (3) at (0.52,-1.5);
 \coordinate (5) at (-0.52,-1.5);
 \draw (2) to [out=180,in=180] (0);
 \draw (2) to [out=0,in=0] (0);
 \draw (1) to (2);
 \draw (1) to [out=60,in=120,relative] node[pos=0.2]{\rotatebox{40}{\footnotesize $\bowtie$}} (2);
 \draw[blue] (3) to (5);
 \fill(0) circle (0.7mm); \fill (1) circle (0.7mm); \fill (2) circle (0.7mm);
\end{tikzpicture}
   \hspace{7mm}
\begin{tikzpicture}[scale=0.7,baseline=-8mm]
 \coordinate (0) at (0,0);
 \coordinate (1) at (0,-1);
 \coordinate (2) at (0,-2);
 \draw (2) to [out=180,in=180] (0);
 \draw (2) to [out=0,in=0] (0);
 \draw (1) to (2);
 \draw (1) to [out=60,in=120,relative] node[pos=0.2]{\rotatebox{40}{\footnotesize $\bowtie$}} (2);
 \draw[blue] (0) to node[pos=0.3]{\rotatebox{-60}{\footnotesize $\bowtie$}} (1,0.5);
 \fill(0) circle (0.7mm); \fill (1) circle (0.7mm); \fill (2) circle (0.7mm);
\end{tikzpicture}
     \rightarrow
\begin{tikzpicture}[scale=0.8,baseline=-10mm]
 \coordinate (0) at (0,0);
 \coordinate (1) at (0,-1);
 \coordinate (2) at (0,-2);
 \coordinate (3) at (0.52,-0.5);
 \coordinate (5) at (-0.52,-0.5);
 \draw (2) to [out=180,in=180] (0);
 \draw (2) to [out=0,in=0] (0);
 \draw (1) to (2);
 \draw (1) to [out=60,in=120,relative] node[pos=0.2]{\rotatebox{40}{\footnotesize $\bowtie$}} (2);
 \draw[blue] (3) to (5);
 \fill(0) circle (0.7mm); \fill (1) circle (0.7mm); \fill (2) circle (0.7mm);
\end{tikzpicture}
   \hspace{7mm}
\begin{tikzpicture}[baseline=-3mm,scale=0.7]
 \coordinate (0) at (0,0);
 \coordinate (1) at (-0.4,0);
 \coordinate (1') at (0.4,0);
 \coordinate (d) at (0,-1);
 \draw (0) circle (1);
 \draw[blue] (d) to node[pos=0.35]{\rotatebox{60}{\footnotesize $\bowtie$}} (1,-1.5);
 \fill (1) circle (1mm); \fill (1') circle (1mm); \fill (d) circle (1mm);
\end{tikzpicture}
     \rightarrow
\begin{tikzpicture}[baseline=-2mm,scale=0.7]
 \coordinate (0) at (0,0);
 \coordinate (1) at (-0.4,0);
 \coordinate (1') at (0.4,0);
 \coordinate (d) at (0,-1);
 \coordinate (l) at (-0.9,-0.4);
 \coordinate (r) at (0.9,-0.4);
 \draw (0) circle (1);
 \draw[blue] (l) to (r);
 \fill (1) circle (1mm); \fill (1') circle (1mm); \fill (d) circle (1mm);
\end{tikzpicture}
\]
\caption{Identifications in $\square$}\label{out}
\end{figure}
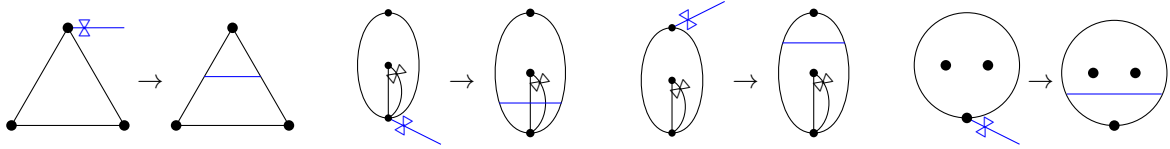

\clearpage
\renewcommand{\arraystretch}{1.7}
{\begin{table}[ht]
\begin{minipage}{0.38\hsize}
  \begin{center}\begin{tabular}{c|c|c}
 \multicolumn{3}{c}{}\\
 \multicolumn{2}{c|}{Segments} & $(a_1,a_2,a_3)$
\\
\hline
\multirow{3}{*}{\begin{tikzpicture}
 \coordinate (0) at (0,0);
 \coordinate (1) at (120:1.5);
 \coordinate (2) at (180:1.5);
 \coordinate (3) at (150:1.3);
  \coordinate (6) at (180:0.9);   
  \coordinate (7) at (120:0.9);  
 \draw (0) to (1);
 \draw (1) to node[above=3,left]{$\tau_i$} (2);
 \draw (0) to (2);  
 \draw [blue] (6) to node[fill=white,inner sep=1]{$e_i$} (7);
 \fill(0) circle (0.7mm); \fill (1) circle (0.7mm); \fill (2) circle (0.7mm);
\end{tikzpicture}}
 & $e_1$ & $(0,1,1)$\\\cline{2-3}
 & $e_2$ & $(1,0,1)$\\\cline{2-3}
 & $e_3$ & $(1,1,0)$
\\
\hline
\multirow{3}{*}{\begin{tikzpicture}
 \coordinate (0) at (0,0);
 \coordinate (1) at (120:1.5);
 \coordinate (2) at (180:1.5);
 \coordinate (3) at (150:1.3);
 \draw (0) to (1);
 \draw (1) to node[above=3,left]{$\tau_i$} (2);
 \draw (0) to (2); 
 \draw [blue] (0) to node[fill=white,inner sep=1,pos=0.6]{$h_i$} (3);
 \fill(0) circle (0.7mm); \fill (1) circle (0.7mm); \fill (2) circle (0.7mm);
\end{tikzpicture}}
 & $h_1$ & $(1,0,0)$\\\cline{2-3}
 & $h_2$ & $(0,1,0)$\\\cline{2-3}
 & $h_3$ & $(0,0,1)$
\\
\hline
 \multicolumn{3}{c}{}\\ \multicolumn{3}{c}{}\\ \multicolumn{3}{c}{}\\ \multicolumn{3}{c}{}\\ \multicolumn{3}{c}{}\\ \multicolumn{3}{c}{}\\
  \end{tabular}\end{center}
\end{minipage}
\begin{minipage}{0.58\hsize}
  \begin{center}\begin{tabular}{c|c|c|c|c}
 \multicolumn{2}{c|}{}& \multicolumn{3}{c}{$(a_1,a_2,a_3)$}\\\cline{3-5}
 \multicolumn{2}{c|}{Segments} & Case($-$) & Case($\tau_1$) & Case($\tau_1$,$\tau_2$)
\\
\hline
\multirow{3}{*}{\begin{tikzpicture}
 \coordinate (0) at (0,0);
 \coordinate (1) at (120:1.5);
 \coordinate (2) at (180:1.5);
 \coordinate (3) at (150:1.3);
 \draw (0) to (1);
 \draw (1) to node[above=3,left]{$\tau_i$} (2);
 \draw (0) to (2); 
 \draw [blue] (0) to node[fill=white,inner sep=1,pos=0.6]{$\overline{h_i}$} node[pos=0.2]{\rotatebox{60}{\footnotesize $\bowtie$}}  (3);
 \fill(0) circle (0.7mm); \fill (1) circle (0.7mm); \fill (2) circle (0.7mm);
\end{tikzpicture}}
 & $\overline{h_1}$ & \multirow{3}{*}{$(1,1,1)$} & $(1,1,1)$ & $(3,2,2)$\\\cline{2-2}\cline{4-5}
 & $\overline{h_2}$ & & $(2,2,1)$ & $(2,3,2)$\\\cline{2-2}\cline{4-5}
 & $\overline{h_3}$ & & $(2,1,2)$ & $(2,2,3)$
\\
\hline
\multirow{3}{*}{\begin{tikzpicture}
 \coordinate (0) at (0,0);
 \coordinate (1) at (120:1.5);
 \coordinate (2) at (180:1.5);
 \coordinate (3) at (150:1.3);
 \draw (0) to (1);
 \draw (1) to node[above=3,left]{$\tau_i$} (2);
 \draw (0) to (2);  
 \draw [blue] (1) [out=30, in=150, relative] to node[fill=white,inner sep=1,below=5,right=-3]{$\overline{E_{ir}}$} node[pos=0.2]{\rotatebox{-10}{\footnotesize $\bowtie$}}  (2);
 \fill(0) circle (0.7mm); \fill (1) circle (0.7mm); \fill (2) circle (0.7mm);
\end{tikzpicture}}
& $\overline{E_{1r}}$ & $(0,1,0)$ & $\times$ & $\times$\\\cline{2-5}
 & $\overline{E_{2r}}$ & $(0,0,1)$ & $(0,0,1)$ & $\times$\\\cline{2-5}
 & $\overline{E_{3r}}$ & $(1,0,0)$ & $(2,1,0)$ & $\times$
\\
\hline
\multirow{3}{*}{\begin{tikzpicture}
 \coordinate (0) at (0,0);
 \coordinate (1) at (120:1.5);
 \coordinate (2) at (180:1.5);
 \coordinate (3) at (150:1.3);
 \draw (0) to (1);
 \draw (1) to node[above=3,left]{$\tau_i$} (2);
 \draw (0) to (2);  
 \draw [blue] (1) [out=30, in=150, relative] to node[fill=white,inner sep=1,below=5,right=-3]{$\overline{E_{il}}$} node[pos=0.8]{\rotatebox{130}{\footnotesize $\bowtie$}}  (2);
 \fill(0) circle (0.7mm); \fill (1) circle (0.7mm); \fill (2) circle (0.7mm);
\end{tikzpicture}}
& $\overline{E_{1l}}$ & $(0,0,1)$ & $\times$ & $\times$\\\cline{2-5}
 & $\overline{E_{2l}}$ & $(1,0,0)$ & $(2,0,1)$ & $\times$\\\cline{2-5}
 & $\overline{E_{3l}}$ & $(0,1,0)$ & $(0,1,0)$ & $\times$
\\
\hline
\multirow{3}{*}{\begin{tikzpicture}
 \coordinate (0) at (0,0);
 \coordinate (1) at (120:1.5);
 \coordinate (2) at (180:1.5);
 \coordinate (3) at (150:1.3);
 \draw (0) to (1);
 \draw (1) to node[above=3,left]{$\tau_i$} (2);
 \draw (0) to (2);  
 \draw [blue] (1) [out=30, in=150, relative] to node[fill=white,inner sep=1,below=5,right=-3]{$\overline{\overline{E_{i}}}$} node[pos=0.2]{\rotatebox{-10}{\footnotesize $\bowtie$}} node[pos=0.8]{\rotatebox{130}{\footnotesize $\bowtie$}}  (2);
 \fill(0) circle (0.7mm); \fill (1) circle (0.7mm); \fill (2) circle (0.7mm);
\end{tikzpicture}}
& $\overline{\overline{E_{1}}}$ & $(2,1,1)$ & $(4,2,2)$ & \multirow{3}{*}{$(4,4,4)$}\\\cline{2-4}
 & $\overline{\overline{E_{2}}}$ & $(1,2,1)$ & \multirow{2}{*}{$(2,2,2)$} &\\\cline{2-3}
 & $\overline{\overline{E_{3}}}$ & $(1,1,2)$ & &
\\
\hline
  \end{tabular}\end{center}
\end{minipage}\vspace{3mm}
 \caption{Segments of a tagged arc in triangle pieces and the corresponding intersection sub-vectors $(a_1,a_2,a_3)$}\label{S0}
\end{table}}

\renewcommand{\arraystretch}{1.7}
{\begin{table}[!h]
\begin{minipage}{0.48\hsize}\begin{center}\vspace{-5mm}
\begin{tabular}{c|c}
 \multicolumn{2}{c}{Triangle piece case}\\\hline
 Segments & Modified segments\\\hline
 $\overline{h_i}$ & $y$\\\hline
 $\overline{2h_i}$ & $\{e_1,e_2,e_3\}$\\\hline
 $\overline{E_{ir}}$ & $h_{i+1}$\\\hline
 $\overline{E_{il}}$ & $h_{i-1}$\\\hline
 $\overline{\overline{E_i}}$ & $\{2e_{i-1},2e_{i+1}\}$\\\hline
 \multicolumn{2}{c}{\vspace{-0.5cm}}\\
 \multicolumn{2}{c}{$2$-puncture piece case}\\\hline
 $Q$ & $2h$\\\hline
  $\overline{R_0}$ & $\{r_0,r_1\}$\\\hline
 $\overline{L_0}$ & $\{l_0,l_{-1}\}$\\\hline
 $2\overline{c_n}$ & $\{s_{n-1},s_n,h\}$\\\hline
\end{tabular}
\end{center}\end{minipage}
\begin{minipage}{0.48\hsize}\begin{center}\vspace{-2.07cm}
\begin{tabular}{c|c}
 \multicolumn{2}{c}{$1$-puncture piece case}\\\hline
 Segments & Modified segments\\\hline
 $\overline{l_{\pm}}$ & $y_l$\\\hline
 $\overline{r_{\pm}}$ & $y_r$\\\hline
 $\underline{L}$ & $l_+$\\\hline
 $\underline{R}$ & $r_+$\\\hline
 $\overline{L}$ & $l_-$\\\hline
 $\overline{R}$ & $r_-$\\\hline
 $\underline{\overline{L}}$, $\underline{\overline{R}}$ & $\{u,d\}$\\\hline
 $\overline{P_{-}}$ & $\{r_p,l_p\}$\\\hline
\end{tabular}
\end{center}\end{minipage}\vspace{3mm}
 \caption{Segments $s$ and the corresponding modified segments ${\sf M}_T(s)$ in $\square$ such that $s \neq {\sf M}_T(s)$, where $h_{3k+j}=h_j$ and $e_{3k+j}=e_j$ for any $k, j \in \bZ$}\label{sm}
\end{table}}

\renewcommand{\arraystretch}{2}
{\begin{table}[p]
\vspace{10mm}
{\tabcolsep=5mm
\begin{tabular}{c|c|c|c|c|c}\hline
\begin{tikzpicture}[baseline=-10mm]
 \node at(0,0) {Segments};
\end{tikzpicture} &
\begin{tikzpicture}
 \coordinate (0) at (0,0);
 \coordinate (1) at (0,-1);
 \coordinate (2) at (0,-2);
 \coordinate (3) at (-0.52,-0.5);
 \coordinate (4) at (0.52,-0.5);
 \draw (1) to (2);
 \draw (1) to [out=60,in=120,relative] node[pos=0.2]{\rotatebox{40}{\footnotesize $\bowtie$}} (2);
 \draw (2) to [out=180,in=180] (0);
 \draw (2) to [out=0,in=0] (0);
 \draw[blue] (3) to node[fill=white,inner sep=0.5]{$u$} (4);
 \fill(0) circle (0.7mm); \fill (1) circle (0.7mm); \fill (2) circle (0.7mm);
 \node at(0,0.2) {};
\end{tikzpicture}
&
\begin{tikzpicture}
 \coordinate (0) at (0,0);
 \coordinate (1) at (0,-1);
 \coordinate (2) at (0,-2);
 \coordinate (5) at (-0.52,-1.5);
 \coordinate (6) at (0.52,-1.5);
 \draw (1) to (2);
 \draw (1) to [out=60,in=120,relative] node[pos=0.2]{\rotatebox{40}{\footnotesize $\bowtie$}} (2);
 \draw (2) to [out=180,in=180] (0);
 \draw (2) to [out=0,in=0] (0);
 \draw[blue] (5) to node[fill=white,inner sep=0.5,pos=0.3,above]{$d$} (6);
 \fill(0) circle (0.7mm); \fill (1) circle (0.7mm); \fill (2) circle (0.7mm);
\end{tikzpicture}
&
\begin{tikzpicture}
 \coordinate (0) at (0,0);
 \coordinate (1) at (0,-1);
 \coordinate (2) at (0,-2);
 \draw (1) to (2);
 \draw (1) to [out=60,in=120,relative] node[pos=0.2]{\rotatebox{40}{\footnotesize $\bowtie$}} (2);
 \draw (2) to [out=180,in=180] (0);
 \draw (2) to [out=0,in=0] (0);
 \draw[blue] (1) to node[right=-2]{\small$P_+$} (0);
 \fill(0) circle (0.7mm); \fill (1) circle (0.7mm); \fill (2) circle (0.7mm);
\end{tikzpicture}
&
\begin{tikzpicture}
 \coordinate (0) at (0,0);
 \coordinate (1) at (0,-1);
 \coordinate (2) at (0,-2);
 \coordinate (3) at (-0.52,-0.5);
 \coordinate (5) at (-0.52,-1.5);
 \draw (1) to (2);
 \draw (1) to [out=60,in=120,relative] node[pos=0.2]{\rotatebox{40}{\footnotesize $\bowtie$}} (2);
 \draw (2) to [out=180,in=180] (0);
 \draw (2) to [out=0,in=0] (0);
 \draw[blue] (3) to [out=0,in=90] (0.4,-1);
 \draw[blue] (0.4,-1) to [out=-90,in=0] (5);
 \fill(0) circle (0.7mm); \fill (1) circle (0.7mm); \fill (2) circle (0.7mm);
 \node[blue] at(-0.3,-0.8) {$l$};
\end{tikzpicture}
&
\begin{tikzpicture}
 \coordinate (0) at (0,0);
 \coordinate (1) at (0,-1);
 \coordinate (2) at (0,-2);
 \coordinate (4) at (0.52,-0.5);
 \coordinate (6) at (0.52,-1.5);;
 \draw (1) to (2);
 \draw (1) to [out=60,in=120,relative] node[pos=0.2]{\rotatebox{40}{\footnotesize $\bowtie$}} (2);
 \draw (2) to [out=180,in=180] (0);
 \draw (2) to [out=0,in=0] (0);
  \draw[blue] (4) to [out=180,in=90] (-0.4,-1);
 \draw[blue] (-0.4,-1) to [out=-90,in=180] (6);
 \fill(0) circle (0.7mm); \fill (1) circle (0.7mm); \fill (2) circle (0.7mm);
 \node[blue] at(0.3,-0.8) {$r$};
\end{tikzpicture}
\\\hline
 \hspace{-3mm}$(a_1,a_2,a_3,a_4)$\hspace{-3mm} & (1,1,0,0) & (1,1,1,1) & (0,0,0,1) & (2,0,1,1) & (0,2,1,1)
\\\hline\hline
\begin{tikzpicture}
 \coordinate (0) at (0,0);
 \coordinate (1) at (0,-1);
 \coordinate (2) at (0,-2);
 \coordinate (5) at (-0.52,-1.5);
 \draw (1) to (2);
 \draw (1) to [out=60,in=120,relative] node[pos=0.2]{\rotatebox{40}{\footnotesize $\bowtie$}} (2);
 \draw (2) to [out=180,in=180] (0);
 \draw (2) to [out=0,in=0] (0);
 \draw[blue] (0) .. controls (0.8,-1.3) and (0,-1.7) .. (5);
 \fill(0) circle (0.7mm); \fill (1) circle (0.7mm); \fill (2) circle (0.7mm);
 \node[blue] at(0,-0.6) {$l_+$};
\end{tikzpicture}
&
\begin{tikzpicture}
 \coordinate (0) at (0,0);
 \coordinate (1) at (0,-1);
 \coordinate (2) at (0,-2);
 \coordinate (6) at (0.52,-1.5);
 \draw (1) to (2);
 \draw (1) to [out=60,in=120,relative] node[pos=0.2]{\rotatebox{40}{\footnotesize $\bowtie$}} (2);
 \draw (2) to [out=180,in=180] (0);
 \draw (2) to [out=0,in=0] (0);
 \draw[blue] (0) .. controls (-0.8,-1.3) and (0,-1.7) .. (6);
 \fill(0) circle (0.7mm); \fill (1) circle (0.7mm); \fill (2) circle (0.7mm);
 \node[blue] at(0.1,-0.6) {$r_+$};
\end{tikzpicture}
&
\begin{tikzpicture}
 \coordinate (0) at (0,0);
 \coordinate (1) at (0,-1);
 \coordinate (2) at (0,-2);
 \coordinate (3) at (-0.52,-0.5);
 \draw (1) to (2);
 \draw (1) to [out=60,in=120,relative] node[pos=0.2]{\rotatebox{40}{\footnotesize $\bowtie$}} (2);
 \draw (2) to [out=180,in=180] (0);
 \draw (2) to [out=0,in=0] (0);
 \draw[blue] (2) .. controls (0.9,-1.5) and (0.3,-0.2) .. (3);
 \fill(0) circle (0.7mm); \fill (1) circle (0.7mm); \fill (2) circle (0.7mm);
 \node[blue] at(-0.3,-0.8) {$l_-$};
\end{tikzpicture}
&
\begin{tikzpicture}
 \coordinate (0) at (0,0);
 \coordinate (1) at (0,-1);
 \coordinate (2) at (0,-2);
 \coordinate (4) at (0.52,-0.5);
 \draw (1) to (2);
 \draw (1) to [out=60,in=120,relative] node[pos=0.2]{\rotatebox{40}{\footnotesize $\bowtie$}} (2);
 \draw (2) to [out=180,in=180] (0);
 \draw (2) to [out=0,in=0] (0);
 \draw[blue] (2) .. controls (-0.8,-0.7) and (0,-0.3) .. (4);
 \fill(0) circle (0.7mm); \fill (1) circle (0.7mm); \fill (2) circle (0.7mm);
 \node[blue] at(0.3,-0.8) {$r_-$};
 \node at(0,0.2) {};
\end{tikzpicture}
&
\begin{tikzpicture}
 \coordinate (0) at (0,0);
 \coordinate (1) at (0,-1);
 \coordinate (2) at (0,-2);
 \coordinate (l) at (-0.6,-1);
 \draw (1) to (2);
 \draw (1) to [out=60,in=120,relative] node[pos=0.2]{\rotatebox{40}{\footnotesize $\bowtie$}} (2);
 \draw (2) to [out=180,in=180] (0);
 \draw (2) to [out=0,in=0] (0);
 \draw[blue] (l) to (1);
 \fill(0) circle (0.7mm); \fill (1) circle (0.7mm); \fill (2) circle (0.7mm);
 \node[blue] at(-0.2,-0.7) {$l_p$};
\end{tikzpicture}
&
\begin{tikzpicture}
 \coordinate (0) at (0,0);
 \coordinate (1) at (0,-1);
 \coordinate (2) at (0,-2);
 \coordinate (r) at (0.6,-1);
 \draw (1) to (2);
 \draw (1) to [out=60,in=120,relative] node[pos=0.2]{\rotatebox{40}{\footnotesize $\bowtie$}} (2);
 \draw (2) to [out=180,in=180] (0);
 \draw (2) to [out=0,in=0] (0);
 \draw[blue] (r) to (1);
 \fill(0) circle (0.7mm); \fill (1) circle (0.7mm); \fill (2) circle (0.7mm);
 \node[blue] at(0.2,-0.7) {$r_p$};
\end{tikzpicture}
\\\hline
 (1,0,1,1) & (0,1,1,1) & (1,0,0,0) & (0,1,0,0) & (1,0,0,1) & (0,1,0,1)
\\\hline\hline
\begin{tikzpicture}
 \coordinate (0) at (0,0);
 \coordinate (1) at (0,-1);
 \coordinate (2) at (0,-2);
 \coordinate (5) at (-0.52,-1.5);
 \draw (1) to (2);
 \draw (1) to [out=60,in=120,relative] node[pos=0.2]{\rotatebox{40}{\footnotesize $\bowtie$}} (2);
 \draw (2) to [out=180,in=180] (0);
 \draw (2) to [out=0,in=0] (0);
 \draw[blue] (0) .. controls (0.8,-1.3) and (0,-1.7) .. node[pos=0.05]{\rotatebox{30}{\footnotesize $\bowtie$}} (5);
 \fill(0) circle (0.7mm); \fill (1) circle (0.7mm); \fill (2) circle (0.7mm);
 \node[blue] at(-0.1,-0.6) {$\overline{l_+}$};
\end{tikzpicture}
&
\begin{tikzpicture}
 \coordinate (0) at (0,0);
 \coordinate (1) at (0,-1);
 \coordinate (2) at (0,-2);
 \coordinate (6) at (0.52,-1.5);
 \draw (1) to (2);
 \draw (1) to [out=60,in=120,relative] node[pos=0.2]{\rotatebox{40}{\footnotesize $\bowtie$}} (2);
 \draw (2) to [out=180,in=180] (0);
 \draw (2) to [out=0,in=0] (0);
 \draw[blue] (0) .. controls (-0.8,-1.3) and (0,-1.7) .. node[pos=0.05]{\rotatebox{150}{\footnotesize $\bowtie$}} (6);
 \fill(0) circle (0.7mm); \fill (1) circle (0.7mm); \fill (2) circle (0.7mm);
 \node[blue] at(0.1,-0.6) {$\overline{r_+}$};
\end{tikzpicture}
&
\begin{tikzpicture}
 \coordinate (0) at (0,0);
 \coordinate (1) at (0,-1);
 \coordinate (2) at (0,-2);
 \coordinate (3) at (-0.52,-0.5);
 \draw (1) to (2);
 \draw (1) to [out=60,in=120,relative] node[pos=0.2]{\rotatebox{40}{\footnotesize $\bowtie$}} (2);
 \draw (2) to [out=180,in=180] (0);
 \draw (2) to [out=0,in=0] (0);
 \draw[blue] (2) .. controls (0.85,-1.5) and (0.3,-0.2) .. node[pos=0.25]{\rotatebox{-25}{\footnotesize $\bowtie$}} (3);
 \fill(0) circle (0.7mm); \fill (1) circle (0.7mm); \fill (2) circle (0.7mm);
 \node[blue] at(-0.3,-0.8) {$\overline{l_-}$};
\end{tikzpicture}
&
\begin{tikzpicture}
 \coordinate (0) at (0,0);
 \coordinate (1) at (0,-1);
 \coordinate (2) at (0,-2);
 \coordinate (4) at (0.52,-0.5);
 \draw (1) to (2);
 \draw (1) to [out=60,in=120,relative] node[pos=0.2]{\rotatebox{40}{\footnotesize $\bowtie$}} (2);
 \draw (2) to [out=180,in=180] (0);
 \draw (2) to [out=0,in=0] (0);
 \draw[blue] (2) .. controls (-0.8,-0.7) and (0,-0.3) .. node[pos=0.1]{\rotatebox{30}{\footnotesize $\bowtie$}} (4);
 \fill(0) circle (0.7mm); \fill (1) circle (0.7mm); \fill (2) circle (0.7mm);
 \node[blue] at(0.3,-0.8) {$\overline{r_-}$};
 \node at(0,0.2) {};
\end{tikzpicture}
&
\begin{tikzpicture}
 \coordinate (0) at (0,0);
 \coordinate (1) at (0,-1);
 \coordinate (2) at (0,-2);
 \draw (1) to (2);
 \draw (1) to [out=60,in=120,relative] node[pos=0.2]{\rotatebox{40}{\footnotesize $\bowtie$}} (2);
 \draw (2) to [out=180,in=180] (0);
 \draw (2) to [out=0,in=0] (0);
 \draw[blue] (2) to [out=-55,in=-125,relative] node[pos=0.3]{\rotatebox{-20}{\footnotesize $\bowtie$}} (0);
 \fill(0) circle (0.7mm); \fill (1) circle (0.7mm); \fill (2) circle (0.7mm);
 \node[blue] at(0,-0.6) {$\underline{L}$};
\end{tikzpicture}
&
\begin{tikzpicture}
 \coordinate (0) at (0,0);
 \coordinate (1) at (0,-1);
 \coordinate (2) at (0,-2);
 \draw (1) to (2);
 \draw (1) to [out=60,in=120,relative] node[pos=0.2]{\rotatebox{40}{\footnotesize $\bowtie$}} (2);
 \draw (2) to [out=180,in=180] (0);
 \draw (2) to [out=0,in=0] (0);
 \draw[blue] (2) to [out=45,in=135,relative] node[pos=0.15]{\rotatebox{35}{\footnotesize $\bowtie$}} (0);
 \fill(0) circle (0.7mm); \fill (1) circle (0.7mm); \fill (2) circle (0.7mm);
 \node[blue] at(0,-0.6) {$\underline{R}$};
\end{tikzpicture}
\\\hline
 (2,1,1,1) & (1,2,1,1) & (2,1,1,1) & (1,2,1,1) & (1,0,1,1) & (0,1,1,1)\\\hline
 (3,2,2,2) & (2,3,2,2) & (3,2,1,1) & (2,3,1,1) & $\times$ & $\times$
\\\hline\hline
\begin{tikzpicture}
 \coordinate (0) at (0,0);
 \coordinate (1) at (0,-1);
 \coordinate (2) at (0,-2);
 \draw (1) to (2);
 \draw (1) to [out=60,in=120,relative] node[pos=0.2]{\rotatebox{40}{\footnotesize $\bowtie$}} (2);
 \draw (2) to [out=180,in=180] (0);
 \draw (2) to [out=0,in=0] (0);
 \draw[blue] (2) to [out=-55,in=-125,relative] node[pos=0.85]{\rotatebox{35}{\footnotesize $\bowtie$}} (0);
 \fill(0) circle (0.7mm); \fill (1) circle (0.7mm); \fill (2) circle (0.7mm);
 \node[blue] at(0,-0.6) {$\overline{L}$};
\end{tikzpicture}
&
\begin{tikzpicture}
 \coordinate (0) at (0,0);
 \coordinate (1) at (0,-1);
 \coordinate (2) at (0,-2);
 \draw (1) to (2);
 \draw (1) to [out=60,in=120,relative] node[pos=0.2]{\rotatebox{40}{\footnotesize $\bowtie$}} (2);
 \draw (2) to [out=180,in=180] (0);
 \draw (2) to [out=0,in=0] (0);
 \draw[blue] (2) to [out=45,in=135,relative] node[pos=0.85]{\rotatebox{-35}{\footnotesize $\bowtie$}} (0);
 \fill(0) circle (0.7mm); \fill (1) circle (0.7mm); \fill (2) circle (0.7mm);
 \node[blue] at(0,-0.6) {$\overline{R}$};
 \node at(0,0.2) {};
\end{tikzpicture}
&
\begin{tikzpicture}
 \coordinate (0) at (0,0);
 \coordinate (1) at (0,-1);
 \coordinate (2) at (0,-2);
 \draw (1) to (2);
 \draw (1) to [out=60,in=120,relative] node[pos=0.2]{\rotatebox{40}{\footnotesize $\bowtie$}} (2);
 \draw (2) to [out=180,in=180] (0);
 \draw (2) to [out=0,in=0] (0);
 \draw[blue] (2) to [out=-55,in=-125,relative] node[pos=0.3]{\rotatebox{-20}{\footnotesize $\bowtie$}} node[pos=0.85]{\rotatebox{35}{\footnotesize $\bowtie$}} (0);
 \fill(0) circle (0.7mm); \fill (1) circle (0.7mm); \fill (2) circle (0.7mm);
 \node[blue] at(0,-0.6) {$\underline{\overline{L}}$};
\end{tikzpicture}
&
\begin{tikzpicture}
 \coordinate (0) at (0,0);
 \coordinate (1) at (0,-1);
 \coordinate (2) at (0,-2);
 \draw (1) to (2);
 \draw (1) to [out=60,in=120,relative] node[pos=0.2]{\rotatebox{40}{\footnotesize $\bowtie$}} (2);
 \draw (2) to [out=180,in=180] (0);
 \draw (2) to [out=0,in=0] (0);
 \draw[blue] (2) to [out=45,in=135,relative] node[pos=0.85]{\rotatebox{-35}{\footnotesize $\bowtie$}} node[pos=0.15]{\rotatebox{35}{\footnotesize $\bowtie$}} (0);
 \fill(0) circle (0.7mm); \fill (1) circle (0.7mm); \fill (2) circle (0.7mm);
 \node[blue] at(0,-0.6) {$\underline{\overline{R}}$};
\end{tikzpicture}
&
\begin{tikzpicture}
 \coordinate (0) at (0,0);
 \coordinate (1) at (0,-1);
 \coordinate (2) at (0,-2);
 \draw (1) to (2);
 \draw (1) to [out=60,in=120,relative] node[pos=0.2]{\rotatebox{40}{\footnotesize $\bowtie$}} (2);
 \draw (2) to [out=180,in=180] (0);
 \draw (2) to [out=0,in=0] (0);
 \draw[blue] (1) to node[pos=0.8]{\footnotesize $\bowtie$} node[right=-2,pos=0.4]{\small$\overline{P_+}$} (0);
 \fill(0) circle (0.7mm); \fill (1) circle (0.7mm); \fill (2) circle (0.7mm);
\end{tikzpicture}
&
\begin{tikzpicture}
 \coordinate (0) at (0,0);
 \coordinate (1) at (0,-1);
 \coordinate (2) at (0,-2);
 \draw (1) to (2);
 \draw (1) to [out=60,in=120,relative] node[pos=0.2]{\rotatebox{40}{\footnotesize $\bowtie$}} (2);
 \draw (2) to [out=180,in=180] (0);
 \draw (2) to [out=0,in=0] (0);
 \draw[blue] (2) to [out=60,in=120,relative] node[pos=0.25]{\rotatebox{35}{\footnotesize $\bowtie$}} (1);
 \fill(0) circle (0.7mm); \fill (1) circle (0.7mm); \fill (2) circle (0.7mm);
 \node[blue] at(-0.2,-0.8) {\small$\overline{P_-}$};
\end{tikzpicture}
\\\hline
 (1,0,0,0) & (0,1,0,0) & (2,2,1,1) & (2,2,1,1) & (1,1,0,1) & (1,1,0,2)\\\hline
 $\times$ & $\times$ & (4,4,2,2) & (4,4,2,2) & (2,2,1,2) & (2,2,0,2)\\\hline
  \end{tabular}}\vspace{5mm}
 \caption{Segments of a tagged arc in $1$-puncture pieces and the corresponding intersection sub-vectors $(a_1,a_2,a_3,a_4)$ that are values of Case($-$) (above) and of Case($\tau_1$) (below)}\label{S1}
\end{table}}

\renewcommand{\arraystretch}{1.5}
{\begin{table}[!p]
\begin{tabular}{c|c|cc}
 \multicolumn{2}{c|}{Segments} & $(a_1,a_2,a_3,a_4,a_5)$\\\hline
\multirow{3}{*}{$\cdots$
     \hspace{2mm}
\begin{tikzpicture}[baseline=0mm,scale=0.7]
 \coordinate (0) at (0,0);
 \coordinate (1) at (-0.4,0);
 \coordinate (1') at (0.4,0);
 \coordinate (d) at (0,-1);
 \coordinate (l) at (-1,0);
 \coordinate (r) at (1,0);
 \draw (0) circle (1);
 \draw[blue] (l) .. controls (-0.3,1) and (0.3,-1) .. (r);
 \fill (1) circle (1mm); \fill (1') circle (1mm); \fill (d) circle (1mm);
 \node[fill=white,inner sep=0.5] at(0,-0.5) {\color{blue}$s_{-1}$};
\end{tikzpicture}
     \hspace{2mm}
\begin{tikzpicture}[baseline=0mm,scale=0.7]
 \coordinate (0) at (0,0);
 \coordinate (1) at (-0.4,0);
 \coordinate (1') at (0.4,0);
 \coordinate (d) at (0,-1);
 \coordinate (l) at (-1,0);
 \coordinate (r) at (1,0);
 \draw (0) circle (1);
 \draw[blue] (l) .. controls (-0.3,-1) and (0.3,1) .. (r);
 \fill (1) circle (1mm); \fill (1') circle (1mm); \fill (d) circle (1mm);
 \node[fill=white,inner sep=0.5] at(0) {\color{blue}$s_{0}$};
\end{tikzpicture}
     \hspace{2mm}
\begin{tikzpicture}[baseline=0mm,scale=0.7]
 \coordinate (0) at (0,0);
 \coordinate (1) at (-0.4,0);
 \coordinate (1') at (0.4,0);
 \coordinate (d) at (0,-1);
 \coordinate (l) at (-1,0);
 \coordinate (r) at (1,0);
 \coordinate (l') at (-0.7,0);
 \coordinate (r') at (0.7,0);
 \draw (0) circle (1);
 \draw[blue] (l) .. controls (-0.5,-1) and (0.7,-0.7) .. (r');
 \draw[blue] (r') .. controls (0.6,0.5) and (0.1,0.3) .. (0);
 \draw[blue] (l') .. controls (-0.6,-0.5) and (-0.1,-0.3) .. (0);
 \draw[blue] (r) .. controls (0.5,1) and (-0.7,0.7) .. (l');
 \fill (1) circle (1mm); \fill (1') circle (1mm); \fill (d) circle (1mm);
 \node[fill=white,inner sep=0.5] at(0) {\color{blue}$s_{1}$};
\end{tikzpicture}
     \hspace{2mm}
 $\cdots$
}
& \multirow{3}{*}{$s_n$} & \multirow{3}{*}{$(2,|n+1|,|n+1|,|n|,|n|)$} &\\
& & &\\
& & &
\\\hline
 \multirow{3}{*}{$\cdots$
     \hspace{2mm}
\begin{tikzpicture}[baseline=0mm,scale=0.7]
 \coordinate (0) at (0,0);
 \coordinate (1) at (-0.4,0);
 \coordinate (1') at (0.4,0);
 \coordinate (d) at (0,-1);
 \coordinate (u) at (0,1);
 \coordinate (l') at (-0.7,0);
 \coordinate (r') at (0.7,0);
 \draw (0) circle (1);
 \draw[blue] (d) .. controls (-0.7,-0.5) and (-0.7,-0.2) .. (l');
 \draw[blue] (l') .. controls (-0.6,0.5) and (-0.1,0.3) .. (0);
 \draw[blue] (r') .. controls (0.6,-0.5) and (0.1,-0.3) .. (0);
 \draw[blue] (u) .. controls (0.7,0.5) and (0.7,0.2) .. (r');
 \fill (1) circle (1mm); \fill (1') circle (1mm); \fill (d) circle (1mm);
 \node at(0,-0.5) {\color{blue}$c_{-1}$};
\end{tikzpicture}
     \hspace{2mm}
\begin{tikzpicture}[baseline=0mm,scale=0.7]
 \coordinate (0) at (0,0);
 \coordinate (1) at (-0.4,0);
 \coordinate (1') at (0.4,0);
 \coordinate (d) at (0,-1);
 \coordinate (u) at (0,1);
 \draw (0) circle (1);
 \draw[blue] (d) to (u);
 \fill (1) circle (1mm); \fill (1') circle (1mm); \fill (d) circle (1mm);
 \node[fill=white,inner sep=0.5] at(0) {\color{blue}$c_{0}$};
\end{tikzpicture}
     \hspace{2mm}
\begin{tikzpicture}[baseline=0mm,scale=0.7]
 \coordinate (0) at (0,0);
 \coordinate (1) at (-0.4,0);
 \coordinate (1') at (0.4,0);
 \coordinate (d) at (0,-1);
 \coordinate (u) at (0,1);
 \coordinate (l') at (-0.7,0);
 \coordinate (r') at (0.7,0);
 \draw (0) circle (1);
 \draw[blue] (d) .. controls (0.7,-0.5) and (0.7,-0.2) .. (r');
 \draw[blue] (l') .. controls (-0.6,-0.5) and (-0.1,-0.3) .. (0);
 \draw[blue] (r') .. controls (0.6,0.5) and (0.1,0.3) .. (0);
 \draw[blue] (u) .. controls (-0.7,0.5) and (-0.7,0.2) .. (l');
 \fill (1) circle (1mm); \fill (1') circle (1mm); \fill (d) circle (1mm);
 \node[fill=white,inner sep=0.5] at(0) {\color{blue}$c_{1}$};
\end{tikzpicture}
     \hspace{2mm}
 $\cdots$
}
& \multirow{3}{*}{$c_n$} & $(1,0,0,0,0)$ & ($n = 0$)\\
& & $(1,n,n,n-1,n-1)$ & ($n > 0$)\\
& & $(1,-n-1,-n-1,-n,-n)$ & ($n < 0$)
\\\hline
 \multirow{3}{*}{$\cdots$
     \hspace{2mm}
\begin{tikzpicture}[baseline=0mm,scale=0.7]
 \coordinate (0) at (0,0);
 \coordinate (1) at (-0.4,0);
 \coordinate (1') at (0.4,0);
 \coordinate (d) at (0,-1);
 \coordinate (u) at (0,1);
 \coordinate (l') at (-0.7,0);
 \coordinate (r') at (0.7,0);
 \draw (0) circle (1);
 \draw[blue] (d) .. controls (-0.7,-0.5) and (-0.7,-0.2) .. node[pos=0.1]{\rotatebox{50}{\footnotesize $\bowtie$}} (l');
 \draw[blue] (l') .. controls (-0.6,0.5) and (-0.1,0.3) .. (0);
 \draw[blue] (r') .. controls (0.6,-0.5) and (0.1,-0.3) .. (0);
 \draw[blue] (u) .. controls (0.7,0.5) and (0.7,0.2) .. (r');
 \fill (1) circle (1mm); \fill (1') circle (1mm); \fill (d) circle (1mm);
 \node at(0,-0.5) {\color{blue}$\overline{c_{-1}}$};
\end{tikzpicture}
     \hspace{2mm}
\begin{tikzpicture}[baseline=0mm,scale=0.7]
 \coordinate (0) at (0,0);
 \coordinate (1) at (-0.4,0);
 \coordinate (1') at (0.4,0);
 \coordinate (d) at (0,-1);
 \coordinate (u) at (0,1);
 \draw (0) circle (1);
 \draw[blue] (d) to node[pos=0.1]{\rotatebox{0}{\footnotesize $\bowtie$}} (u);
 \fill (1) circle (1mm); \fill (1') circle (1mm); \fill (d) circle (1mm);
 \node[fill=white,inner sep=0.5] at(0) {\color{blue}$\overline{c_{0}}$};
\end{tikzpicture}
     \hspace{2mm}
\begin{tikzpicture}[baseline=0mm,scale=0.7]
 \coordinate (0) at (0,0);
 \coordinate (1) at (-0.4,0);
 \coordinate (1') at (0.4,0);
 \coordinate (d) at (0,-1);
 \coordinate (u) at (0,1);
 \coordinate (l') at (-0.7,0);
 \coordinate (r') at (0.7,0);
 \draw (0) circle (1);
 \draw[blue] (d) .. controls (0.7,-0.5) and (0.7,-0.2) .. node[pos=0.1]{\rotatebox{-50}{\footnotesize $\bowtie$}} (r');
 \draw[blue] (l') .. controls (-0.6,-0.5) and (-0.1,-0.3) .. (0);
 \draw[blue] (r') .. controls (0.6,0.5) and (0.1,0.3) .. (0);
 \draw[blue] (u) .. controls (-0.7,0.5) and (-0.7,0.2) .. (l');
 \fill (1) circle (1mm); \fill (1') circle (1mm); \fill (d) circle (1mm);
 \node[fill=white,inner sep=0.5] at(0) {\color{blue}$\overline{c_{1}}$};
\end{tikzpicture}
     \hspace{2mm}
 $\cdots$
}
& \multirow{3}{*}{$\overline{c_n}$} & $(3,1,1,1,1)$ & ($n = 0$)\\
& & $(3,n+1,n+1,n,n)$ & ($n > 0$)\\
& & $(3,-n,-n,-n+1,-n+1)$ & ($n < 0$)
\\\hline
 \multirow{3}{*}{
 $\cdots$
     \hspace{2mm}
\begin{tikzpicture}[baseline=0mm,scale=0.7]
 \coordinate (0) at (0,0);
 \coordinate (1) at (-0.4,0);
 \coordinate (1') at (0.4,0);
 \coordinate (d) at (0,-1);
 \coordinate (r) at (1,0);
 \coordinate (l') at (-0.7,0);
 \draw (0) circle (1);
 \draw[blue] (l') .. controls (-0.6,0.5) and (0.1,0.5) .. (1');
 \draw[blue] (r) .. controls (0.5,-1) and (-0.7,-0.7) .. (l');
 \fill (1) circle (1mm); \fill (1') circle (1mm); \fill (d) circle (1mm);
 \node at(0,-0.3) {\color{blue}$r_{-1}$};
\end{tikzpicture}
     \hspace{2mm}
\begin{tikzpicture}[baseline=0mm,scale=0.7]
 \coordinate (0) at (0,0);
 \coordinate (1) at (-0.4,0);
 \coordinate (1') at (0.4,0);
 \coordinate (d) at (0,-1);
 \coordinate (r) at (1,0);
 \draw (0) circle (1);
 \draw[blue] (1') to node[above=6,left]{$r_0$} (r);
 \fill (1) circle (1mm); \fill (1') circle (1mm); \fill (d) circle (1mm);
\end{tikzpicture}
     \hspace{2mm}
\begin{tikzpicture}[baseline=0mm,scale=0.7]
 \coordinate (0) at (0,0);
 \coordinate (1) at (-0.4,0);
 \coordinate (1') at (0.4,0);
 \coordinate (d) at (0,-1);
 \coordinate (r) at (1,0);
 \coordinate (l') at (-0.7,0);
 \draw (0) circle (1);
 \draw[blue] (l') .. controls (-0.6,-0.5) and (0.1,-0.5) .. (1');
 \draw[blue] (r) .. controls (0.5,1) and (-0.7,0.7) .. (l');
 \fill (1) circle (1mm); \fill (1') circle (1mm); \fill (d) circle (1mm);
 \node at(0,0.3) {\color{blue}$r_{1}$};
\end{tikzpicture}
     \hspace{2mm}
 $\cdots$
}
& \multirow{3}{*}{$r_n$} & \multirow{2}{*}{$(1,n,n,n-1,n)$} & \multirow{2}{*}{($n > 0$)}\\
& & \multirow{2}{*}{$(1,-n,-n,-n,-n+1)$} & \multirow{2}{*}{($n \le 0$)}\\
& & &
\\\hline
 \multirow{3}{*}{
 $\cdots$
     \hspace{2mm}
\begin{tikzpicture}[baseline=0mm,scale=0.7]
 \coordinate (0) at (0,0);
 \coordinate (1) at (-0.4,0);
 \coordinate (1') at (0.4,0);
 \coordinate (d) at (0,-1);
 \coordinate (l) at (-1,0);
 \coordinate (r') at (0.7,0);
 \draw (0) circle (1);
 \draw[blue] (r') .. controls (0.6,-0.5) and (-0.1,-0.5) .. (1);
 \draw[blue] (l) .. controls (-0.5,1) and (0.7,0.7) .. (r');
 \fill (1) circle (1mm); \fill (1') circle (1mm); \fill (d) circle (1mm);
 \node at(0,0.3) {\color{blue}$l_{-1}$};
\end{tikzpicture}
     \hspace{2mm}
\begin{tikzpicture}[baseline=0mm,scale=0.7]
 \coordinate (0) at (0,0);
 \coordinate (1) at (-0.4,0);
 \coordinate (1') at (0.4,0);
 \coordinate (d) at (0,-1);
 \coordinate (l) at (-1,0);
 \draw (0) circle (1);
 \draw[blue] (1) to node[above=6,right]{$l_0$} (l);
 \fill (1) circle (1mm); \fill (1') circle (1mm); \fill (d) circle (1mm);
\end{tikzpicture}
     \hspace{2mm}
\begin{tikzpicture}[baseline=0mm,scale=0.7]
 \coordinate (0) at (0,0);
 \coordinate (1) at (-0.4,0);
 \coordinate (1') at (0.4,0);
 \coordinate (d) at (0,-1);
 \coordinate (l) at (-1,0);
 \coordinate (r') at (0.7,0);
 \draw (0) circle (1);
 \draw[blue] (r') .. controls (0.6,0.5) and (-0.1,0.5) .. (1);
 \draw[blue] (l) .. controls (-0.5,-1) and (0.7,-0.7) .. (r');
 \fill (1) circle (1mm); \fill (1') circle (1mm); \fill (d) circle (1mm);
 \node at(0,-0.3) {\color{blue}$l_{1}$};
\end{tikzpicture}
     \hspace{2mm}
 $\cdots$
}
& \multirow{3}{*}{$l_n$} & \multirow{2}{*}{$(1,n,n+1,n,n)$} & \multirow{2}{*}{($n \ge 0$)}\\
& & \multirow{2}{*}{$(1,-n-1,-n,-n,-n)$} & \multirow{2}{*}{($n < 0$)}\\
& & &
\\\hline
 \multirow{3}{*}{
 $\cdots$
     \hspace{2mm}
\begin{tikzpicture}[baseline=0mm,scale=0.7]
 \coordinate (0) at (0,0);
 \coordinate (1) at (-0.4,0);
 \coordinate (1') at (0.4,0);
 \coordinate (d) at (0,-1);
 \coordinate (l') at (-0.7,0);
 \draw (0) circle (1);
 \draw[blue] (d) .. controls (-0.7,-0.5) and (-0.7,-0.2) .. (l');
 \draw[blue] (l') .. controls (-0.6,0.5) and (0.1,0.5) .. (1');
 \fill (1) circle (1mm); \fill (1') circle (1mm); \fill (d) circle (1mm);
 \node at(0,-0.4) {\small\color{blue}$R_{-1}$};
\end{tikzpicture}
     \hspace{2mm}
$\Biggl($
\begin{tikzpicture}[baseline=-2mm,scale=0.5]
 \coordinate (0) at (0,0);
 \coordinate (1) at (-0.4,0);
 \coordinate (1') at (0.4,0);
 \coordinate (d) at (0,-1);
 \draw (0) circle (1);
 \draw[blue] (d) to (1');
 \fill (1) circle (1mm); \fill (1') circle (1mm); \fill (d) circle (1mm);
 \node[blue] at(0,-1.5) {\small$R_{0}$};
\end{tikzpicture}
$\Biggr)$
     \hspace{2mm}
\begin{tikzpicture}[baseline=0mm,scale=0.7]
 \coordinate (0) at (0,0);
 \coordinate (1) at (-0.4,0);
 \coordinate (1') at (0.4,0);
 \coordinate (d) at (0,-1);
 \coordinate (l') at (-0.7,0);
 \coordinate (r') at (0.7,0);
 \draw (0) circle (1);
 \draw[blue] (d) .. controls (0.7,-0.5) and (0.7,-0.2) .. (r');
 \draw[blue] (l') .. controls (-0.7,0.7) and (0.7,0.7) .. (r');
 \draw[blue] (l') .. controls (-0.6,-0.5) and (0.1,-0.5) .. (1');
 \fill (1) circle (1mm); \fill (1') circle (1mm); \fill (d) circle (1mm);
 \node[fill=white,inner sep=0.5] at(0,0.1) {\small\color{blue}$R_{1}$};
\end{tikzpicture}
     \hspace{2mm}
 $\cdots$
}
& \multirow{3}{*}{$R_n$} & \multirow{2}{*}{$(0,n,n,n-1,n)$} & \multirow{2}{*}{($n > 0$)}\\
& & \multirow{2}{*}{$(0,-n-1,-n-1,-n-1,-n)$} & \multirow{2}{*}{($n < 0$)}\\
& & &
\\\hline
 \multirow{3}{*}{
 $\cdots$
     \hspace{2mm}
\begin{tikzpicture}[baseline=0mm,scale=0.7]
 \coordinate (0) at (0,0);
 \coordinate (1) at (-0.4,0);
 \coordinate (1') at (0.4,0);
 \coordinate (d) at (0,-1);
 \coordinate (l') at (-0.7,0);
 \coordinate (r') at (0.7,0);
 \draw (0) circle (1);
 \draw[blue] (d) .. controls (-0.7,-0.5) and (-0.7,-0.2) .. (l');
 \draw[blue] (r') .. controls (0.7,0.7) and (-0.7,0.7) .. (l');
 \draw[blue] (r') .. controls (0.6,-0.5) and (-0.1,-0.5) .. (1);
 \fill (1) circle (1mm); \fill (1') circle (1mm); \fill (d) circle (1mm);
 \node[fill=white,inner sep=0.5] at(0,0.5) {\small\color{blue}$L_{-1}$};
\end{tikzpicture}
     \hspace{2mm}
$\Biggl($
\begin{tikzpicture}[baseline=-2mm,scale=0.5]
 \coordinate (0) at (0,0);
 \coordinate (1) at (-0.4,0);
 \coordinate (1') at (0.4,0);
 \coordinate (d) at (0,-1);
 \draw (0) circle (1);
 \draw[blue] (d) to (1);
 \fill (1) circle (1mm); \fill (1') circle (1mm); \fill (d) circle (1mm);
 \node[blue] at(0,-1.5) {\small$L_{0}$};
\end{tikzpicture}
$\Biggl)$
     \hspace{2mm}
\begin{tikzpicture}[baseline=0mm,scale=0.7]
 \coordinate (0) at (0,0);
 \coordinate (1) at (-0.4,0);
 \coordinate (1') at (0.4,0);
 \coordinate (d) at (0,-1);
 \coordinate (r') at (0.7,0);
 \draw (0) circle (1);
 \draw[blue] (d) .. controls (0.7,-0.5) and (0.7,-0.2) .. (r');
 \draw[blue] (r') .. controls (0.6,0.5) and (-0.1,0.5) .. (1);
 \fill (1) circle (1mm); \fill (1') circle (1mm); \fill (d) circle (1mm);
 \node[fill=white,inner sep=0.5] at(0,-0.4) {\small\color{blue}$L_{1}$};
\end{tikzpicture}
     \hspace{2mm}
 $\cdots$
}
& \multirow{3}{*}{$L_n$} & \multirow{2}{*}{$(0,n-1,n,n-1,n-1)$} & \multirow{2}{*}{($n > 0$)}\\
& & \multirow{2}{*}{$(0,-n-1,-n,-n,-n)$} & \multirow{2}{*}{($n < 0$)}\\
& & &
\\\hline
 \multirow{3}{*}{
 $\cdots$
     \hspace{2mm}
\begin{tikzpicture}[baseline=0mm,scale=0.7]
 \coordinate (0) at (0,0);
 \coordinate (1) at (-0.4,0);
 \coordinate (1') at (0.4,0);
 \coordinate (d) at (0,-1);
 \coordinate (l') at (-0.7,0);
 \draw (0) circle (1);
 \draw[blue] (d) .. controls (-0.7,-0.5) and (-0.7,-0.2) .. node[pos=0.1]{\rotatebox{50}{\footnotesize $\bowtie$}} (l');
 \draw[blue] (l') .. controls (-0.6,0.5) and (0.1,0.5) .. (1');
 \fill (1) circle (1mm); \fill (1') circle (1mm); \fill (d) circle (1mm);
 \node at(0,-0.4) {\small\color{blue}$\overline{R_{-1}}$};
\end{tikzpicture}
     \hspace{2mm}
\begin{tikzpicture}[baseline=0mm,scale=0.7]
 \coordinate (0) at (0,0);
 \coordinate (1) at (-0.4,0);
 \coordinate (1') at (0.4,0);
 \coordinate (d) at (0,-1);
 \draw (0) circle (1);
 \draw[blue] (d) to  node[pos=0.2]{\rotatebox{-25}{\footnotesize $\bowtie$}} (1');
 \fill (1) circle (1mm); \fill (1') circle (1mm); \fill (d) circle (1mm);
 \node at(0) {\small\color{blue}$\overline{R_{0}}$};
\end{tikzpicture}
     \hspace{2mm}
\begin{tikzpicture}[baseline=0mm,scale=0.7]
 \coordinate (0) at (0,0);
 \coordinate (1) at (-0.4,0);
 \coordinate (1') at (0.4,0);
 \coordinate (d) at (0,-1);
 \coordinate (l') at (-0.7,0);
 \coordinate (r') at (0.7,0);
 \draw (0) circle (1);
 \draw[blue] (d) .. controls (0.7,-0.5) and (0.7,-0.2) .. node[pos=0.1]{\rotatebox{-50}{\footnotesize $\bowtie$}} (r');
 \draw[blue] (l') .. controls (-0.7,0.7) and (0.7,0.7) .. (r');
 \draw[blue] (l') .. controls (-0.6,-0.5) and (0.1,-0.5) .. (1');
 \fill (1) circle (1mm); \fill (1') circle (1mm); \fill (d) circle (1mm);
 \node at(0,0.1) {\small\color{blue}$\overline{R_{1}}$};
\end{tikzpicture}
     \hspace{2mm}
 $\cdots$
}
& \multirow{3}{*}{$\overline{R_n}$} & $(2,1,1,0,2)$ & ($n = 0$)\\
& & $(2,n+1,n+1,n,n+1)$ & ($n > 0$)\\
& & $(2,-n,-n,-n,-n+1)$ & ($n < 0$)
\\\hline
 \multirow{3}{*}{
 $\cdots$
     \hspace{2mm}
\begin{tikzpicture}[baseline=0mm,scale=0.7]
 \coordinate (0) at (0,0);
 \coordinate (1) at (-0.4,0);
 \coordinate (1') at (0.4,0);
 \coordinate (d) at (0,-1);
 \coordinate (l') at (-0.7,0);
 \coordinate (r') at (0.7,0);
 \draw (0) circle (1);
 \draw[blue] (d) .. controls (-0.7,-0.5) and (-0.7,-0.2) .. node[pos=0.1]{\rotatebox{50}{\footnotesize $\bowtie$}} (l');
 \draw[blue] (r') .. controls (0.7,0.7) and (-0.7,0.7) .. (l');
 \draw[blue] (r') .. controls (0.6,-0.5) and (-0.1,-0.5) .. (1);
 \fill (1) circle (1mm); \fill (1') circle (1mm); \fill (d) circle (1mm);
 \node[fill=white,inner sep=0.5] at(0,0.5) {\small\color{blue}$\overline{L_{-1}}$};
\end{tikzpicture}
     \hspace{2mm}
\begin{tikzpicture}[baseline=0mm,scale=0.7]
 \coordinate (0) at (0,0);
 \coordinate (1) at (-0.4,0);
 \coordinate (1') at (0.4,0);
 \coordinate (d) at (0,-1);
 \draw (0) circle (1);
 \draw[blue] (d) to node[pos=0.2]{\rotatebox{25}{\footnotesize $\bowtie$}} (1);
 \fill (1) circle (1mm); \fill (1') circle (1mm); \fill (d) circle (1mm);
 \node at(0) {\small\color{blue}$\overline{L_{0}}$};
\end{tikzpicture}
     \hspace{2mm}
\begin{tikzpicture}[baseline=0mm,scale=0.7]
 \coordinate (0) at (0,0);
 \coordinate (1) at (-0.4,0);
 \coordinate (1') at (0.4,0);
 \coordinate (d) at (0,-1);
 \coordinate (r') at (0.7,0);
 \draw (0) circle (1);
 \draw[blue] (d) .. controls (0.7,-0.5) and (0.7,-0.2) .. node[pos=0.1]{\rotatebox{-50}{\footnotesize $\bowtie$}} (r');
 \draw[blue] (r') .. controls (0.6,0.5) and (-0.1,0.5) .. (1);
 \fill (1) circle (1mm); \fill (1') circle (1mm); \fill (d) circle (1mm);
 \node[fill=white,inner sep=0.5] at(0,-0.4) {\small\color{blue}$\overline{L_{1}}$};
\end{tikzpicture}
     \hspace{2mm}
 $\cdots$
}
& \multirow{3}{*}{$\overline{L_n}$} & $(2,0,2,1,1)$ & ($n = 0$)\\
& & $(2,n,n+1,n,n)$ & ($n > 0$)\\
& & $(2,-n,-n+1,-n+1,-n+1)$ & ($n < 0$)
\\\hline
\end{tabular}
\vspace{1mm}
\[
\begin{tikzpicture}[baseline=0mm,scale=0.7]
 \coordinate (0) at (0,0);
 \coordinate (1) at (-0.4,0);
 \coordinate (1') at (0.4,0);
 \coordinate (d) at (0,-1);
 \coordinate (l) at (-0.9,-0.4);
 \coordinate (r) at (0.9,-0.4);
 \draw (0) circle (1);
 \draw[blue] (l) to node[fill=white,inner sep=0.5]{\color{blue}$h$} (r);
 \fill (1) circle (1mm); \fill (1') circle (1mm); \fill (d) circle (1mm);
\end{tikzpicture}
   \hspace{2mm} (2,1,1,1,1) \hspace{7mm}
\begin{tikzpicture}[baseline=0mm,scale=0.7]
 \coordinate (0) at (0,0);
 \coordinate (1) at (-0.4,0);
 \coordinate (1') at (0.4,0);
 \coordinate (d) at (0,-1);
 \draw (0) circle (1);
 \draw[blue] (1') to node[above]{\color{blue}$H$} (1);
 \fill (1) circle (1mm); \fill (1') circle (1mm); \fill (d) circle (1mm);
\end{tikzpicture}
   \hspace{2mm} (0,0,1,0,1) \hspace{7mm}
\begin{tikzpicture}[baseline=0mm,scale=0.7]
 \coordinate (0) at (0,0);
 \coordinate (1) at (-0.4,0);
 \coordinate (1') at (0.4,0);
 \coordinate (d) at (0,-1);
 \draw (0) circle (1);
 \draw[blue] (d)  .. controls (-1,-0.5) and (-1,0.5) .. node[pos=0.1]{\rotatebox{50}{\footnotesize $\bowtie$}} (0,0.5);
 \draw[blue] (d)  .. controls (1,-0.5) and (1,0.5) .. node[pos=0.1]{\rotatebox{-50}{\footnotesize $\bowtie$}} (0,0.5);
 \fill (1) circle (1mm); \fill (1') circle (1mm); \fill (d) circle (1mm);
 \node at(0,0.1) {\color{blue}$Q$};
\end{tikzpicture}
   \hspace{2mm} (4,2,2,2,2)
\]\vspace{5mm}
 \caption{Segments of a tagged arc in $2$-puncture pieces and the corresponding intersection sub-vectors $(a_1,a_2,a_3,a_4,a_5)$, where $s_n$ (resp., $c_n$, $\overline{c_n}$) is obtained from $s_0$ (resp., $c_0$, $\overline{c_0}$) by moving its endpoints along the boundary clockwise in angle $\pi$ and the other cases are in angle $2\pi$}\label{S2}
\end{table}}
\newpage

\renewcommand{\arraystretch}{1.6}
{\begin{table}[!p]
\begin{tabular}{c|c|c}
 \multicolumn{3}{c}{Case($-$)}\\\hline
 Sets of segments & Sets of modified segments & $(a_1,a_2,a_3)$\\\hline\hline
 \multicolumn{2}{c|}{$\{m_1e_1,m_2e_2,m_3e_3\}$} & $(m_2+m_3, m_3+m_1, m_1+m_2)$\\\hline
 \multicolumn{2}{c|}{$\{h_1,m_2e_2,m_3e_3\}$} & $(1+m_2+m_3, m_3, m_2)$\\\hline
 \multicolumn{2}{c|}{$\{2h_1,m_2e_2,m_3e_3\}$} & $(2+m_2+m_3, m_3, m_2)$\\\hline
 $\{\overline{h_1},m_2e_2,m_3e_3\}$ & $\{y,m_2e_2,m_3e_3\}$ & $(1+m_2+m_3, 1+m_3, 1+m_2)$\\\hline
 $\{\overline{h_1},e_1,m_2e_2,m_3e_3\}$ & $\{y,e_1,m_2e_2,m_3e_3\}$ & $(1+m_2+m_3, 2+m_3, 2+m_2)$\\\hline
 $\{2\overline{h_1},m_2e_2,m_3e_3\}$ & $\{e_1,(m_2+1)e_2,(m_3+1)e_3\}$ & \multirow{7}{*}{appear in above}\\\cline{1-2}
 $\{\overline{h_2},k_1e_1,k_2e_2,k_3e_3\}$ & \multirow{2}{*}{$\{y,k_1e_1,k_2e_2,k_3e_3\}$}\\\cline{1-1}
 $\{\overline{h_3},k_1e_1,k_2e_2,k_3e_3\}$ & &\\\cline{1-2}
 $\{2\overline{h_2},m_3e_3\}$ & $\{e_1,e_2,(1+m_3)e_3\}$\\\cline{1-2}
 $\{2\overline{h_3},m_2e_2\}$ & $\{e_1,(1+m_2)e_2,e_3\}$\\\cline{1-2}
 $\{\overline{E_{3r}}\}$, $\{\overline{E_{2l}}\}$ & $\{h_1\}$\\\cline{1-2}
 $\{\overline{\overline{E_1}}\}$ & $\{e_2,e_3\}$\\\hline
\end{tabular}\\\vspace{7mm}
 Case($\tau_i$) and Case($\tau_i$,$\tau_j$) come down to Case($-$) as follows:\\
\vspace{2mm}
\begin{tabular}{c|c|c|c}
 \multicolumn{4}{c}{Case($\tau_i$) and not appear in Case($-$)}\\\hline
 i & Sets of segments & Sets of modified segments & $(a_1,a_2,a_3)$\\\hline\hline
 1 & $\{\overline{E_{2l}},e_2\}$ & $\{h_1,e_2\}$ & \multirow{6}{*}{appear in above}\\\cline{1-3}
 1 & $\{\overline{E_{3r}},e_3\}$ & $\{h_1,e_3\}$\\\cline{1-3}
 1 & $\{\overline{\overline{E_1}},e_2,e_3\}$ & $\{2e_2,2e_3\}$\\\cline{1-3}
 2, 3 & $\{\overline{\overline{E_1}},e_1\}$ & \multirow{3}{*}{$\{e_1,e_2,e_3\}$}\\\cline{1-2}
 1, 3 & $\{\overline{\overline{E_2}},e_2\}$ &\\\cline{1-2}
 1, 2 & $\{\overline{\overline{E_3}},e_3\}$ &\\\hline
\end{tabular}\\\vspace{3mm}
\begin{tabular}{c|c|c}
 \multicolumn{3}{c}{Case($\tau_i$,$\tau_j$) and not appear in others}\\\hline
 Sets of segments & Sets of modified segments & $(a_1,a_2,a_3)$\\\hline\hline
 $\{\overline{\overline{E_k}},e_1,e_2,e_3\}$ & $\{2e_1,2e_2,2e_3\}$ & appear in above\\\hline
\end{tabular}\vspace{5mm}
 \caption{Sets of segments and the corresponding sets of modified segments in a triangle piece for the case of $a_1 \ge a_2, a_3$, where $m_i \in \bZ_{\ge 0}$ and $k_i \in \{0,1\}$ such that $k_1 \le k_2, k_3$}\label{ss0}
\end{table}}

\renewcommand{\arraystretch}{1.5}
{\begin{table}[!p]
\scalebox{0.99}{\begin{tabular}{c|c|c}
 \multicolumn{3}{c}{Case($-$)}\\\hline
 Sets of segments & Sets of modified segments & $(a_1,a_2,a_3,a_4)$\\\hline\hline
 \multicolumn{2}{c|}{$\{m_1r,m_2u,m_3d\}$} & $(m_2+m_3,2m_1+m_2+m_3,m_1+m_3,m_1+m_3)$\\\hline
 \multicolumn{2}{c|}{$\{r_+,m_1r,m_3d\}$} & $(m_3,1+2m_1+m_3,1+m_1+m_3,1+m_1+m_3)$\\\hline
 \multicolumn{2}{c|}{$\{2r_+,m_1r,m_3d\}$} & $(m_3,2+2m_1+m_3,2+m_1+m_3,2+m_1+m_3)$\\\hline
 \multicolumn{2}{c|}{$\{r_-,m_1r,m_2u\}$} & $(m_2,1+2m_1+m_2,m_1,m_1)$\\\hline
 \multicolumn{2}{c|}{$\{2r_-,m_1r,m_2u\}$} & $(m_2,2+2m_1+m_2,m_1,m_1)$\\\hline

 $\{\overline{r_+},m_1r,m_3d\}$ & $\{y_r,m_1r,m_3d\}$ & $(1+m_3,2+2m_1+m_3,1+m_1+m_3,1+m_1+m_3)$\\\hline
 $\{\overline{r_+},u,m_1r,m_3d\}$ & $\{y_r,u,m_1r,m_3d\}$ & $(2+m_3,3+2m_1+m_3,1+m_1+m_3,1+m_1+m_3)$\\\hline

 $\{\overline{r_-},m_1r,m_2u\}$ & $\{y_r,m_1r,m_2u\}$ & $(1+m_2,2+2m_1+m_2,1+m_1,1+m_1)$\\\hline
 $\{\overline{r_-},d,m_1r,m_2u\}$ & $\{y_r,d,m_1r,m_2u\}$ & $(2+m_2,3+2m_1+m_2,2+m_1,2+m_1)$\\\hline

 \multicolumn{2}{c|}{$\{r_p,m_1r,m_2u,m_3d\}$} & $(m_2+m_3,1+2m_1+m_2+m_3,m_1+m_3,1+m_1+m_3)$\\\hline
 \multicolumn{2}{c|}{$\{r_p,r_+,m_1r,m_3d\}$} & $(m_3,2+2m_1+m_3,1+m_1+m_3,2+m_1+m_3)$\\\hline
 \multicolumn{2}{c|}{$\{r_p,r_-,m_1r,m_2u\}$} & $(m_2,2+2m_1+m_2,m_1,1+m_1)$\\\hline

 $\{r_p,\overline{r_+},m_1r,m_3d\}$ & $\{r_p,y_r,m_1r,m_3d\}$ & $(1+m_3,3+2m_1+m_3,1+m_1+m_3,2+m_1+m_3)$\\\hline
 $\{r_p,\overline{r_+},u,m_1r,m_3d\}$ & $\{r_p,y_r,u,m_1r,m_3d\}$ & $(2+m_3,4+2m_1+m_3,1+m_1+m_3,2+m_1+m_3)$\\\hline
 $\{r_p,\overline{r_-},m_1r,m_2u\}$ & $\{r_p,y_r,m_1r,m_2u\}$ & $(1+m_2,3+2m_1+m_2,1+m_1,2+m_1)$\\\hline
 $\{r_p,\overline{r_-},d,m_1r,m_2u\}$ & $\{r_p,y_r,d,m_1r,m_2u\}$ & $(2+m_2,4+2m_1+m_2,2+m_1,3+m_1)$\\\hline

 \multicolumn{2}{c|}{$\{P_+\}$} & $(0,0,0,1)$\\\hline
 \multicolumn{2}{c|}{$\{\overline{P_+}\}$} & $(1,1,0,1)$\\\hline

 \multicolumn{2}{c|}{$\{2r_p,m_1r,m_2u,m_3d\}$} & $(m_2+m_3,2+2m_1+m_2+m_3,m_1+m_3,2+m_1+m_3)$\\\hline
 \multicolumn{2}{c|}{$\{r_p,l_p,m_2u,m_3d\}$} & $(1+m_2+m_3,1+m_2+m_3,m_3,2+m_3)$\\\hline

 $\{2\overline{r_+},m_1r,m_3d\}$ & $\{(m_1+2)r,2u,(m_3+2)d\}$ & \multirow{6}{*}{appear in above}\\\cline{1-2}
 $\{2\overline{r_-},m_1r,m_2u\}$ & $\{(m_1+2)r,(m_2+2)u,2d\}$ &\\\cline{1-2}
 $\{\overline{P_-}\}$ & $\{r_p,l_p\}$ &\\\cline{1-2}
 $\{\underline{R}\}$ & $\{r_+\}$ &\\\cline{1-2}
 $\{\overline{R}\}$ & $\{r_-\}$ &\\\cline{1-2}
 $\{\overline{\overline{R}}\}$, $\{\overline{\overline{L}}\}$ & $\{u,d\}$ &\\\hline

\multicolumn{3}{c}{}\\
 \multicolumn{3}{c}{Case($\tau_i$) and not appear in Case($-$)}\\\hline\hline
 \multicolumn{2}{c|}{$\{\overline{P_+},d\}$} & $(2,2,1,2)$\\\hline
 $\{\overline{P_-},u\}$ & $\{r_p,l_p,u\}$ & \multirow{2}{*}{appear in above}\\\cline{1-2}
 $\{\overline{\overline{R}},u,d\}$, $\{\overline{\overline{L}},u,d\}$ & $\{2u,2d\}$ &\\\hline
\end{tabular}}\vspace{5mm}
 \caption{Sets of segments and the corresponding sets of modified segments in a $1$-puncture piece for the case of $a_1 \le a_2$, where $m_i \in \bZ_{\ge 0}$}\label{ss1}
\end{table}}

\renewcommand{\arraystretch}{1.4}
{\begin{table}[ht]
 \caption{Sets of segments and the corresponding sets of modified segments in a $2$-puncture piece for the case of $a_3-a_2 \le a_5-a_4$, where $m_i \in \bZ_{\ge 0}$ and the notation $(\uparrow)$ (resp., $(\leftarrow)$) means that it is equal to the polynomial just above (resp., left)}\label{ss2}
\scalebox{0.88}{\begin{tabular}{l}
 Sets of segments = Sets of modified segments\\\hline
 $(a_1,a_2,a_3,a_4,a_5)$\\\hline\hline

 $\{c_n,m_1s_{n-1},m_2s_n\}$\\\hline
$\begin{array}{ll}
 (1+2(m_1+m_2),m_2,m_2,m_1,m_1) & (n=0)\\
 ((\uparrow),n+m_1(n)+m_2(n+1),(\leftarrow),n-1+m_1(n-1)+m_2(n),(\leftarrow)) & (n>0)\\
 ((\uparrow),-n-1+m_1(-n)+m_2(-n-1),(\leftarrow),-n+m_1(-n+1)+m_2(-n),(\leftarrow)) & (n<0)
\end{array}$\\\hline\hline

 $\{\overline{c_n},m_1s_{n-1},m_2s_n\}$\\\hline
$\begin{array}{ll}
 (3+2(m_1+m_2),1+m_2,1+m_2,1+m_1,1+m_1) & (n=0)\\
 ((\uparrow),n+1+m_1(n)+m_2(n+1),(\leftarrow),n+m_1(n-1)+m_2(n),(\leftarrow)) & (n>0)\\
 ((\uparrow),-n+m_1(-n)+m_2(-n-1),(\leftarrow),-n+1+m_1(-n+1)+m_2(-n),(\leftarrow)) & (n<0)
\end{array}$\\\hline\hline

 $\{\overline{c_n},m_1s_{n-1},m_2s_n,h\}$\\\hline
$\begin{array}{ll}
 (5+2(m_1+m_2),2+m_2,(\leftarrow),2+m_1,(\leftarrow)) & (n=0)\\
 ((\uparrow),n+2+m_1(n)+m_2(n+1),(\leftarrow),n+1+m_1(n-1)+m_2(n),(\leftarrow)) & (n>0)\\
 ((\uparrow),-n+1+m_1(-n)+m_2(-n-1),(\leftarrow),-n+2+m_1(-n+1)+m_2(-n),(\leftarrow)) & (n<0)
\end{array}$\\\hline\hline

 $\{r_n,m_1s_{2n-2},m_2s_{2n-1},m_3h\}$ ($m_1 \neq 0$)\\\hline
$\begin{array}{ll}
 (1+2(m_1+m_2+m_3),m_1+m_3,(\leftarrow),2m_1+m_2+m_3,1+2m_1+m_2+m_3) & (n=0)\\
 ((\uparrow),n+m_1(2n-1)+m_2(2n)+m_3,(\leftarrow),n-1+m_1(2n-2)+m_2(2n-1)+m_3,n+(\leftarrow)) & (n>0)\\
 ((\uparrow),-n+m_1(-2n+1)+m_2(-2n)+m_3,(\leftarrow),-n+m_1(-2n+2)+m_2(-2n+1)+m_3,-n+1+(\leftarrow)) & (n<0)
\end{array}$\\\hline\hline

 $\{r_n,m_1s_{2n-1},m_2s_{2n},m_3h\}$\\\hline
$\begin{array}{ll}
 (1+2(m_1+m_2+m_3),m_2+m_3,(\leftarrow),m_1+m_3,1+m_1+m_3) & (n=0)\\
 ((\uparrow),n+m_1(2n)+m_2(2n+1)+m_3,(\leftarrow),n-1+m_1(2n-1)+m_2(2n)+m_3,n+(\leftarrow)) & (n>0)\\
 ((\uparrow),-n+m_1(-2n)+m_2(-2n-1)+m_3,(\leftarrow),-n+m_1(-2n+1)+m_2(-2n)+m_3,-n+1+(\leftarrow)) & (n<0)
\end{array}$\\\hline\hline

 $\{2r_n,m_1s_{2n-2},m_2s_{2n-1},m_3h\}$ ($m_1 \neq 0$)\\\hline
$\begin{array}{ll}
 (2+2(m_1+m_2+m_3),m_1+m_3,(\leftarrow),2m_1+m_2+m_3,2+2m_1+m_2+m_3) & (n=0)\\
 ((\uparrow),2n+m_1(2n-1)+m_2(2n)+m_3,(\leftarrow),2n-2+m_1(2n-2)+m_2(2n-1)+m_3,2n+(\leftarrow)) & (n>0)\\
 ((\uparrow),-2n+m_1(-2n+1)+m_2(-2n)+m_3,(\leftarrow),-2n+m_1(-2n+2)+m_2(-2n+1)+m_3,-2n+2+(\leftarrow)) & (n<0)
\end{array}$\\\hline\hline

 $\{2r_n,m_1s_{2n-1},m_2s_{2n},m_3h\}$\\\hline
$\begin{array}{ll}
 (2+2(m_1+m_2+m_3),m_2+m_3,(\leftarrow),m_1+m_3,2+m_1+m_3) & (n=0)\\
 ((\uparrow),2n+m_1(2n)+m_2(2n+1)+m_3,(\leftarrow),2n-2+m_1(2n-1)+m_2(2n)+m_3,2n+(\leftarrow)) & (n>0)\\
 ((\uparrow),-2n+m_1(-2n)+m_2(-2n-1)+m_3,(\leftarrow),-2n+m_1(-2n+1)+m_2(-2n)+m_3,-2n+2+(\leftarrow)) & (n<0)
\end{array}$\\\hline\hline

 $\{r_n,r_{n+1},m_1s_{2n},m_3h\}$\\\hline
$\begin{array}{ll}
 (2+2(m_1+m_3),1+m_1+m_3,(\leftarrow),m_3,2+m_3) & (n=0)\\
 ((\uparrow),1+m_1+m_3,(\leftarrow),1+2m_1+m_3,3+2m_1+m_3) & (n=-1)\\
 ((\uparrow),2n+1+m_1(2n+1)+m_3,(\leftarrow),2n-1+m_1(2n)+m_3,2n+1+(\leftarrow)) & (n>0)\\
 ((\uparrow),-2n-1+m_1(-2n-1)+m_3,(\leftarrow),-2n-1+m_1(-2n)+m_3,-2n+1+(\leftarrow)) & (n<-1)
\end{array}$\\\hline\hline

 $\{\overline{R_n}\}$ ($n \neq 0$)\\\hline
$\begin{array}{ll}
 (2,n+1,n+1,n,n+1) & (n>0)\\
 (2,-n,-n,-n,-n+1) & (n<0)
\end{array}$\\\hline\hline
\end{tabular}}
\end{table}}
\renewcommand{\arraystretch}{1.4}
{\begin{table}[ht]
\scalebox{0.85}{
\begin{tabular}{l}\hline\hline
 $\{r_n,c_{2n-1},m_1s_{2n-2},m_2s_{2n-1}\}$\\\hline
$\begin{array}{ll}
 (2+2(m_1+m_2),m_1,(\leftarrow),1+2m_1+m_2,2+2m_1+m_2) & (n=0)\\
 ((\uparrow),3n-1+m_1(2n-1)+m_2(2n),(\leftarrow),3n-3+m_1(2n-2)+m_2(2n-1),3n-2+(\leftarrow)) & (n>0)\\
 ((\uparrow),-3n+m_1(-2n+1)+m_2(-2n),(\leftarrow),-3n+1+m_1(-2n+2)+m_2(-2n+1),-3n+2+(\leftarrow)) & (n<0)
\end{array}$\\\hline\hline

 $\{r_n,c_{2n},m_1s_{2n-1},m_2s_{2n}\}$\\\hline
$\begin{array}{ll}
 (2+2(m_1+m_2),m_2,(\leftarrow),m_1,1+m_1) & (n=0)\\
 ((\uparrow),3n+m_1(2n)+m_2(2n+1),(\leftarrow),3n-2+m_1(2n-1)+m_2(2n),3n-1+(\leftarrow)) & (n>0)\\
 ((\uparrow),-3n-1+m_1(-2n)+m_2(-2n-1),(\leftarrow),-3n+m_1(-2n+1)+m_2(-2n),-3n+1+(\leftarrow)) & (n<0)
\end{array}$\\\hline\hline

 $\{r_n,\overline{c_{2n-1}},m_1s_{2n-2},m_2s_{2n-1}\}$\\\hline
$\begin{array}{ll}
 (4+2(m_1+m_2),1+m_1,(\leftarrow),2+2m_1+m_2,3+2m_1+m_2) & (n=0)\\
 ((\uparrow),3n+m_1(2n-1)+m_2(2n),(\leftarrow),3n-2+m_1(2n-2)+m_2(2n-1),3n-1+(\leftarrow)) & (n>0)\\
 ((\uparrow),-3n+1+m_1(-2n+1)+m_2(-2n),(\leftarrow),-3n+2+m_1(-2n+2)+m_2(-2n+1),-3n+3+(\leftarrow)) & (n<0)
\end{array}$\\\hline\hline

 $\{r_n,\overline{c_{2n}},m_1s_{2n-1},m_2s_{2n}\}$\\\hline
$\begin{array}{ll}
 (4+2(m_1+m_2),1+m_2,(\leftarrow),1+m_1,2+m_1) & (n=0)\\
 ((\uparrow),3n+1+m_1(2n)+m_2(2n+1),(\leftarrow),3n-1+m_1(2n-1)+m_2(2n),3n+(\leftarrow)) & (n>0)\\
 ((\uparrow),-3n+m_1(-2n)+m_2(-2n-1),(\leftarrow),-3n+1+m_1(-2n+1)+m_2(-2n),-3n+2+(\leftarrow)) & (n<0)
\end{array}$\\\hline\hline

 $\{r_n,l_{n-1},m_1s_{2n-2},m_2s_{2n-1},m_3h\}$\\\hline
$\small\begin{array}{ll}
 (2+2(m_1+m_2+m_3),m_1+m_3,1+m_1+m_3,1+2m_1+m_2+m_3,2+2m_1+m_2+m_3) & (n=0)\\
 ((\uparrow),2n-1+m_1(2n-1)+m_2(2n)+m_3,2n+(\leftarrow),2n-2+m_1(2n-2)+m_2(2n-1)+m_3,2n-1+(\leftarrow)) & (n>0)\\
 ((\uparrow),-2n+m_1(-2n+1)+m_2(-2n)+m_3,-2n+1+(\leftarrow),-2n+1+m_1(-2n+2)+m_2(-2n+1)+m_3,-2n+2+(\leftarrow)) & (n<0)
\end{array}$\\\hline\hline

 $\{r_n,l_n,m_1s_{2n-1},m_2s_{2n},m_3h\}$\\\hline
$\begin{array}{ll}
 (2+2(m_1+m_2+m_3),m_2+m_3,1+m_2+m_3,m_1+m_3,1+m_1+m_3) & (n=0)\\
 ((\uparrow),2n+m_1(2n)+m_2(2n+1)+m_3,2n+1+(\leftarrow),2n-1+m_1(2n-1)+m_2(2n)+m_3,2n+(\leftarrow)) & (n>0)\\
 ((\uparrow),-2n-1+m_1(-2n)+m_2(-2n-1)+m_3,-2n+(\leftarrow),-2n+m_1(-2n+1)+m_2(-2n)+m_3,-2n+1+(\leftarrow)) & (n<0)
\end{array}$\\\hline\hline

 $\{2c_n,m_1s_{n-1},m_2s_n\}$\\\hline
$\begin{array}{ll}
 (2+2(m_1+m_2),m_2,m_2,m_1,m_1) & (n=0)\\
 ((\uparrow),2n+m_1(n)+m_2(n+1),(\leftarrow),2n-2+m_1(n-1)+m_2(n),(\leftarrow)) & (n>0)\\
 ((\uparrow),-2n-2+m_1(-n)+m_2(-n-1),(\leftarrow),-2n+m_1(-n+1)+m_2(-n),(\leftarrow)) & (n<0)
\end{array}$\\\hline\hline

 $\{m_1s_{n-1},m_2s_n,m_3h\}$\\\hline
$\begin{array}{ll}
 (2(m_1+m_2+m_3),m_2+m_3,(\leftarrow),m_1+m_3,(\leftarrow)) & (n=0)\\
 ((\uparrow),m_1(n)+m_2(n+1)+m_3,(\leftarrow),m_1(n-1)+m_2(n)+m_3,(\leftarrow)) & (n>0)\\
 ((\uparrow),m_1(-n)+m_2(-n-1)+m_3,(\leftarrow),m_1(-n+1)+m_2(-n)+m_3,(\leftarrow)) & (n<0)
\end{array}$\\\hline\hline

 $\{H\}$\\\hline
 $(0,0,1,0,1)$\\\hline\hline

 $\{R_n\}$\\\hline
$\begin{array}{ll}
 (0,n,n,n-1,n) & (n>0)\\
 (0,-n-1,-n-1,-n-1,-n) & (n<0)
\end{array}$\\\hline\hline
\end{tabular}
}
\\\vspace{5mm}
 The other cases are modified to the above cases as follows:
\\\vspace{3mm}
\begin{tabular}{c||c|c|c}\hline\hline
 Sets of segments & $\{2\overline{c_n},m_1s_{n-1},m_2s_n\}$ & $\{Q\}$ & $\{\overline{R_0}\}$\\\hline
 Sets of modified segments & $\{(m_1+1)s_{n-1},(m_2+1)s_n,h\}$ & $\{2h\}$ & $\{r_0,r_1\}$\\\hline\hline
\end{tabular}\\\vspace{3mm}
\end{table}}

\clearpage
\section{Proof of Theorem \ref{intinj}}\label{pfintinj}

 Let $T$ be a tagged triangulation of $(S,M)$. To prove Theorem \ref{intinj}, we can assume that $T$ satisfies $(\Diamond)$. Let $v \in \bZ_{\ge 0}^n$ be an intersection vector with respect to $T$. We show that there is a unique modified tagged arc ${\sf m}$ such that $\Int(T,{\sf m})=v$.

 First of all, we assume that $T$ is not $T_3$. We only need to show that, for any puzzle piece $\square$, there is a unique set of modified segments $S=S_{\square}$ in $\square$ such that $\sum_{s \in S}\Int(s,\square)=v|_{\square}$. Indeed, gluing puzzle pieces of $T$, their segments are glued simultaneously. Then we can obtain ${\sf m}$.

 First, we consider the case that $\square$ is a triangle piece. That is, $v|_{\square}=(a_1,a_2,a_3)$. By symmetry, we can assume that $a_1 \ge a_2, a_3$. We consider the simultaneous equations
\begin{align*}
\begin{cases}
 m_2+m_3=a_1\\
 m_3+m_1=a_2\\
 m_1+m_2=a_3.\\
\end{cases}
\end{align*}
 If $(m_1,m_2,m_3)\in\bZ^3_{\ge 0}$, then we have $S=\{m_1e_1, m_2e_2,m_3e_3\}$. Now, we assume that $(m_1,m_2,m_3)\notin\bZ^3_{\geq0}$.
\begin{itemize}\setlength{\leftskip}{-5mm}
 \item If $a_2+a_3=a_1-1$, then we have $S=\{h_1,a_3e_2,a_2e_3\}$.
 \item If $a_2+a_3=a_1-2$, then we have $S=\{2h_1,a_3e_2,a_2e_3\}$.
 \item If $a_2+a_3=a_1+1$, then we have $S=\{y,(a_3-1)e_2,(a_2-1)e_3\}$.
 \item If $a_2+a_3=a_1+2$, then we have $S=\{y,e_1,(a_3-2)e_2,(a_2-2)e_3\}$.
\end{itemize}
 By Table \ref{ss0}, these cover all cases of $(a_1,a_2,a_3)$. Therefore, $v|_{\square}$ gives the unique set of modified segments $S_{\square}$.

 Second, we consider the case that $\square$ is a $1$-puncture piece. That is, $v|_{\square}=(a_1,a_2,a_3,a_4)$. By symmetry, we can assume that $a_1 \le a_2$.

\begin{itemize}\setlength{\leftskip}{-3mm}
\item[a)] Suppose that $a_3=a_4$. In this case, $S$ is one as in Table \ref{case1}.
\renewcommand{\arraystretch}{1.4}
{\begin{table}[ht]
\begin{tabular}{c|c|c}
$S$ & $a_1-a_2+2a_3$ & $a_1+a_2-2a_3$\\\hline\hline
 $\{m_1r,m_2u, m_3d\}$ & $2m_3$ & $2m_2$\\\hline
 $\{r_+, m_1r, m_3d\}$ & $1+2m_3$ & $-1$\\\hline
 $\{2r_+, m_1r, m_3d\}$ & $2+2m_3$ & $-2$\\\hline
 $\{r_-, m_1r, m_2u\}$ & $-1$ & $1+2m_2$\\\hline
 $\{2r_-, m_1r, m_2u\}$ & $-2$ & $2+2m_2$\\\hline
 $\{y_r, m_1r, m_3d\}$ & $1+2m_3$ & $1$\\\hline
 $\{y_r, u, m_1r, m_3d\}$ & $1+2m_3$ & $3$\\\hline
 $\{y_r, m_1r, m_2u\}$ & $1$ & $1+2m_2$\\\hline
 $\{y_r, d, m_1r, m_2u\}$ & $3$ & $1+2m_2$\\\hline
\end{tabular}\vspace{5mm}
 \caption{All cases of $S$ for a $1$-puncture piece and $a_3=a_4$}\label{case1}
\end{table}}
\begin{itemize}\setlength{\leftskip}{-5mm}
 \item[a1)] If $a_1-a_2+2a_3=-2$, then $S=\{2r_-, a_3r, a_1u\}$.

 \item[a2)] If $a_1-a_2+2a_3=-1$, then $S=\{r_-, a_3r, a_1u\}$.

 \item[a3)] If $a_1-a_2+2a_3=0$, then $S=\{a_3r,a_1u\}$.

 \item[a4)] If $a_1-a_2+2a_3=1$, then $S$ is either $\{r_+, m_1r\}$ or $\{y_r, m_1r, m_2u\}$.
\begin{itemize}\setlength{\leftskip}{-11mm}
 \item If $a_1+a_2-2a_3=-1$, then $S=\{r_+, (a_3-1)r\}$.
 \item If $a_1+a_2-2a_3 \neq -1$, then $S=\{y_r, (a_3-1)r, (a_1-1)u\}$.
\end{itemize}

 \item[a5)] If $a_1-a_2+2a_3=2$, then $S$ is either $\{m_1r,m_2u, d\}$ or $\{2r_+, m_1r\}$.
\begin{itemize}\setlength{\leftskip}{-11mm}
 \item If $a_1+a_2-2a_3=-2$, then $S=\{2r_+, (a_3-2)r\}$.
 \item If $a_1+a_2-2a_3 \neq -2$, then $S=\{(a_3-1)r,(a_1-1)u, d\}$.
\end{itemize}

 \item[a6)] If $a_1-a_2+2a_3=3$, then $S$ is either $\{r_+, m_1r, d\}$ or $\{y_r, d, m_1r, m_2u\}$.
\begin{itemize}\setlength{\leftskip}{-11mm}
 \item If $a_1+a_2-2a_3=-1$, then $S=\{r_+, (a_3-2)r, d\}$.
 \item If $a_1+a_2-2a_3 \neq -1$, then $S=\{y_r, d, (a_3-2)r,(a_1-2)u\}$.
\end{itemize}

 \item[a7)] If $a_1-a_2+2a_3\in 2\bZ_{\ge2}$, then $S$ is either $\{m_1r,m_2u, m_3d\}$ or $\{2r_+, m_1r, m_3d\}$.
\begin{itemize}\setlength{\leftskip}{-11mm}
 \item If $a_1+a_2-2a_3=-2$, then $S=\{2r_+, (a_3-a_1-2)r, a_1d\}$.
 \item If $a_1+a_2-2a_3 \neq -2$, then $S=\bigl\{\tfrac{1}{2}(a_2-a_1)r,\bigl(\tfrac{1}{2}(a_1+a_2)-a_3\bigr)u, \bigl(\tfrac{1}{2}(a_1-a_2)+a_3\bigr)d\bigr\}$.
\end{itemize}

 \item[a8)] If $a_1-a_2+2a_3 \in 2\bZ_{\ge2}+1$, then $S$ is one of $\{r_+, m_1r, m_3d\}$, $\{y_r, m_1r, m_3d\}$ and $\{y_r, u, m_1r, m_3d\}$.
\begin{itemize}\setlength{\leftskip}{-11mm}
 \item If $a_1+a_2-2a_3=-1$, then $S=\{r_+, (a_3-a_1-1)r, a_1d\}$.
 \item If $a_1+a_2-2a_3=1$, then $S=\{y_r, (a_3-a_1)r, (a_1-1)d\}$.
 \item If $a_1+a_2-2a_3=3$, then $S=\{y_r, u, (a_3-a_1+1)r, (a_1-2)d\}$.
\end{itemize}
\end{itemize}
\vspace{5mm}

\item[b)] Suppose that $a_4-a_3=1$. In this case, $S$ is one as in Table \ref{case2}.
\renewcommand{\arraystretch}{1.4}
{\begin{table}[ht]
\begin{tabular}{c|c|c}
$S$ & $a_1-a_2+2a_3$ & $a_1+a_2-2a_3$\\\hline\hline

 $\{r_p,m_1r,m_2u,m_3d\}$ & $-1+2m_3$ & $1+2m_2$\\\hline
 $\{r_p,r_+,m_1r,m_3d\}$ & $2m_3$ & $0$\\\hline
 $\{r_p,r_-,m_1r,m_2u\}$ & $-2$ & $2+2m_2$\\\hline

 $\{r_p,y_r,m_1r,m_3d\}$ & $2m_3$ & $2$\\\hline
 $\{r_p,y_r,u,m_1r,m_3d\}$ & $2m_3$ & $4$\\\hline
 $\{r_p,y_r,m_1r,m_2u\}$ & $0$ & $2+2m_2$\\\hline
 $\{r_p,y_r,d,m_1r,m_2u\}$ & $2$ & $4+2m_2$\\\hline

 $\{P_+\}$ & $0$ & $0$\\\hline
 $\{\overline{P_+}\}$ & $0$ & $2$\\\hline
 $\{\overline{P_+},d\}$ & $2$ & $2$\\\hline
\end{tabular}\vspace{5mm}
 \caption{All cases of $S$ for a $1$-puncture piece and $a_4-a_3=1$}\label{case2}
\end{table}}

\begin{itemize}\setlength{\leftskip}{-5mm}
 \item[b1)] If $a_1-a_2+2a_3\in2\bZ+1$, then $S=\bigl\{r_p,\tfrac{1}{2}(a_2-a_1-1)r,\bigl(\tfrac{1}{2}(a_1+a_2-1)-a_3\bigr)u,\bigl(\tfrac{1}{2}(a_1-a_2+1)+a_3\bigr)d\bigr\}$.

 \item[b2)] If $a_1-a_2+2a_3=-2$, then $S=\{r_p,r_-,a_3r,a_1u\}$.

 \item[b3)] If $a_1-a_2+2a_3=0$ and $a_1+a_2-2a_3=0$, then $S$ is either $\{r_p,r_+,m_1r\}$ or $\{P_+\}$.
\begin{itemize}\setlength{\leftskip}{-11mm}
 \item If $a_3=0$, then $S=\{P_+\}$.
 \item If $a_3 \neq 0$, then $S=\{r_p,r_+,(a_3-1)r\}$.
\end{itemize}

 \item[b4)] If $a_1-a_2+2a_3=0$ and $a_1+a_2-2a_3=2$, then $S$ is either $\{r_p,y_r,m_1r\}$ or $\{\overline{P_+}\}$.
\begin{itemize}\setlength{\leftskip}{-11mm}
 \item If $a_3=0$, then $S=\{\overline{P_+}\}$.
 \item If $a_3 \neq 0$, then $S=\{r_p,y_r,(a_3-1)r\}$.
\end{itemize}

 \item[b5)] If $a_1-a_2+2a_3=0$ and $a_1+a_2-2a_3\ge4$, then $S=\{r_p,y_r,(a_3-1)r,(a_1-1)u\}$.

 \item[b6)] If $a_1-a_2+2a_3=2$ and $a_1+a_2-2a_3=0$, then $S=\{r_p,r_+,(a_3-2)r,d\}$.

 \item[b7)] If $a_1-a_2+2a_3=2$ and $a_1+a_2-2a_3=2$, then $S$ is either $\{r_p,y_r,m_1r,d\}$ or $\{\overline{P_+},d\}$.
\begin{itemize}\setlength{\leftskip}{-11mm}
 \item If $a_3=1$, then $S=\{\overline{P_+},d\}$.
 \item If $a_3 \neq 1$, then $S=\{r_p,y_r,(a_3-2)r,d\}$.
\end{itemize}

 \item[b8)] If $a_1-a_2+2a_3=2$ and $a_1+a_2-2a_3\ge4$, then $S=\{r_p,y_r,d,(a_3-2)r,(a_1-2)u\}$.

 \item[b9)] If $a_1-a_2+2a_3\in2\bZ_{\ge2}$ and $a_1+a_2-2a_3=0$, then $S=\{r_p,r_+,(a_3-a_1-1)r,a_1d\}$.

 \item[b10)] If $a_1-a_2+2a_3\in2\bZ_{\ge2}$ and $a_1+a_2-2a_3=2$, then $S=\{r_p,y_r,(a_3-a_1)r,(a_1-1)d\}$.

 \item[b11)] If $a_1-a_2+2a_3\in2\bZ_{\ge2}$ and $a_1+a_2-2a_3=4$, then $S=\{r_p,y_r,u,(a_3-a_1+1)r,(a_1-2)d\}$.
\end{itemize}
\vspace{5mm}

\item[c)] Suppose that $a_4-a_3=2$. In this case, $S$ is either $\{2r_p,m_1r,m_2u,m_3d\}$ or $\{r_p,l_p,m_2u,m_3d\}$.
\begin{itemize}\setlength{\leftskip}{-5mm}
 \item[c1)] If $a_1=a_2$, then $S=\{r_p,l_p,(a_1-a_3-1)u,a_3d\}$.

 \item[c2)] If $a_1 \neq a_2$, then  $S=\bigl\{2r_p,\tfrac{1}{2}(a_2-a_1-2)r,\bigl(\tfrac{1}{2}(a_1+a_2)-a_3-1\bigr)u,\bigl(\tfrac{1}{2}(a_1-a_2)+a_3+1\bigr)d\bigr\}$.
\end{itemize}
\end{itemize}
 By Table \ref{ss1}, these cover all cases of $(a_1,a_2,a_3,a_4)$. Therefore, $v|_{\square}$ gives the unique set of modified segments $S_{\square}$.

 Finally, we consider the case that $\square$ is a $2$-puncture piece. That is, $v|_{\square}=(a_1,a_2,a_3,a_4,a_5)$. By symmetry, we can assume that $a_3-a_2 \le a_5-a_4$. We consider the division into cases as in Table \ref{div}.

\renewcommand{\arraystretch}{1.4}
{\begin{table}[ht]
\begin{tabular}{c|c|c|c|c}
   \multicolumn{4}{c|}{$a_1 = 0$} & a) \\ \hline
   \multirow{6}{*}{$a_1 \neq 0$} & \multirow{2}{*}{$a_1$ : odd} & \multicolumn{2}{c|}{$a_5-a_4 = 0$} & b) \\ \cline{3-5}
    & & \multicolumn{2}{c|}{$a_5-a_4 \neq 0$} & c) \\ \cline{2-5}
    & \multirow{4}{*}{$a_1$ : even} & \multicolumn{2}{c|}{$a_5-a_4 = 2$} & d) \\ \cline{3-5}
    & & \multirow{2}{*}{$a_5-a_4 = 1$} & $a_3-a_2 = 0$ & e) \\ \cline{4-5}
    & & & $a_3-a_2 = 1$ & f) \\ \cline{3-5}
    & & \multicolumn{2}{c|}{$a_5-a_4 = 0$} & g)
\end{tabular}
\begin{itemize}
 \item[a)] $\{H\}$, $\{R_n\}$
 \item[b)] $\{c_n,m_1s_{n-1},m_2s_n\}$, $\{\overline{c_n},m_1s_{n-1},m_2s_n\}$, $\{\overline{c_n},m_1s_{n-1},m_2s_n,h\}$
 \item[c)] $\{r_n,m_1s_{2n-2},m_2s_{2n-1},m_3h\}$ $(m_1 \neq 0)$, $\{r_n,m_1s_{2n-1},m_2s_{2n},m_3h\}$
 \item[d)] $\{2r_n,m_1s_{2n-2},m_2s_{2n-1},m_3h\}$ $(m_1 \neq 0)$, $\{2r_n,m_1s_{2n-1},m_2s_{2n},m_3h\}$, $\{r_n,r_{n+1},m_1s_{2n},m_3h\}$
 \item[e)] $\{\overline{R_n}\}$ ($n \neq 0$), $\{r_n,c_{i+1},m_1s_{i},m_2s_{i+1}\}$, $\{r_n,\overline{c_{i+1}},m_1s_{i},m_2s_{i+1}\}$ ($i = 2n-2$ or $2n-1$)
 \item[f)] $\{r_n,l_{n-1},m_1s_{2n-2},m_2s_{2n-1},m_3h\}$, $\{r_n,l_n,m_1s_{2n-1},m_2s_{2n},m_3h\}$
 \item[g)] $\{2c_n,m_1s_{n-1},m_2s_n\}$, $\{m_1s_{n-1},m_2s_n,m_3h\}$ (If $n>0$, $m_2 \neq 0$. If $n<0$, $m_1 \neq 0$.)
\end{itemize}\vspace{5mm}
 \caption{Division into cases}\label{div}
\end{table}}

\begin{itemize}\setlength{\leftskip}{-3mm}\setlength{\itemsep}{3mm}
 \item[a)] Suppose that $a_1 = 0$. If $a_2 = 0$, then $S = \{H\}$. If $a_2 \neq 0$, then $S = \{R_{a_2}\}$.

 \item[b)] Suppose that $a_1 \neq 0$ is odd and $a_5-a_4=0$. In this case, $S$ is one of the followings:
\begin{center}$\displaystyle
 \{c_n,m_1s_{n-1},m_2s_n\}, \{\overline{c_n},m_1s_{n-1},m_2s_n\}, \{\overline{c_n},m_1s_{n-1},m_2s_n,h\}.
$\end{center}
 Set
\begin{center}$\displaystyle
m = \frac{a_1-1}{2}.
$\end{center}
\begin{itemize}\setlength{\leftskip}{-5mm}\setlength{\itemsep}{3mm}
 \item[b1)] If $|a_2-a_4|>m$, then $S=\{c_n,m_1s_{n-1},m_2s_n\}$ for $n \neq 0$.
     \begin{itemize}\setlength{\leftskip}{-7mm}\setlength{\itemsep}{3mm}
 \item[b1i)] If $a_2-a_4>m$, then $n>0$. In this case,
\begin{center}$\displaystyle
 n+m n \le a_2 \le n+m(n+1),
$\end{center}
 thus
\begin{center}$\displaystyle
 \frac{a_2-m}{m+1} \le n \le \frac{a_2}{m+1}.
$\end{center}
 Since
\begin{center}$\displaystyle
\frac{a_2}{m+1}-\frac{a_2-m}{m+1}=\frac{m}{m+1}<1,
$\end{center}
 then $n$ is uniquely given as
\begin{center}$\displaystyle
\biggl\lfloor\frac{a_2}{m+1}\biggr\rfloor,
$\end{center}
 where $\lfloor x \rfloor := \max\{n \in \mathbb{Z} \mid n \le x\}$. We have $m_2 = a_2-(m+1)n$ and $m_1=m-m_2$, that is,\\
\begin{center}$\displaystyle
 S=\biggl\{c_n,\biggl(\frac{n+1}{2}a_1-a_2+\frac{n-1}{2}\biggr)s_{n-1},\biggl(-\frac{n}{2}a_1+a_2-\frac{n}{2}\biggr)s_n\biggl\}, n=\biggl\lfloor\frac{2a_2}{a_1+1}\biggr\rfloor.
$\end{center}

 \item[b1ii)] If $a_2-a_4<-m$, then $n<0$. In the same way as b1i), we obtain
\begin{center}$\displaystyle
 n=\biggl\lfloor\frac{-a_2-1}{m+1}\biggr\rfloor, m_1 = a_2+(m+1)n+m+1 \text{ and  } m_2=m-m_1,
$\end{center}
 that is,
\begin{center}\small$\displaystyle
 S=\biggl\{c_n,\biggl( \frac{n+1}{2}a_1+a_2+\frac{n+1}{2} \biggr)s_{n-1},\biggl( -\frac{n}{2}a_1-a_2-\frac{n+2}{2} \biggr)s_n \biggl\}, n=\biggl\lfloor \frac{-2(a_2+1)}{a_1+1} \biggr\rfloor.
$\end{center}
     \end{itemize}

 \item[b2)] If $|a_2-a_4| \le m$, then $S=\{c_0,m_1s_{-1},m_2s_0\}$, $S=\{\overline{c_n},m_1s_{n-1},m_2s_n\}$ or $\{\overline{c_n},m_1s_{n-1},m_2s_n,h\}$.
     \begin{itemize}\setlength{\leftskip}{-7mm}\setlength{\itemsep}{3mm}
 \item[b2i)] If $a_2+a_4 \le m$, then $S=\{c_0,a_4s_{-1},a_2s_0\}$.

 \item[b2ii)] If $a_2+a_4 = m+1$, then
\begin{center}$\displaystyle
 S=\{\overline{c_0},(m-a_2)s_{-1},(m-a_4)s_0\}=\{\overline{c_0},(\tfrac{1}{2}a_1-a_2-\tfrac{1}{2})s_{-1},(\tfrac{1}{2}a_1-a_4-\tfrac{1}{2})s_0\}.
$\end{center}

 \item[b2iii)] If $a_2+a_4=m+2$, then $S$ is one of the followings:
\[
 \{\overline{c_1},(a_2-2)s_0\}, \{\overline{c_{-1}},(a_4-2)s_{-1}\}, \{\overline{c_0},(m-a_2)s_{-1},(m-a_4)s_0,h\}.
\]
 If $a_4=1$, then $S=\{\overline{c_1},(a_2-2)s_0\}$. If $a_2=1$, then $S=\{\overline{c_{-1}},(a_4-2)s_{-1}\}$. Otherwise,
\begin{center}$\displaystyle
 S=\{\overline{c_0},(m-a_2)s_{-1},(m-a_4)s_0,h\}=\{\overline{c_0},(\tfrac{1}{2}a_1-a_2-\tfrac{1}{2})s_{-1},(\tfrac{1}{2}a_1-a_4-\tfrac{1}{2})s_0,h\}.
$\end{center}

 \item[b2iv)] If $a_2+a_4 > m+2$ and $|a_2-a_4|=m$, then $S=\{\overline{c_n},m_1s_{n-1},m_2s_n\}$, $n \neq 0$ and it is not as in (b2iii). In the same way as b1i), if $a_2-a_4=m$, then we have
\begin{center}$\displaystyle
 n=\biggl\lfloor\frac{2(a_2-1)}{a_1-1}\biggr\rfloor, m_2 = a_2-mn-1 \text{ and } m_1=m-1-m_2,
$\end{center}
 that is,
\begin{center}$\displaystyle
 S=\biggl\{\overline{c_n},\biggl(\frac{n+1}{2}a_1-a_2-\frac{n+1}{2}\biggr)s_{n-1},\biggl(-\frac{n}{2}a_1+a_2+\frac{n-2}{2}\biggr)s_n\biggl\}, n=\biggl\lfloor\frac{2(a_2-1)}{a_1-1}\biggr\rfloor.
$\end{center}
 If $a_2-a_4=-m$, then we have
\begin{center}$\displaystyle
 n=\biggl\lfloor-\frac{2a_2}{a_1-1}\biggr\rfloor, m_1 = a_2+mn+m-1 \text{ and } m_2=m-1-m_1,
$\end{center}
 that is,
\begin{center}$\displaystyle
 S=\biggl\{\overline{c_n},\biggl(\frac{n+1}{2}a_1+a_2-\frac{n+3}{2}\biggr)s_{n-1},\biggl(-\frac{n}{2}a_1-a_2+\frac{n}{2}\biggr)s_n\biggl\}, n=\biggl\lfloor-\frac{2a_2}{a_1-1}\biggr\rfloor.
$\end{center}

 \item[b2v)] If $a_2+a_4 > m+2$ and $|a_2-a_4|=m-1$, then $S=\{\overline{c_n},m_1s_{n-1},m_2s_n,h\}$, $n \neq 0$ and it is not as in (b2iii). In the same way as b1i), if $a_2-a_4=m-1$, then we have
\begin{center}$\displaystyle
 n=\biggl\lfloor\frac{2(a_2-2)}{a_1-3}\biggr\rfloor, m_2 = a_2-(m-1)n-2 \text{ and } m_1=m-2-m_2,
$\end{center}
 that is,
\begin{center}\small$\displaystyle
 S=\biggl\{\overline{c_n},\biggl(\frac{n+1}{2}a_1-a_2-\frac{3n+1}{2}\biggr)s_{n-1},\biggl(-\frac{n}{2}a_1+a_2+\frac{3n-4}{2}\biggr)s_n,h\biggl\}, n=\biggl\lfloor\frac{2(a_2-2)}{a_1-3}\biggr\rfloor.
$\end{center}
 If $a_2-a_4=-m+1$, then we have
\begin{center}$\displaystyle
 n=\biggl\lfloor\frac{-2a_2+2}{a_1-3}\biggr\rfloor, m_1 = a_2+(m-1)n+m-3 \text{ and } m_2=m-2-m_1,
$\end{center}
 that is,
\begin{center}\small$\displaystyle
 S=\biggl\{\overline{c_n},\biggl(\frac{n+1}{2}a_1+a_2-\frac{3n+7}{2}\biggr)s_{n-1},\biggl(-\frac{n}{2}a_1-a_2+\frac{3n+2}{2}\biggr)s_n,h\biggl\}, n=\biggl\lfloor\frac{-2a_2+2}{a_1-3}\biggr\rfloor.
$\end{center}
     \end{itemize}
\end{itemize}

 \item[c)] Suppose that $a_1 \neq 0$ is odd and $a_5-a_4\neq0$. In this case, $S$ is one of the followings:
\begin{center}$\displaystyle
 \{r_n,m_1s_{2n-2},m_2s_{2n-1},m_3h\}, \{r_n,m_1s_{2n-1},m_2s_{2n},m_3h\}.
$\end{center}
 Set
\begin{center}$\displaystyle
 m = \frac{a_1-1}{2}.
$\end{center}
\begin{itemize}\setlength{\leftskip}{-5mm}\setlength{\itemsep}{3mm}
 \item[c1)] If $a_2 \le m$, then $S=\{r_0,m_1s_{i},m_2s_{i+1},m_3h\}$ for $i = 2n-2$ or $2n-1$. If $a_4 > m$, then
\begin{center}$\displaystyle
 S=\biggl\{r_0,\biggl(-\frac{1}{2}a_1+a_4+\frac{1}{2}\biggr)s_{-2},\biggl(\frac{1}{2}a_1-a_2-\frac{1}{2}\biggr)s_{-1},\biggl(\frac{1}{2}a_1+a_2-a_4-\frac{1}{2}\biggr)h\biggl\}.
$\end{center}
 If $a_4 \le m$, then
\begin{center}$\displaystyle
 S=\biggl\{r_0,\biggl(\frac{1}{2}a_1-a_2-\frac{1}{2}\biggr)s_{-2},\biggl(\frac{1}{2}a_1-a_4-\frac{1}{2}\biggr)s_{-1},\biggl(-\frac{1}{2}a_1+a_2+a_4+\frac{1}{2}\biggr)h\biggl\}.
$\end{center}

 \item[c2)] If $a_2>m$, then $S=\{r_n,m_1s_{i},m_2s_{i+1},m_3h\}$ for $n\neq0$ and $i = 2n-2$ or $2n-1$.

     \begin{itemize}\setlength{\leftskip}{-7mm}\setlength{\itemsep}{3mm}
 \item[c2i)] If $a_2-a_4>0$, then $n>0$ and $m_3=m-(a_2-a_4-1)$. In this case,
\begin{center}$\displaystyle
 n+(m-m_3) (2n-1)+m_3 \le a_2 \le n+(m-m_3)(2n+1)+m_3,
$\end{center}
 thus
\begin{center}$\displaystyle
 \frac{a_2-m}{2m-2m_3+1} \le n \le \frac{a_2+m-2m_3}{2m-2m_3+1}.
$\end{center}
 Since
\begin{center}$\displaystyle
 \frac{a_2+m-2m_3}{2m-2m_3+1}-\frac{a_2-m}{2m-2m_3+1}=\frac{2m-2m_3}{2m-2m_3+1}<1,
$\end{center}
 then $n$ is uniquely given as
\begin{center}$\displaystyle
 \biggl\lfloor\frac{a_2+m-2m_3}{2m-2m_3+1}\biggr\rfloor.
$\end{center}
 Let $f=a_2-(n+(m-m_3) (2n-1)+m_3)$. If $0 \le f < m-m_3$, then
\begin{center}$\displaystyle
{\setlength\arraycolsep{0.5mm}
\begin{array}{lll}
 S &=& \{r_n,(m-m_3-f)s_{2n-2},f s_{2n-1},m_3h\}\\
    &=& \biggl\{r_n,\biggl(\frac{1}{2}a_1+(2n-2)a_2-(2n-1)a_4-\frac{2n-1}{2}\biggr)s_{2n-2},\\
    &&    \biggl(-\frac{1}{2}a_1+(-2n+3)a_2+(2n-2)a_4+\frac{2n-3}{2}\biggr)s_{2n-1},\biggl(\frac{1}{2}a_1-a_2+a_4+\frac{1}{2}\biggr)h\biggl\},
\end{array}}
$\end{center}
 if $m-m_3 \le f \le 2(m-m_3)$, then
\begin{center}$\displaystyle
{\setlength\arraycolsep{0.5mm}
\begin{array}{lll}
 S &=& \{r_n,(2(m-m_3)-f)s_{2n-1},(f-(m-m_3))s_{2n},m_3h\}\\
    &=& \biggl\{r_n,\biggl(\frac{1}{2}a_1+(2n-1)a_2-2na_4-\frac{2n+1}{2}\biggr)s_{2n-1},\\
    &&    \biggl(-\frac{1}{2}a_1-(2n-2)a_2+(2n-1)a_4+\frac{2n-1}{2}\biggr)s_{2n},\biggl(\frac{1}{2}a_1-a_2+a_4+\frac{1}{2}\biggr)h\biggl\},
\end{array}}
$\end{center}
where
\begin{center}$\displaystyle
 n=\biggl\lfloor\displaystyle{\frac{-a_1+6a_2-4a_4-3}{2(2a_2-2a_4-1)}}\biggr\rfloor.
$\end{center}

 \item[c2ii)] If $a_2-a_4\le0$, then $n<0$ and $m_3=m-(a_4-a_2)$. In the same way as c2i), $n$ is uniquely given as
\begin{center}$\displaystyle
 \biggl\lfloor\frac{m-a_2}{2m-2m_3+1}\biggr\rfloor.
$\end{center}
 Let $f=a_2-(-n-(m-m_3) (2n+1)+m_3)$. If $0 \le f \le m-m_3$, then
\begin{center}$\displaystyle
{\setlength\arraycolsep{0.5mm}
\begin{array}{lll}
 S &=& \{r_n,f s_{2n-1},(m-m_3-f) s_{2n},m_3h\}\\
    &=& \biggl\{r_n,\biggl(-\frac{1}{2}a_1-(2n+1)a_2+(2n+2)a_4+\frac{2n+1}{2}\biggr)s_{2n-1},\\
    &&    \biggl(\frac{1}{2}a_1+2na_2-(2n+1)a_4-\frac{2n+1}{2}\biggr)s_{2n},\biggl(\frac{1}{2}a_1+a_2-a_4-\frac{1}{2}\biggr)h\biggl\},
\end{array}}
$\end{center}
if $m-m_3 < f \le 2(m-m_3)$, then
\begin{center}$\displaystyle
{\setlength\arraycolsep{0.5mm}
\begin{array}{lll}
 S &=& \{r_n,(f-(m-m_3))s_{2n-2},(2(m-m_3)-f)s_{2n-1},m_3h\}\\
    &=& \biggl\{r_n,\biggl(-\frac{1}{2}a_1-2na_2+(2n+1)a_4+\frac{2n+1}{2}\biggr)s_{2n-2},\\
    &&    \biggl(\frac{1}{2}a_1+(2n-1)a_2-2na_4-\frac{2n+1}{2}\biggr)s_{2n-1},\biggl(\frac{1}{2}a_1+a_2-a_4-\frac{1}{2}\biggr)h\biggl\},
\end{array}}
$\end{center}
where
\begin{center}$\displaystyle
 n=\biggl\lfloor\displaystyle{\frac{m-a_2}{2m-2m_3+1}}\biggr\rfloor.
$\end{center}
\end{itemize}
\end{itemize}

 \item[d)] Suppose that $a_1 \neq 0$ is even and $a_5-a_4=2$. In this case, $S$ is one of the followings:
\begin{center}$\displaystyle
 \{2r_n,m_1s_{2n-2},m_2s_{2n-1},m_3h\}, \{2r_n,m_1s_{2n-1},m_2s_{2n},m_3h\}, \{r_n,r_{n+1},m_1s_{2n},m_3h\}.
$\end{center}
 Set
\begin{center}$\displaystyle
 m = \frac{a_1-2}{2}.
$\end{center}
\begin{itemize}\setlength{\leftskip}{-5mm}\setlength{\itemsep}{3mm}
 \item[d1)] Suppose that $a_2=m+1$. If $a_2 \le a_4$, then $S=\{r_{-1},r_{0},(a_4-a_2)s_{-2},(m+a_2-a_4)h\}$. If $a_2 > a_4$, then $S=\{r_0,r_{1},(a_2-a_4-1)s_{0},a_4h\}$.

 \item[d2)] Suppose that $a_2<m+1$. If $a_4>m$, then
\begin{center}$\displaystyle
{\setlength\arraycolsep{0.5mm}
\begin{array}{lll}
 S &=& \{2r_0,(a_4-m)s_{-2},(m-a_2)s_{-1},(m+a_2-a_4)h\}\\
    &=& \{2r_0,(-\tfrac{1}{2}a_1+a_4+1)s_{-2},(\tfrac{1}{2}a_1-a_2-1)s_{-1},(\tfrac{1}{2}a_1+a_2-a_4-1)h\}.
\end{array}}
$\end{center}
 If $a_4 \le m$, then
\begin{center}$\displaystyle
{\setlength\arraycolsep{0.5mm}
\begin{array}{lll}
 S &=& \{2r_0,(m-a_2)s_{-1},(m-a_4)s_{0},(a_2+a_4-m)h\}\\
    &=& \{2r_0,(\tfrac{1}{2}a_1-a_2-1)s_{-1},(\tfrac{1}{2}a_1-a_4-1)s_{0},(-\tfrac{1}{2}a_1+a_2+a_4+1)h\}.
\end{array}}
$\end{center}

 \item[d3)] Suppose that $a_2>m+1$.

\begin{itemize}\setlength{\leftskip}{-7mm}\setlength{\itemsep}{3mm}
 \item[d3i)] If $a_2-a_4>0$, then $n>0$ and $m_3=m-a_2+a_4+2$. If $S = \{2r_n,m_1s_{2n-2},m_2s_{2n-1},m_3h\}$ or $\{2r_n,m_1s_{2n-1},m_2s_{2n},m_3h\}$, $a_2$ satisfy
\begin{center}$\displaystyle
 2n+(m-m_3)(2n-1)+m_3 \le a_2 \le 2n+(m-m_3)(2n+1)+m_3
$\end{center}
 In particular, there is no $n' \in \mathbb{Z}$ such that $a_2=2n'+1+(m-m_3)(2n'+1)+m_3$. If there is such $n' \in \mathbb{Z}$, then
\begin{center}$\displaystyle
 S=\{r_n,r_{n+1},(a_2-a_4-2)s_{2n},(m-a_2+a_4+2)h\},
$\end{center}
where
\begin{center}$\displaystyle
 n=n'=\frac{a_2-m-1}{2(m-m_3)+2}.
$\end{center}
 If not, in the same way as c2i), we have
\begin{center}$\displaystyle
 n=\biggl\lfloor\frac{a_2+m-2m_3}{2(m-m_3)+2}\biggr\rfloor.
$\end{center}
 Set $f=a_2-(2n+(m-m_3)(2n-1)+m_3)$. If $0 \le f \le m-m_3$, then
\begin{center}$\displaystyle
{\setlength\arraycolsep{0.5mm}
\begin{array}{lll}
 S &=& \{2r_n,(m-m_3-f)s_{2n-2},fs_{2n-1},m_3h\}\\
    &=& \bigl\{2r_n,\bigl(\tfrac{1}{2}a_1+(2n-2)a_2-(2n-1)a_4-2n+1\bigr)s_{2n-2},\\
    &&    \bigl(-\tfrac{1}{2}a_1-(2n-3)a_2+(2n-2)a_4+2n-3\bigr)s_{2n-1},\bigl(\tfrac{1}{2}a_1-a_2+a_4+1\bigr)h\bigl\},
\end{array}}
$\end{center}
 if $m-m_3 < f \le 2(m-m_3)$, then
\begin{center}$\displaystyle
{\setlength\arraycolsep{0.5mm}
\begin{array}{lll}
 S &=& \{2r_n,(2(a_2-a_4-2)-f)s_{2n-1},(f+a_2-a_4-2)s_{2n},m_3h\}\\
    &=& \bigl\{2r_n,\bigl(\tfrac{1}{2}a_1+(2n-1)a_2-2na_4-2n-1\bigr)s_{2n-1},\\
    &&    \bigl(-\tfrac{1}{2}a_1-(2n-2)a_2+(2n-1)a_4+2n-1\bigr)s_{2n},\bigl(\tfrac{1}{2}a_1-a_2+a_4+1\bigr)h\bigl\},
\end{array}}
$\end{center}
where
\begin{center}$\displaystyle
 n=\biggl\lfloor\displaystyle{\frac{a_2+m-2m_3}{2(m-m_3)+2}}\biggr\rfloor.
$\end{center}

 \item[d3ii)] If $a_2-a_4<0$, then $n<0$ and $m_3=m+a_2-a_4$. In the same way as d3i), if there is $n' \in \mathbb{Z}$ such that $a_2=-2n'-1+(m-m_3)(-2n'-1)+m_3$, then
\begin{center}$\displaystyle
 S=\{r_n,r_{n+1},(-a_2+a_4)s_{2n},(m+a_2-a_4)h\},
$\end{center}
where
\begin{center}$\displaystyle
 n=n'=\biggl\lfloor\frac{-m+2m_3-a_2+1}{2(m-m_3)+2}\biggr\rfloor.
$\end{center}
 If not, we have
\begin{center}$\displaystyle
 n=\biggl\lfloor\frac{m-a_2}{2(m-m_3)+2}\biggr\rfloor
$\end{center}
 and $S$ is obtained as follows: Set $f=a_2-(-2n+(m-m_3)(-2n-1)+m_3)$. If $0 \le f \le m-m_3$, then
\begin{center}$\displaystyle
{\setlength\arraycolsep{0.5mm}
\begin{array}{lll}
 S &=& \{2r_n,fs_{2n-1},(m-m_3-f)s_{2n},m_3h\}\\
    &=& \bigl\{2r_n,\bigl(-\tfrac{1}{2}a_1-(2n+1)a_2+(2n+2)a_4+2n+1\bigr)s_{2n-1},\\
    &&    \bigl(\tfrac{1}{2}a_1+2na_2-(2n+1)a_4-2n-1\bigr)s_{2n},\bigl(\tfrac{1}{2}a_1+a_2-a_4-1\bigr)h\bigl\},
\end{array}}
$\end{center}
 if $m-m_3 < f \le 2(m-m_3)$, then
\begin{center}$\displaystyle
{\setlength\arraycolsep{0.5mm}
\begin{array}{lll}
 S &=& \{2r_n,(f-m+m_3)s_{2n-2},(2(m-m_3)-f)s_{2n-1},m_3h\}\\
    &=& \bigl\{2r_n,\bigl(-\tfrac{1}{2}a_1-2na_2+(2n+1)a_4+2n+1\bigr)s_{2n-2},\\
    &&    \bigl(\tfrac{1}{2}a_1+(2n-1)a_2-2na_4-2n-1\bigr)s_{2n-1},\bigl(\tfrac{1}{2}a_1+a_2-a_4-1\bigr)h\bigl\},
\end{array}}
$\end{center}
where
\begin{center}$\displaystyle
 n=\biggl\lfloor\displaystyle{\frac{m-a_2}{2(m-m_3)+2}}\biggr\rfloor.
$\end{center}
\end{itemize}
\end{itemize}

 \item[e)] Suppose that $a_1 \neq 0$ is even, $a_5-a_4=1$ and $a_3-a_2=0$. In this case, $S$ is one of the followings:
\begin{center}$\displaystyle
 \{\overline{R_n}\} (n \neq 0), \{r_n,c_{i+1},m_1s_{i},m_2s_{i+1}\}, \{r_n,\overline{c_{i+1}},m_1s_{i},m_2s_{i+1}\}
$\end{center}
for $i = 2n-2$ or $2n-1$. Suppose that $a_1=2$, $a_2 \neq 0$, $|a_2-a_4| \le 1$ and $|a_2-a_5| \le 1$. If $|a_2-a_4| = 1$, then $S=\{\overline{R_{a_4}}\}$. If $|a_2-a_4| = 0$, then $S=\{\overline{R_{-a_4}}\}$. Otherwise, $S=\{r_n,c_{i+1},m_1s_{i},m_2s_{i+1}\}$ or $\{r_n,\overline{c_{i+1}},m_1s_{i},m_2s_{i+1}\}$. Set
\begin{center}$\displaystyle
 m = \frac{a_1-2}{2}.
$\end{center}
\begin{itemize}\setlength{\leftskip}{-5mm}\setlength{\itemsep}{3mm}
 \item[e1)] If $a_2 \le m$, then $n=0$.

\begin{itemize}\setlength{\leftskip}{-7mm}\setlength{\itemsep}{3mm}
 \item[e1i)] Suppose that $a_4>m$. If $a_4-a_2=m+1$, then
\begin{center}$\displaystyle
 S=\{r_0,c_{-1},a_2s_{-2},(\tfrac{1}{2}a_1-a_2-1)s_{-1}\}.
$\end{center}
 If not, then
\begin{center}$\displaystyle
 S=\{r_0,\overline{c_{-1}},(a_2-1)s_{-2},(\tfrac{1}{2}a_1-a_2-1)s_{-1}\}.
$\end{center}

 \item[e1ii)] Suppose that $a_4 \le m$. If $a_2+a_4=m$, then
\begin{center}$\displaystyle
 S=\{r_0,c_{0},a_4s_{-1},a_2s_{0}\}.
$\end{center}
 If not, then
\begin{center}$\displaystyle
 S=\{r_0,\overline{c_{0}},(a_4-1)s_{-1},(a_2-1)s_{0}\}.
$\end{center}
\end{itemize}

 \item[e2)] Suppose that $a_2>m$, then $n \neq 0$.

\begin{itemize}\setlength{\leftskip}{-7mm}\setlength{\itemsep}{3mm}
 \item[e2i)] If $a_2-a_4>0$, then $n>0$. Suppose that $a_2-a_4=m+2$. In the same way as b1i), then $n$ is uniquely given as
\begin{center}$\displaystyle
 \biggl\lfloor\frac{a_2+m+1}{2m+3}\biggr\rfloor.
$\end{center}
 Set $f=a_2-(3n-1+m(2n-1))$. If $0 \le f \le m$, then
\begin{center}$\displaystyle
{\setlength\arraycolsep{0.5mm}
\begin{array}{lll}
 S &=& \{r_n,c_{2n-1},(m-f)s_{2n-2},fs_{2n-1}\}\\
    &=& \bigl\{r_n,c_{2n-1},(na_1-a_2+n-1)s_{2n-2},\bigl((-n+\tfrac{1}{2})a_1+a_2-n\bigr)s_{2n-1}\bigr\},
\end{array}}
$\end{center}
 if $m+1 \le f \le 2m+1$, then
\begin{center}$\displaystyle
{\setlength\arraycolsep{0.5mm}
\begin{array}{lll}
 S &=& \{r_n,c_{2n},(2m+1-f)s_{2n-1},(f-m-1)s_{2n}\}\\
    &=& \bigl\{r_n,c_{2n},\bigl((n+\tfrac{1}{2})a_1-a_2+n-1\bigr)s_{2n-1},(-na_1+a_2-n)s_{2n}\bigr\},
\end{array}}
$\end{center}
where
\begin{center}$\displaystyle
 n=\biggl\lfloor\frac{a_1+2a_2}{2(a_1+1)}\biggr\rfloor.
$\end{center}
 Suppose that $a_2-a_4 \neq m+2$. Then $n$ is uniquely given as
\begin{center}$\displaystyle
 \biggl\lfloor\frac{a_2+m-1}{2m+1}\biggr\rfloor.
$\end{center}
 Set $f=a_2-(3n+(m-1)(2n-1))$. If $0 \le f \le m-1$, then
\begin{center}$\displaystyle
{\setlength\arraycolsep{0.5mm}
\begin{array}{lll}
 S &=& \{r_n,\overline{c_{2n-1}},(m-1-f)s_{2n-2},fs_{2n-1}\}\\
    &=& \bigl\{r_n,\overline{c_{2n-1}},(na_1-a_2-n)s_{2n-2},\bigl((-n+\tfrac{1}{2})a_1+a_2+n-2\bigr)s_{2n-1}\bigr\},
\end{array}}
$\end{center}
 if $m \le f \le 2m-1$, then
\begin{center}$\displaystyle
{\setlength\arraycolsep{0.5mm}
\begin{array}{lll}
 S &=& \{r_n,\overline{c_{2n}},(2m-1-f)s_{2n-1},(f-m)s_{2n}\}\\
    &=& \bigl\{r_n,\overline{c_{2n}},\bigl((n+\tfrac{1}{2})a_1-a_2-n-1\bigr)s_{2n-1},(-na_1+a_2+n-1)s_{2n}\bigr\},
\end{array}}
$\end{center}
where
\begin{center}$\displaystyle
 n=\biggl\lfloor\frac{a_1+2a_2-4}{2(a_1-1)}\biggr\rfloor.
$\end{center}

 \item[e2ii)] If $a_2-a_4 \le 0$, then $n<0$. Suppose that $a_2-a_4=-m-1$. In the same way as above, then $n$ is uniquely given as
\begin{center}$\displaystyle
 \biggl\lfloor\frac{-a_2+m}{2m+3}\biggr\rfloor.
$\end{center}
 Set $f=a_2-(-3n-1+m(-2n-1))$. If $0 \le f \le m$, then
\begin{center}$\displaystyle
{\setlength\arraycolsep{0.5mm}
\begin{array}{lll}
 S &=& \{r_n,c_{2n},fs_{2n-1},(m-f)s_{2n}\}\\
    &=& \bigl\{r_n,c_{2n},\bigl((n+\tfrac{1}{2})a_1+a_2+n\bigr)s_{2n-1},(-na_1-a_2-n-1)s_{2n}\bigr\},
\end{array}}
$\end{center}
 if $m+1 \le f \le 2m+1$, then
\begin{center}$\displaystyle
{\setlength\arraycolsep{0.5mm}
\begin{array}{lll}
 S &=& \{r_n,c_{2n-1},(f-m-1)s_{2n-2},(2m+1-f)s_{2n-1}\}\\
    &=& \bigl\{r_n,c_{2n-1},(na_1+a_2+n)s_{2n-2},\bigl(-(n+\frac{1}{2})a_1-a_2-n-1\bigr)s_{2n-1}\bigr\},
\end{array}}
$\end{center}
where
\begin{center}$\displaystyle
 n=\biggl\lfloor\frac{a_1-2a_2-2}{2(a_1+1)}\biggr\rfloor.
$\end{center}
 Suppose that $a_2-a_4 \neq -m-1$. Then $n$ is uniquely given as
\begin{center}$\displaystyle
 \biggl\lfloor\frac{-a_2+m}{2m+1}\biggr\rfloor.
$\end{center}
 Set $f=a_2-(-3n+(m-1)(-2n-1))$. If $0 \le f \le m-1$, then
\begin{center}$\displaystyle
{\setlength\arraycolsep{0.5mm}
\begin{array}{lll}
 S &=& \{r_n,\overline{c_{2n}},fs_{2n-1},(m-1-f)s_{2n}\}\\
    &=& \bigl\{r_n,\overline{c_{2n}},\bigl((n+\tfrac{1}{2})a_1+a_2-n-2\bigr)s_{2n-1},(-na_1-a_2+n)s_{2n}\bigr\},
\end{array}}
$\end{center}
 if $m \le f \le 2m-1$, then
\begin{center}$\displaystyle
{\setlength\arraycolsep{0.5mm}
\begin{array}{lll}
 S &=& \{r_n,\overline{c_{2n-1}},(f-m)s_{2n-2},(2m-1-f)s_{2n-1}\}\\
    &=& \bigl\{r_n,\overline{c_{2n-1}},(na_1+a_2-n-1)s_{2n-1},\bigl((-n+\tfrac{1}{2})a_1-a_2+n-1\bigr)s_{2n}\bigr\},
\end{array}}
$\end{center}
where
\begin{center}$\displaystyle
 n=\biggl\lfloor\frac{a_1-2a_2-2}{2(a_1-1)}\biggr\rfloor.
$\end{center}
\end{itemize}
\end{itemize}

 \item[f)] Suppose that $a_1 \neq 0$ is even, $a_5-a_4=1$ and $a_3-a_2=1$. In this case, $S$ is one of the followings:
\begin{center}$\displaystyle
 \{r_n,l_{n-1},m_1s_{2n-2},m_2s_{2n-1},m_3h\}, \{r_n,l_n,m_1s_{2n-1},m_2s_{2n},m_3h\}.
$\end{center}
 Set
\begin{center}$\displaystyle
 m = \frac{a_1-2}{2}.
$\end{center}
 In the same way as d), we construct $S$.

\begin{itemize}\setlength{\leftskip}{-5mm}\setlength{\itemsep}{3mm}
 \item[f1)] If $a_2 \le m$, then $n=0$. If $a_4>m$, then
\begin{center}$\displaystyle
{\setlength\arraycolsep{0.5mm}
\begin{array}{lll}
 S &=& \{r_0,l_{-1},(a_4-m-1)s_{-2},(m-a_2)s_{-1},(m+a_2-a_4+1)h\}\\
    &=& \{r_0,l_{-1},(-\tfrac{1}{2}a_1+a_4)s_{-2},(\tfrac{1}{2}a_1-a_2-1)s_{-1},(\tfrac{1}{2}a_1+a_2-a_4)h\}.
\end{array}}
$\end{center}
 If not, then
\begin{center}$\displaystyle
{\setlength\arraycolsep{0.5mm}
\begin{array}{lll}
 S &=& \{r_0,l_{0},(m-a_2)s_{-1},(m-a_4)s_{0},(a_2+a_4-m)h\}\\
    &=& \{r_0,l_{0},(\tfrac{1}{2}a_1-a_2-1)s_{-1},(\tfrac{1}{2}a_1-a_4-1)s_{0},(-\tfrac{1}{2}a_1+a_2+a_4+1)h\}.
\end{array}}
$\end{center}

 \item[f2)] If $a_2>m$, then $n \neq 0$.

\begin{itemize}\setlength{\leftskip}{-7mm}\setlength{\itemsep}{3mm}
 \item[f2i)] If $a_2-a_4>0$, then $n>0$ and $m_3=m-(a_2-a_4-1)$. Moreover, $n$ is uniquely given as
\begin{center}$\displaystyle
 \biggl\lfloor\frac{a_2+m-2m_3+1}{2(m-m_3+1)}\biggr\rfloor.
$\end{center}
 Let $f=a_2-(2n-1+(m-m_3)(2n-1)+m_3)$. If $0 \le f \le m-m_3$, then
\begin{center}$\displaystyle
{\setlength\arraycolsep{0.5mm}
\begin{array}{lll}
 S &=& \{r_n,l_{n-1},(m-m_3-f)s_{2n-2},fs_{2n-1},m_3h\}\\
    &=& \bigl\{r_n,l_{n-1},\bigl(\tfrac{1}{2}a_1+(2n-2)a_2-(2n-1)a_4-1\bigr)s_{2n-2},\\
    &&    \bigl(-\tfrac{1}{2}a_1+(-2n+3)a_2+(2n-2)a_4\bigr)s_{2n-1},\bigl(\tfrac{1}{2}a_1-a_2+a_4\bigr)h\bigl\},
\end{array}}
$\end{center}
if $m-m_3+1 \le f \le 2(m-m_3)+1$, then
\begin{center}$\displaystyle
{\setlength\arraycolsep{0.5mm}
\begin{array}{lll}
 S &=& \{r_n,l_{n},(2(m-m_3)+1-f)s_{2n-1},(f-m+m_3-1)s_{2n},m_3h\}\\
    &=& \bigl\{r_n,l_{n},\bigl(\tfrac{1}{2}a_1+(2n-1)a_2-2na_4-1\bigr)s_{2n-1},\\
    &&    \bigl(-\tfrac{1}{2}a_1+(-2n+2)a_2+(2n-1)a_4\bigr)s_{2n},\bigl(\tfrac{1}{2}a_1-a_2+a_4\bigr)h\bigl\},
\end{array}}
$\end{center}
where
\begin{center}$\displaystyle
 n=\biggl\lfloor\displaystyle{\frac{-a_1+6a_2-4a_4}{4(a_2-a_4)}}\biggr\rfloor.
$\end{center}

 \item[f2ii)] If $a_2-a_4 \le 0$, then $n<0$ and $m_3=m+a_2-a_4+1$. Moreover, $n$ is uniquely given as
\begin{center}$\displaystyle
 \biggl\lfloor\frac{m-a_2}{2(m-m_3+1)}\biggr\rfloor.
$\end{center}
 Let $f=a_2-(-2n-1+(m-m_3)(-2n-1)+m_3)$. If $0 \le f \le m-m_3$, then
\begin{center}$\displaystyle
{\setlength\arraycolsep{0.5mm}
\begin{array}{lll}
 S &=& \{r_n,l_{n},fs_{2n-1},(m-m_3-f)s_{2n},m_3h\}\\
    &=& \bigl\{r_n,l_{n},\bigl(-\tfrac{1}{2}a_1-(2n+1)a_2+(2n+2)a_4\bigr)s_{2n-1},\\
    &&    \bigl(\tfrac{1}{2}a_1+2na_2-(2n+1)a_4-1\bigr)s_{2n},\bigl(\tfrac{1}{2}a_1+a_2-a_4\bigr)h\bigl\},
\end{array}}
$\end{center}
if $m-m_3+1 \le f \le 2(m-m_3)+1$, then
\begin{center}$\displaystyle
{\setlength\arraycolsep{0.5mm}
\begin{array}{lll}
 S &=& \{r_n,l_{n-1},(f-m+m_3-1)s_{2n-2},(2(m-m_3)+1-f)s_{2n-1},m_3h\}\\
    &=& \bigl\{r_n,l_{n},\bigl(-\tfrac{1}{2}a_1-2na_2+(2n+1)a_4\bigr)s_{2n-2},\\
    &&    \bigl(\tfrac{1}{2}a_1+(2n-1)a_2-2na_4-1\bigr)s_{2n-1},\bigl(\tfrac{1}{2}a_1+a_2-a_4\bigr)h\bigl\},
\end{array}}
$\end{center}
where
\begin{center}$\displaystyle
 n=\biggl\lfloor\displaystyle{\frac{a_1-2a_2-2}{4(a_2-a_4)}}\biggr\rfloor.
$\end{center}
\end{itemize}
\end{itemize}

 \item[g)] Suppose that $a_1 \neq 0$ is even, $a_5-a_4=0$. In this case, $S$ is one of the followings:
\begin{center}$\displaystyle
 \{2c_n,m_1s_{n-1},m_2s_n\}, \{m_1s_{n-1},m_2s_n,m_3h\},
$\end{center}
where for $S=\{m_1s_{n-1},m_2s_n,m_3h\}$ we assume the following conditions: If $n>0$, $m_2 \neq 0$; If $n<0$, $m_1 \neq 0$. Note that if $S=\{m_3h\}$, we have $n=0$.
 Set $m = a_1/2$.

\begin{itemize}\setlength{\leftskip}{-5mm}\setlength{\itemsep}{3mm}
 \item[g1)] Suppose that $|a_2-a_4|>m$. Then $S=\{2c_n,m_1s_{n-1},m_2s_n\}$ and $n \neq 0$. In the same way as b1i), if $a_2-a_4>m$, then
\begin{center}$\displaystyle
{\setlength\arraycolsep{0.5mm}
\begin{array}{lll}
 S &=& \{2c_n,((m+1)n+m-a_2-1)s_{n-1},(a_2-(m+1)n)s_n\}\\
    &=& \biggl\{2c_n,\biggl(\frac{n+1}{2}a_1-a_2+n-1\biggl)s_{n-1},\biggl(-\frac{n}{2}a_1+a_2-n\biggr)s_n\biggr\},
\end{array}}
$\end{center}
where
\begin{center}$\displaystyle
 n=\biggl\lfloor\frac{a_2}{m+1}\biggr\rfloor=\biggl\lfloor\frac{2a_2}{a_1+2}\biggr\rfloor.
$\end{center}
 If $a_2-a_4<-m$, then
\begin{center}$\displaystyle
{\setlength\arraycolsep{0.5mm}
\begin{array}{lll}
 S &=& \{2c_n,(a_2+(m+1)(n+1))s_{n-1},(-a_2-(m+1)n-2)s_n\}\\
    &=& \biggl\{2c_n,\biggl(\frac{n+1}{2}a_1+a_2+n+1\biggl)s_{n-1},\biggl(-\frac{n}{2}a_1-a_2-n-2\biggr)s_n\biggr\},
\end{array}}
$\end{center}
where
\begin{center}$\displaystyle
 n=\biggl\lfloor\frac{-a_2-2}{m+1}\biggr\rfloor=\biggl\lfloor\frac{-2(a_2+2)}{a_1+2}\biggr\rfloor.
$\end{center}

 \item[g2)] Suppose that $|a_2-a_4| \le m$. If $a_2+a_4<m$, then
\begin{center}$\displaystyle
 S=\{2c_0,a_4s_{-1},a_2s_0\}.
$\end{center}
 Suppose that $a_2+a_4 \ge m$.

\begin{itemize}\setlength{\leftskip}{-7mm}\setlength{\itemsep}{3mm}
 \item[g2i)] If $a_2, a_4 \le m$, then
\begin{center}$\displaystyle
 S=\{(\tfrac{1}{2}a_1-a_2)s_{-1},(\tfrac{1}{2}a_1-a_4)s_{0},(-\tfrac{1}{2}a_1+a_2+a_4)h\}.
$\end{center}

 \item[g2ii)] If either $a_2>m$ or $a_4>m$ holds, then $S=\{m_1s_{n-1},m_2s_n,m_3h\}$ for $n \neq 0$. Since $S$ is not $\{m_3 h\}$ by our assumptions, $a_2 \neq a_4$.  In the same way as b1i), if $a_2-a_4>0$, then
\begin{center}$\displaystyle
{\setlength\arraycolsep{0.5mm}
\begin{array}{lll}
 S &=& \{((a_2-a_4)(n+1)+m_3-a_2)s_{n-1},(a_2-((a_2-a_4)n+m_3))s_{n},(m-a_2+a_4)h\}\\
    &=& \{(\tfrac{1}{2}a_1+(n+1)a_2-na_4)s_{n-1},(-\tfrac{1}{2}a_1+(-n+2)a_2+(n-1)a_4)s_{n},(\frac{1}{2}a_1-a_2+a_4)h\},
\end{array}}
$\end{center}
where
\begin{center}$\displaystyle
 n=\biggl\lfloor\frac{2a_2-a_4-m}{a_2-a_4}\biggr\rfloor=\biggl\lfloor\frac{-a_1+4a_2-2a_4}{2(a_2-a_4)}\biggr\rfloor.
$\end{center}
 If $a_2-a_4<0$, then
\begin{center}$\displaystyle
{\setlength\arraycolsep{0.5mm}
\begin{array}{lll}
 S &=& \{(a_2-((-a_2+a_4)(-n-1)+m_3))s_{n-1},((-a_2+a_4)(-n)+m_3-a_2)s_{n},(m+a_2-a_4)h\}\\
    &=& \{(-\tfrac{1}{2}a_1-(n+1)a_2+(n+2)a_4)s_{n-1},(\tfrac{1}{2}a_1+na_2-(n+1)a_4)s_{n},(\tfrac{1}{2}a_1+a_2-a_4)h\},
\end{array}}
$\end{center}
where
\begin{center}$\displaystyle
 n=\biggl\lfloor\frac{m-a_4}{-a_2+a_4}\biggr\rfloor=\biggl\lfloor\frac{a_1-2a_4}{2(-a_2+a_4)}\biggr\rfloor.
$\end{center}
\end{itemize}
\end{itemize}
\end{itemize}
 By Table \ref{ss2}, these cover all cases of $(a_1,a_2,a_3,a_4,a_5)$. Therefore, $v|_{\square}$ gives the unique set of modified segments $S_{\square}$. This finishes the proof for $T \neq T_3$.

 To finish the proof of Theorem \ref{intinj}, we assume that $T = T_3$ (which consists of three pairs $(\tau_i,\tau_i')$ of conjugate arcs for $i \in \{1,2,3\}$ as in Figure \ref{QT}). For a modified tagged arc ${\sf m}$, we set $\Int(T_3,{\sf m}) = (a_1,a_1',a_2,a_2',a_3,a_3') \in \bZ_{\ge 0}^n$, where $a_i$ (resp., $a_i'$) is the intersection number of $\tau_i$ (resp., $\tau_i'$) and ${\sf m}$ as follows:
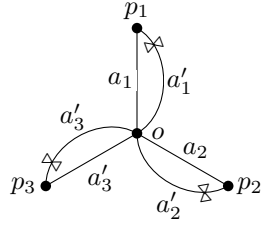
\[
\begin{tikzpicture}[baseline=0mm,scale=1.4]
 \coordinate (0) at (0,0); \node[right=2] at(0) {$o$};
 \coordinate (u) at (90:1); \node[above] at(u) {$p_1$};
 \coordinate (r) at (-30:1); \node[right] at(r) {$p_2$};
 \coordinate (l) at (210:1); \node[left] at(l) {$p_3$};
 \draw (0) to node[left,fill=white,inner sep=1]{$a_1$} (u);
 \draw (0) to [out=-60,in=-120,relative] node[pos=0.8]{\rotatebox{20}{\footnotesize $\bowtie$}} node[right,fill=white,inner sep=1]{$a_1'$} (u);
 \draw (0) to node[right=5,above=-3]{$a_2$} (r);
 \draw (0) to [out=-60,in=-120,relative] node[pos=0.8]{\rotatebox{100}{\footnotesize $\bowtie$}} node[below]{$a_2'$} (r);
 \draw (0) to node[below=7,right=-5]{$a_3'$} (l);
 \draw (0) to [out=-60,in=-120,relative] node[pos=0.8]{\rotatebox{-30}{\footnotesize $\bowtie$}} node[left=2,above=-1]{$a_3'$} (l);
 \fill (0) circle (0.5mm); \fill (u) circle (0.5mm); \fill (l) circle (0.5mm); \fill (r) circle (0.5mm);
\end{tikzpicture}
\]
 We show that ${\sf m}$ is uniquely determined by $\Int(T_3,{\sf m})$. By symmetry, we can assume that $a_i'-a_i \ge 0$ for $i \in \{1,2,3\}$. Let $\ell_i$ be a loop at $o$ cutting out a monogon with exactly one puncture $p_i$ as in Figure \ref{elli}.
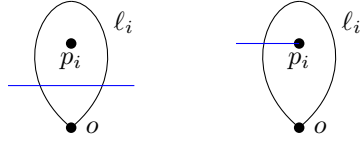
\begin{figure}[ht]
\begin{tikzpicture}[baseline=-1mm,scale=1.4]
 \coordinate (0) at (0,-0.2);
 \coordinate (1) at (0,-1); \node[right=2] at(1) {$o$};
 \draw (1) .. controls (0.6,-0.5) and (0.3,0.2) .. (0,0.2);
 \draw (1) .. controls (-0.6,-0.5) and (-0.3,0.2) .. (0,0.2);
 \fill(1) circle (0.5mm); \fill(0) circle (0.5mm);
 \node[below] at(0) {$p_i$}; \node at(0.5,0) {$\ell_i$};
 \draw[blue] (-0.6,-0.6)--(0.6,-0.6);
\end{tikzpicture}
   \hspace{10mm}
\begin{tikzpicture}[baseline=-1mm,scale=1.4]
 \coordinate (0) at (0,-0.2);
 \coordinate (1) at (0,-1); \node[right=2] at(1) {$o$};
 \draw (1) .. controls (0.6,-0.5) and (0.3,0.2) .. (0,0.2);
 \draw (1) .. controls (-0.6,-0.5) and (-0.3,0.2) .. (0,0.2);
 \fill(1) circle (0.5mm); \fill(0) circle (0.5mm);
 \node[below] at(0) {$p_i$}; \node at(0.5,0) {$\ell_i$};
 \draw[blue] (-0.6,-0.2)--(0);
\end{tikzpicture}
\caption{The loop $\ell_i$ corresponding to $a_i'$ and two kinds of segments of {\sf m} intersecting with $\ell_i$}\label{elli}
\end{figure}
 Note that $\ell_i$ is not a tagged arc, but we can define the intersection number $\Int(\ell_i,{\sf m})$ of $\ell_i$ and ${\sf m}$. It is clear that the number of intersection points of ${\sf m}$ and $\tau_i$ coincides with one of ${\sf m}$ and $\tau_i'$ except at $p_i$. Thus $a_i'-a_i$ is the number of the end points of ${\sf m}$ at $p_i$. Since ${\sf m}$ only intersects with $\ell_i$ in two ways as in Figure \ref{elli}, the set of segments of ${\sf m}$ in the monogon enclosed by $\ell_i$ consists of $a_i$ segments in the left diagram of Figure \ref{elli} and $a_i'-a_i$ segments in the right diagram of Figure \ref{elli}, in particular, is uniquely determined. Furthermore, by this observation, we have $\Int(\ell_i,{\sf m}) = 2 a_i + (a_i'-a_i)$. The set of segments of ${\sf m}$ in the triangle consisting of $\ell_1$, $\ell_2$ and $\ell_3$ is unique determined in the same way as a triangle piece which is Case($\tau_i$,$\tau_{i+1}$). Gluing their segments simultaneously, we can obtain ${\sf m}$. This finishes the proof of Theorem \ref{intinj}.

\section{Example}\label{Ex}

 Let $T$ be the tagged triangulation in Example \ref{ex}. We consider the following intersection vectors:
\[
\begin{tikzpicture}[baseline=0mm,scale=0.7]
 \coordinate (u) at (0,2);
 \coordinate (l) at (-150:2);
 \coordinate (r) at (-30:2);
 \coordinate (cd) at (0,-1);
 \coordinate (cl) at (150:1);
 \coordinate (cr) at (30:1);
 \draw (0,0) circle (2);
 \draw (r)--node[fill=white,inner sep=0.1]{$a_9$}(cr)--node[fill=white,inner sep=0.1]{$a_2$}(u)--node[fill=white,inner sep=0.1]{$a_1$}(cl)--node[fill=white,inner sep=0.1]{$a_3$}(cr);
 \draw (r)--node[fill=white,inner sep=0.1]{$a_8$}(cl)--node[fill=white,inner sep=0.1,pos=0.65]{$a_4$}(l);
 \draw (cl) .. controls (-145:2) and (-100:2) .. node[fill=white,inner sep=0.1,pos=0.8]{$a_5$} (r);
 \draw (cl) to node[fill=white,inner sep=0.1]{$a_6$} (cd);
 \draw (cl) to [out=-50,in=-100,relative] node[pos=0.85]{\rotatebox{80}{\footnotesize $\bowtie$}} node[fill=white,inner sep=0.1,pos=0.55]{$a_7$} (cd);
 \fill(u) circle (1mm); \fill(l) circle (1mm); \fill(r) circle (1mm); \fill(cd) circle (1mm); \fill(cl) circle (1mm); \fill(cr) circle (1mm);
\end{tikzpicture}\hspace{3mm}
\renewcommand{\arraystretch}{1}{\begin{array}{c}
      v_1=(1,1,0,0,0,0,0,0,0) \\
      v_2=(1,1,1,0,0,0,0,0,1) \\
      v_3=(1,2,2,1,1,0,0,1,2) \\
      v_4=(0,1,0,0,0,0,0,0,1) \\
      v_5=(1,1,2,1,1,1,1,1,1)
\end{array}}\hspace{3mm}
\renewcommand{\arraystretch}{1}{\begin{array}{c}
      v_6=(0,1,1,0,0,0,1,1,1) \\
      v_7=(0,1,1,0,0,1,0,1,1) \\
      v_8=(0,1,1,0,1,0,0,1,1) \\
      v_9=(0,1,1,0,0,0,0,0,0)
\end{array}}
\]
 By Theorem \ref{intinj}, each $v_i$ give a unique modified tagged arc ${\sf m}_i$ with respect to $T$. In fact, for instance, ${\sf m}_3$ is given in the way of the previous section as follows:
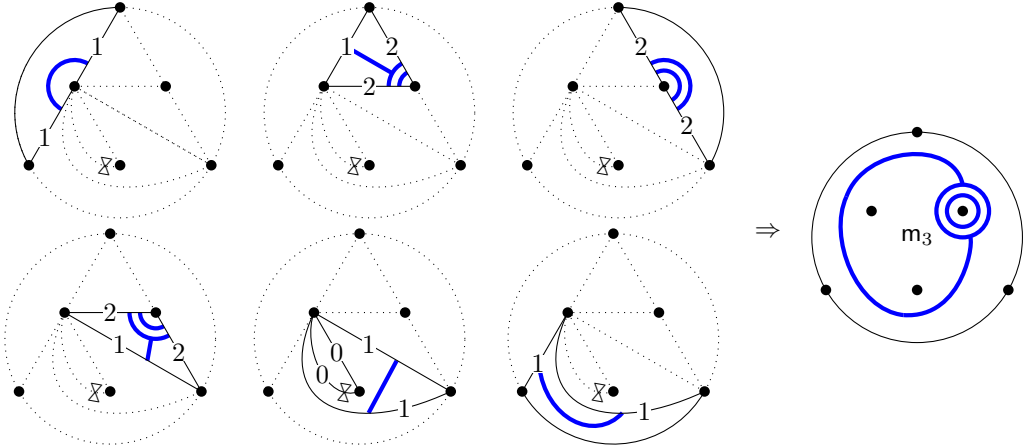
\[
\renewcommand{\arraystretch}{1}{\begin{array}{c}
\begin{tikzpicture}[baseline=0mm,scale=0.7]
 \coordinate (u) at (0,2);
 \coordinate (l) at (-150:2);
 \coordinate (r) at (-30:2);
 \coordinate (cd) at (0,-1);
 \coordinate (cl) at (150:1);
 \coordinate (cr) at (30:1);
 \draw[ultra thick,blue] (cl)+(60:0.5) arc (60:240:0.5);
 \draw (u) arc (90:210:2); \draw[dotted] (u) arc (90:-210:2);
 \draw[dotted] (cl)--(r)--(cr)--(u) (cl)--(cr);
 \draw (u)--node[fill=white,inner sep=0.1]{$1$}(cl);
 \draw[dotted] (r)--(cl);
 \draw (cl)--node[fill=white,inner sep=0.1,pos=0.65]{$1$}(l);
 \draw[dotted] (cl) .. controls (-145:2) and (-100:2) .. (r);
 \draw[dotted] (cl) to (cd);
 \draw[dotted] (cl) to [out=-50,in=-100,relative] node[pos=0.85]{\rotatebox{80}{\footnotesize $\bowtie$}} (cd);
 \fill(u) circle (1mm); \fill(l) circle (1mm); \fill(r) circle (1mm); \fill(cd) circle (1mm); \fill(cl) circle (1mm); \fill(cr) circle (1mm);
\end{tikzpicture}
     \hspace{5mm}
\begin{tikzpicture}[baseline=0mm,scale=0.7]
 \coordinate (u) at (0,2);
 \coordinate (l) at (-150:2);
 \coordinate (r) at (-30:2);
 \coordinate (cd) at (0,-1);
 \coordinate (cl) at (150:1);
 \coordinate (cr) at (30:1);
 \draw[ultra thick,blue] (cr)+(120:0.3) arc (120:180:0.3);
 \draw[ultra thick,blue] (cr)+(120:0.5) arc (120:180:0.5);
 \draw[ultra thick,blue] (cr)+(150:0.5) to +(150:1.5);
 \draw[dotted] (0,0) circle (2);
 \draw[dotted] (r)--(cr);
 \draw (cr)--node[fill=white,inner sep=0.1]{$2$}(u)--node[fill=white,inner sep=0.1]{$1$}(cl)--node[fill=white,inner sep=0.1]{$2$}(cr);
 \draw[dotted] (r)--(cl)--(l);
 \draw[dotted] (cl) .. controls (-145:2) and (-100:2) .. (r);
 \draw[dotted] (cl) to (cd);
 \draw[dotted] (cl) to [out=-50,in=-100,relative] node[pos=0.85]{\rotatebox{80}{\footnotesize $\bowtie$}} (cd);
 \fill(u) circle (1mm); \fill(l) circle (1mm); \fill(r) circle (1mm); \fill(cd) circle (1mm); \fill(cl) circle (1mm); \fill(cr) circle (1mm);
\end{tikzpicture}
     \hspace{5mm}
\begin{tikzpicture}[baseline=0mm,scale=0.7]
 \coordinate (u) at (0,2);
 \coordinate (l) at (-150:2);
 \coordinate (r) at (-30:2);
 \coordinate (cd) at (0,-1);
 \coordinate (cl) at (150:1);
 \coordinate (cr) at (30:1);
 \draw[ultra thick,blue] (cr)+(120:0.3) arc (120:-60:0.3);
 \draw[ultra thick,blue] (cr)+(120:0.5) arc (120:-60:0.5);
 \draw (u) arc (90:-30:2); \draw[dotted] (u) arc (90:330:2);
 \draw (r)--node[fill=white,inner sep=0.1]{$2$}(cr)--node[fill=white,inner sep=0.1]{$2$}(u);
 \draw[dotted] (u)--(cl)--(cr);
 \draw[dotted] (r)--(cl)--(l);
 \draw[dotted] (cl) .. controls (-145:2) and (-100:2) .. (r);
 \draw[dotted] (cl) to (cd);
 \draw[dotted] (cl) to [out=-50,in=-100,relative] node[pos=0.85]{\rotatebox{80}{\footnotesize $\bowtie$}} (cd);
 \fill(u) circle (1mm); \fill(l) circle (1mm); \fill(r) circle (1mm); \fill(cd) circle (1mm); \fill(cl) circle (1mm); \fill(cr) circle (1mm);
\end{tikzpicture}
\\
\begin{tikzpicture}[baseline=0mm,scale=0.7]
 \coordinate (u) at (0,2);
 \coordinate (l) at (-150:2);
 \coordinate (r) at (-30:2);
 \coordinate (cd) at (0,-1);
 \coordinate (cl) at (150:1);
 \coordinate (cr) at (30:1);
 \draw[ultra thick,blue] (cr)+(180:0.3) arc (180:300:0.3);
 \draw[ultra thick,blue] (cr)+(180:0.5) arc (180:300:0.5);
 \draw[ultra thick,blue] (cr)+(260:0.5) to +(260:0.93);
 \draw[dotted] (0,0) circle (2);
 \draw (r)--node[fill=white,inner sep=0.1]{$2$}(cr) (cl)--node[fill=white,inner sep=0.1]{$2$}(cr);
 \draw[dotted] (cr)--(u)--(cl);
 \draw (r)--node[fill=white,inner sep=0.1,pos=0.6]{$1$}(cl);
 \draw[dotted] (cl)--(l);
 \draw[dotted] (cl) .. controls (-145:2) and (-100:2) .. (r);
 \draw[dotted] (cl) to (cd);
 \draw[dotted] (cl) to [out=-50,in=-100,relative] node[pos=0.85]{\rotatebox{80}{\footnotesize $\bowtie$}} (cd);
 \fill(u) circle (1mm); \fill(l) circle (1mm); \fill(r) circle (1mm); \fill(cd) circle (1mm); \fill(cl) circle (1mm); \fill(cr) circle (1mm);
\end{tikzpicture}
     \hspace{5mm}
\begin{tikzpicture}[baseline=0mm,scale=0.7]
 \coordinate (u) at (0,2);
 \coordinate (l) at (-150:2);
 \coordinate (r) at (-30:2);
 \coordinate (cd) at (0,-1);
 \coordinate (cl) at (150:1);
 \coordinate (cr) at (30:1);
 \draw[ultra thick,blue] (cr)+(260:0.93) to +(250:2.03);
 \draw[dotted] (0,0) circle (2);
 \draw[dotted] (r)--(cr)--(u)--(cl)--(cr);
 \draw[dotted] (cl)--(l);
 \draw (r)--node[fill=white,inner sep=0.1,pos=0.6]{$1$}(cl);
 \draw (cl) .. controls (-145:2) and (-100:2) .. node[fill=white,inner sep=0.1,pos=0.85]{$1$} (r);
 \draw (cl) to node[fill=white,inner sep=0.1]{$0$} (cd);
 \draw (cl) to [out=-50,in=-100,relative] node[pos=0.85]{\rotatebox{80}{\footnotesize $\bowtie$}} node[fill=white,inner sep=0.1,pos=0.55]{$0$} (cd);
 \fill(u) circle (1mm); \fill(l) circle (1mm); \fill(r) circle (1mm); \fill(cd) circle (1mm); \fill(cl) circle (1mm); \fill(cr) circle (1mm);
\end{tikzpicture}
     \hspace{5mm}
\begin{tikzpicture}[baseline=0mm,scale=0.7]
 \coordinate (u) at (0,2);
 \coordinate (l) at (-150:2);
 \coordinate (r) at (-30:2);
 \coordinate (cd) at (0,-1);
 \coordinate (cl) at (150:1);
 \coordinate (cr) at (30:1);
 \draw[ultra thick,blue] (cr)+(250:2.03) .. controls (-100:2) and (-130:2) .. (-160:1.5);
 \draw (r) arc (-30:-150:2); \draw[dotted] (r) arc (-30:210:2);
 \draw[dotted] (r)--(cr)--(u)--(cl)--(cr);
 \draw[dotted] (r)--(cl);
 \draw (cl)--node[fill=white,inner sep=0.1,pos=0.65]{$1$}(l);
 \draw (cl) .. controls (-145:2) and (-100:2) .. node[fill=white,inner sep=0.1,pos=0.8]{$1$} (r);
 \draw[dotted] (cl) to (cd);
 \draw[dotted] (cl) to [out=-50,in=-100,relative] node[pos=0.85]{\rotatebox{80}{\footnotesize $\bowtie$}} (cd);
 \fill(u) circle (1mm); \fill(l) circle (1mm); \fill(r) circle (1mm); \fill(cd) circle (1mm); \fill(cl) circle (1mm); \fill(cr) circle (1mm);
\end{tikzpicture}
\end{array}}
\Rightarrow\hspace{3mm}
\begin{tikzpicture}[baseline=0mm,scale=0.7]
 \coordinate (u) at (0,2);
 \coordinate (l) at (-150:2);
 \coordinate (r) at (-30:2);
 \coordinate (cd) at (0,-1);
 \coordinate (cl) at (150:1);
 \coordinate (cr) at (30:1);
 \draw (0,0) circle (2);
 \draw[ultra thick,blue] (cr) circle(3mm); \draw[ultra thick,blue] (cr) circle(5mm);
 \draw[ultra thick,blue] (cr)+(90:0.5) .. controls (70:2) and (120:2) .. (150:1.55);
 \draw[ultra thick,blue] (150:1.55) .. controls (-180:1.7) and (-130:1.8) .. (-100:1.5);
 \draw[ultra thick,blue] (cr)+(-75:0.5) .. controls (0:1.1) and (-60:1.8) .. (-100:1.5);
 \fill(u) circle (1mm); \fill(l) circle (1mm); \fill(r) circle (1mm); \fill(cd) circle (1mm); \fill(cl) circle (1mm); \fill(cr) circle (1mm);
 \node at(0,0) {${\sf m}_3$};
\end{tikzpicture}
\]
 Similarly, we can obtain all ${\sf m}_i$. Then there is a unique tagged arc $\de_i$ such that ${\sf M}_T(\de_i)={\sf m}_i$ for $i \in \{1,\ldots,4,6,\ldots,9\}$. Finally, there is a unique tagged arc $\de_5$ such that ${\sf M}_T(\de_5)={\sf m}_5$ and $\{\de_1,\ldots,\de_9\}$ is a tagged triangulation as follows:
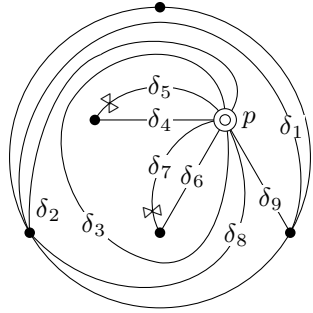
\[
\begin{tikzpicture}[baseline=0mm]
 \coordinate (u) at (0,2);
 \coordinate (l) at (-150:2);
 \coordinate (r) at (-30:2);
 \coordinate (cd) at (0,-1);
 \coordinate (cl) at (150:1);
 \coordinate (cr) at (30:1);
 \draw (0,0) circle (2);
 \draw (l) .. controls (180:2.2) and (130:2.3) .. (0,1.8); \draw (r) .. controls (0:2.2) and (50:2.3) .. node[fill=white,inner sep=0.1,pos=0.4]{$\de_1$} (0,1.8);
 \draw (cr) .. controls (40:2) and (90:1.85) .. (120:1.58);
 \draw (120:1.58) .. controls (140:1.8) and (180:1.8) .. node[right=-1,pos=0.9]{$\de_2$} (l);
 \draw (cr) .. controls (60:1.8) and (110:1.8) .. (150:1.4);
 \draw (150:1.4) .. controls (-180:1.55) and (-130:1.6) .. node[fill=white,inner sep=0.7,pos=0.65]{$\de_3$} (-90:1.4);
 \draw (-90:1.4) .. controls (-60:1.7) and (0:1) .. (cr);
 \draw (cr)--node[fill=white,inner sep=0.7,pos=0.7]{$\de_9$}(r);
 \draw (l) .. controls (-100:2.5) and (-40:2.5) .. node[fill=white,inner sep=0.1,pos=0.7]{$\de_8$} (cr);
 \draw (cl)--node[fill=white,inner sep=0.1]{$\de_4$}(cr);
 \draw (cl) to [out=60,in=120,relative] node[pos=0.15]{\rotatebox{-60}{\footnotesize $\bowtie$}} node[fill=white,inner sep=0.1]{$\de_5$} (cr);
 \draw (cd)--node[fill=white,inner sep=0.1]{$\de_6$}(cr);
 \draw (cd) to [out=60,in=120,relative] node[pos=0.15]{\rotatebox{10}{\footnotesize $\bowtie$}} node[fill=white,inner sep=0.1,pos=0.45]{$\de_7$} (cr);
 \fill[white](cr) circle (1.5mm); \draw (cr) circle (0.15);
 \fill(u) circle (0.7mm); \fill(l) circle (0.7mm); \fill(r) circle (0.7mm); \fill(cd) circle (0.7mm); \fill(cl) circle (0.7mm); \draw (cr) circle (0.7mm);
 \node[right=3] at (cr) {$p$};
\end{tikzpicture}
\]
 where all ends around the puncture $p$ are tagged notched.


\begin{thebibliography}{Y}
 \bibitem[A]{A} C. Amiot, {\it Cluster categories for algebras of global dimension $2$ and quivers with potential}, Ann. Inst. Fourier. Soc. 59 6 (2009) 2525--2590.
 \bibitem[CL]{cl16} P. Cao, F. Li,
{\it Some conjectures on generalized cluster algebras via the cluster formula and D-matrix pattern},
J. Algebra, 493 (2018) 57--78,
 \bibitem[CKLP]{CKLP} G. Cerulli Irelli, B. Keller, D. Labardini-Fragoso, and P.-G. Plamondon,(2013). {\it Linear independence of cluster monomials for skew-symmetric cluster algebras}, Comp. Math., 149(10) (2013), 1753–1764.
 \bibitem[FeST]{FeST} A. Felikson, M. Shapiro and P. Tumarkin, {\it Skew-symmetric cluster algebras of finite mutation type}, J. Eur. Math. Soc. 14 (2012) 1135--1180.
 \bibitem[FeT]{FeT} A. Felikson and P. Tumarkin, {\it Bases for cluster algebras from orbifolds}, Adv. Math. 318 (2017) 191--232.
 \bibitem[FoG06]{FoG06} V. Fock and A. Goncharov, {\it Moduli spaces of local systems and higher Teichm\"{u}ller theory}, Publ. Math. Inst. Hautes \'{E}tudes Sci. No. 103 (2006) 1--211.
 \bibitem[FoG09]{FoG09} V. Fock and A. Goncharov, {\it Cluster ensembles, quantization and the dilogarithm}, Ann. Sci. Ec. Norm. Super. (4) 42, no. 6 (2009) 865--930.
 \bibitem[FoST]{FoST} S. Fomin, M. Shapiro and D. Thurston, {\it Cluster algebras and triangulated surfaces Part I: Cluster complexes}, Acta Math. Vol. 201 (2008) 83--146.
 \bibitem[FoT]{FoT} S. Fomin and D. Thurston, {\it Cluster algebras and triangulated surfaces. Part II: Lambda lengths}, Memoirs AMS, 255(1223) 2018.
 \bibitem[FZ02]{FZ02} S. Fomin and A. Zelevinsky, {\it Cluster algebras I: Foundations}, Amer. Math. Soc. 15 (2002) 497--529.
 \bibitem[FZ04]{FZ04} S. Fomin and A. Zelevinsky, {\it Cluster algebras : Notes for the CDM-03 conference}, Current Developments in Mathematics, 2003, 1--34, International Press, 2004.
 \bibitem[FZ07]{FZ07} S. Fomin and A. Zelevinsky, {\it Cluster algebras IV: Coefficients}, Compos. Math. 143 (2007) 112--164.
 \bibitem[FK]{FK} C. Fu and B. Keller, {\it On cluster algebras with coefficients and $2$-Calabi-Yau categories}, Trans. Amer. Math. Soc. 362, no. 2 (2010) 859--895.
 \bibitem[FuG]{FuG} S. Fujiwara and Y. Gyoda, {\it Duality between front and rear mutations in cluster algebras}, arXiv:1808.02156.
 \bibitem[GSV]{GSV} M. Gekhtman, M. Shapiro and A. Vainshtein, {\it Cluster algebras and Weil-Petersson forms}, Duke Math. J. 127 (2005) 291--311.
 \bibitem[JR]{JR} M. Jungerman and G. Ringel, {\it Minimal triangulations on orientable surfaces}, Acta Math. 145 (1980) 121--154.
 \bibitem[K]{K} B. Keller, {\it Cluster algebras, quiver representations and triangulated categories}, Triangulated categories, London Math. Society Lecture Note Series, Vol. 375, Cambridge University Press (2010) 76--160.
 \bibitem[L]{L} D. Labardini-Fragoso, {\it Quivers with potentials associated to triangulated surfaces}, Proc. Lond. Math. Soc. (3) 98 (2009) 797--839.
 \bibitem[M]{M} M. Mills, {\it Maximal green sequences for quivers of finite mutation type}, Adv. Math. Vol. 319 (2017) 182--210.
\bibitem[MSW11]{MSW11} G. Musiker, R. Schiffler and L. Williams, {\it Positivity for cluster algebras from surfaces}, Adv. Math. Vol. 227 (2011) 2241--2308.
\bibitem[MSW13]{MSW13} G. Musiker, R. Schiffler and L. Williams, {\it Bases for cluster algebras form surfaces}, Compos. Math. 149, 2 (2013) 217--263.
\bibitem[Nak]{Nak} T. Nakanishi, {\it Synchronicity phenomenon in cluster patterns}, arXiv:1906.12036. 
 \bibitem[QZ]{QZ} Y. Qiu and Y. Zhou, {\it Cluster categories for marked surfaces: punctured case}, Compos. Math. 153, no. 9 (2017) 1779--1819.
 \bibitem[Rea]{Rea14} N. Reading, {\it Universal geometric cluster algebras}, Math. Z., 277 (2014), 499–-547. 
 \bibitem[RY]{RY} G.Ringel and J.W.T. Youngs, {\it Solution of the heawood map-coloring problem}, Proc. Natl. Acad. Sci. USA, 60(2) (1968), 438--445.
 \bibitem[Y]{Y} T. Yurikusa, {\it Combinatorial cluster expansion formulas from triangulated surfaces}, arXiv:1808.01567.
\end{thebibliography}
\end{document}